\documentclass[12pt,reqno]{amsart}
\usepackage{amsmath}
\usepackage{mathrsfs,amssymb,graphicx,verbatim,hyperref}
\usepackage{paralist}
\renewcommand{\setminus}{\smallsetminus}
\usepackage{ifthen}

\newcommand{\ifsodaelse}[2]{\ifthenelse{\isundefined{\SODAF}}{#2}{#1}}
\renewcommand{\theta}{\vartheta}
\ifsodaelse    {\usepackage{ltexpprt}\usepackage{amsmath}}{}
\usepackage{mathrsfs,amssymb,graphicx,verbatim}
\usepackage{paralist}

\newcommand\remove[1]{}

\newcommand{\bwr}{\mathbb{\wr}}
\newcommand{\rnote}[1]{}
\newcommand{\jnote}[1]{}
\renewcommand{\d}{\delta}

\newcommand{\1}{\mathbf{1}}
\newcommand{\e}{\varepsilon}
\newcommand{\R}{\mathbb{R}}
\renewcommand{\P}{\mathscr{P}}
\newcommand{\M}{\mathcal M}
\newcommand{\E}{\mathbb{E}}

\newcommand{\N}{\mathbb{N}}

\renewcommand{\P}{\mathscr{P}}

\newcommand{\Lip}{\mathrm{Lip}}

\newcommand{\Prob}{\mathrm{Prob}}

\newcommand{\f}{\varphi}

\DeclareMathOperator{\diam}{diam} 

\renewcommand{\subset}{\subseteq}
\newtheorem{theorem}{Theorem}[section]
\newtheorem{lemma}[theorem]{Lemma}

\newtheorem{corollary}[theorem]{Corollary}

\newtheorem{remark}{Remark}[section]

\newtheorem{conjecture}{Conjecture}
\newtheorem{question}[conjecture]{Question}
\newcommand{\eqdef}{\stackrel{\mathrm{def}}{=}}
\date{}

\renewcommand{\le}{\leqslant}
\renewcommand{\ge}{\geqslant}
\renewcommand{\leq}{\leqslant}

\newcommand\Z{{{\mathbb Z}}}
\renewcommand{\hat}{\widehat}

\include{psfig}

\renewcommand{\epsilon}{\varepsilon}

\theoremstyle{remark}

\newcommand{\F}{\mathcal{F}}

\renewcommand{\phi}{\varphi}

\title{An introduction to the Ribe program}

\thanks{Joram Lindenstrauss, who was my Ph.D. advisor, passed away on April 29, 2012. He was an enormously influential mathematician and the founder of the field of research that
is surveyed here. This article is dedicated to his memory.}

\author{Assaf Naor}
\address{Courant Institute, New York University, New York NY 10012}
\email{naor@cims.nyu.edu}

\date{May 7, 2012}

\begin{document}
\maketitle

\tableofcontents

\section{Introduction}

A 1932 theorem of Mazur and Ulam~\cite{MU32} asserts that if $X$ and $Y$ are Banach spaces and $f:X\to Y$ is an {\em onto} isometry then $f$ must be an affine mapping. The assumption that $f(X)=Y$ is needed here, as exhibited by, say, the mapping $t\mapsto (t,\sin t)$ from $\R$ to $(\R^2,\|\cdot\|_\infty)$. However, a major strengthening  of the Mazur-Ulam theorem  due to Figiel~\cite{Fig68} asserts that if $f:X\to Y$ is an isometry and $f(0)=0$ then there is a unique linear operator $T:\overline{\mathrm{span}}(f(X))\to X$ such that $\|T\|=1$ and $T(f(x))=x$ for every $x\in X$. Thus, when viewed as metric spaces in the isometric category, Banach spaces are highly rigid: their linear structure is completely preserved under isometries, and, in fact, isometries between Banach spaces are themselves rigid.

At the opposite extreme to isometries, the richness of Banach spaces collapses if one removes all quantitative considerations by treating them as topological spaces. Specifically, answering a question posed in 1928 by Fr\'echet~\cite{Fre28} and again in 1932 by Banach~\cite{Ban32}, Kadec~\cite{Kad66,Kad67} proved that any two separable infinite dimensional Banach spaces are homeomorphic.  See~\cite{Kad58,BP60,Bes65,And66} for more information on this topic, as well as its treatment in the monographs~\cite{BP75,FHHMPZ01}. An extension of the Kadec theorem to non-separable spaces was obtained by Toru{\'n}czyk in~\cite{Tor81}.

If one only  considers homeomorphisms between Banach spaces that are ``quantitatively continuous" rather than just continuous, then one recovers a rich and subtle category that exhibits deep rigidity results but does not coincide with the linear theory of Banach spaces. We will explain how this suggests that, despite having no a priori link to Banach spaces, general metric spaces have a hidden structure. Using this point of view, insights from Banach space theory can be harnessed to solve problems in seemingly unrelated disciplines, including group theory, algorithms, data structures, Riemannian geometry, harmonic analysis and probability theory. The purpose of this article is to describe a research program that aims to expose this hidden structure of metric spaces, while highlighting some achievements that were obtained over the past five decades as well as challenging problems that remain open.


In order to make the previous paragraph precise one needs to define
the concept of a quantitatively continuous homeomorphisms. While
there are several meaningful and nonequivalent ways to do this,  we
focus here on {\em uniform homeomorphisms}. Given two metric spaces
$({\mathcal M},d_{\mathcal M})$ and $({\mathcal N},d_{\mathcal N})$,
a bijection $f:{\mathcal M}\to {\mathcal N}$ is called a uniform
homeomorphism if both $f$ and $f^{-1}$ are uniformly continuous, or
equivalently if there exist nondecreasing functions
$\alpha,\beta:[0,\infty)\to (0,\infty]$ with $\lim_{t\to
0}\beta(t)=0$ such that $\alpha(d_{\mathcal M}(a,b))\le d_{\mathcal
N}(f(a),f(b))\le \beta(d_{\mathcal M}(a,b))$ for all distinct
$a,b\in {\mathcal M}$.

In the  seminal 1964 paper~\cite{Lin64} Lindenstrauss proved that,
in contrast to the Kadec theorem, there exist many pairs of
separable infinite dimensional Banach spaces, including $L_p(\mu)$
and $L_q(\nu)$ if $p\neq q$ and $\max\{p,q\}\ge 2$, that are not
uniformly homeomorphic. Henkin proved in~\cite{Hen67} that if $n\ge
2$ then $C^k([0,1]^n)$ is not uniformly homeomorphic to $C^1([0,1])$
for all $k\in \N$ (this result was previously announced by
Grothendieck~\cite{Gro56} with some indication of a proof).
Important work of Enflo~\cite{Enf69-smirnov,Enf69,Enf70}, which was
partly motivated by his profound investigation of Hilbert's fifth
problem in infinite dimensions, obtained additional results along
these lines. In particular, in~\cite{Enf69} Enflo completed
Lindenstrauss' work~\cite{Lin64} by proving that that $L_p(\mu)$ and
$L_q(\nu)$ are not uniformly homeomorphic if $p\neq q$ and $p,q\in
[1,2]$, and  in~\cite{Enf70} he proved that a Banach space
$(X,\|\cdot\|_X)$ which is uniformly homeomorphic to a Hilbert space
$(H,\|\cdot\|_H)$ must be isomorphic to $H$, i.e., there exists a
bounded {\em linear} operator $T:X\to H$ such that $\|Tx\|_H\ge
\|x\|_X$ for all $x\in X$. A later deep theorem of Johnson,
Lindenstrauss and Schechthman~\cite{JLS96} makes the same assertion
with Hilbert space replaced by $\ell_p$, $p\in (0,\infty)$, i.e.,
any Banach space that is uniformly homeomorphic to $\ell_p$ must be
isomorphic to $\ell_p$. At the same time, as shown by Aharoni and
Lindenstrauss~\cite{AL78} and Ribe~\cite{Rib84}, there exist pairs
of uniformly homeomorphic Banach spaces that are not isomorphic.

In 1976 Martin Ribe proved~\cite{Rib76} that if two Banach spaces
are uniformly homeomorphic then they have the same finite
dimensional subspaces. To make this statement precise, recall
James'~\cite{Jam72} notion of (crude) finite representability: a
Banach space $(X,\|\cdot\|_X)$ is said to be {\em finitely
representable} in a Banach space $(Y,\|\cdot\|_Y)$ if there exists
$K\in [1,\infty)$ such that for every finite dimensional linear
subspace $F\subseteq X$ there exists a linear operator $T:F\to Y$
satisfying $\|x\|_X\le \|Tx\|_Y\le K\|x\|_X$ for all $x\in F$. For
example, for all $p\in [1,\infty]$ any $L_p(\mu)$ space is finitely
representable in $\ell_p$, and the classical Dvoretzky
theorem~\cite{Dvo60} asserts that Hilbert space is finitely
representable in any infinite dimensional Banach space. If $p,q\in
[1,\infty]$ and $p\neq q$  then at least one of the spaces
$L_p(\mu)$, $L_q(\nu)$ is not finitely representable in the other;
see, e.g.~\cite{Woj91}.

\begin{theorem}[Ribe's rigidity theorem~\cite{Rib78}]\label{thm:ribe}
If $X$ and $Y$ are uniformly homeomorphic Banach spaces then $X$ is finitely representable in $Y$ and $Y$ is finitely representable in $X$.
\end{theorem}
Influential alternative proofs of Ribe's theorem were obtained by
Heinrich and Mankiewicz~\cite{HM82} and Bourgain~\cite{Bou87}. See
also the treatment in the surveys~\cite{Enf76,Ben85} and Chapter 10
of the book~\cite{BL00}.   In~\cite{Rib78} Ribe obtained a stronger
version of Theorem~\ref{thm:ribe} under additional geometric
assumptions on the spaces $X$ and $Y$. The converse to Ribe's
theorem fails, since for $p\in [1,\infty)\setminus\{2\}$ the spaces
$L_p(\R)$ and $\ell_p$ are finitely representable in each other but
not uniformly homeomorphic; for $p=1$ this was proved by
Enflo~\cite{Ben85}, for $p\in (1,2)$ this was proved by
Bourgain~\cite{Bou87}, and for $p\in (2,\infty)$ this was proved by
Gorelik~\cite{Gor94}.

Theorem~\ref{thm:ribe} (informally) says that isomorphic finite
dimensional linear properties of Banach spaces are preserved under
uniform homeomorphisms, and are thus in essence ``metric
properties". For concreteness, suppose that $X$ satisfies the
following property: for every $n\in \N$ and every $x_1,\ldots,x_n\in
X$ the average of $\|\pm x_1\pm x_2\pm\ldots\pm x_n\|_X^2$ over all
the $2^n$ possible choices of signs is at most
$K(\|x_1\|_X^2+\ldots+\|x_n\|_X^2)$, where $K\in (0,\infty)$ may
depend on the geometry of $X$ but not on $n$ and $x_1,\ldots,x_n$.
Ribe's theorem asserts that if $Y$ is uniformly homeomorphic to $X$
then it also has the same property. Rather than giving a formal
definition, the reader should keep properties of this type in mind:
they are ``finite dimensional linear properties" since they are
given by inequalities between lengths of linear combinations of
finitely many vectors, and they are ``isomorphic" in the sense that
they are insensitive to a loss of a constant factor. Ribe's theorem is
thus a remarkable rigidity statement, asserting that uniform
homeomorphisms between Banach spaces cannot alter their finite
dimensional structure.

Ribe's theorem indicates that in principle any isomorphic finite
dimensional linear property of Banach spaces can be equivalently
formulated using only distances between points and making no
reference whatsoever to the linear structure. Recent work of
Ostrovskii~\cite{Ost11,Ost12-forms,Ost12-test} can be viewed as
making this statement formal in a certain abstract sense. The {\em
Ribe program}, as formulated by Bourgain in 1985 (see~\cite{Bou85}
and mainly~\cite{Bou86}), aims to explicitly study this phenomenon.
If parts of the finite dimensional linear theory of Banach spaces
are in fact a ``nonlinear theory in disguise" then if one could
understand how to formulate them using only the metric structure
this would make it possible to study them in the context of general
metric spaces. As a first step in the Ribe program one would want to
discover metric reformulations of key concepts of Banach space
theory. Bourgain's famous metric characterization of when  a Banach
space admits an equivalent uniformly convex norm~\cite{Bou86} was the
first successful completion of a step in this plan. By doing so,
Bourgain kick-started the Ribe program, and this was quickly
followed by efforts of several researchers leading to satisfactory
progress on key steps of the Ribe program.

The Ribe program does not limit itself to reformulating aspects of
Banach space theory using only metric terms. Indeed, this should be
viewed as only a first (usually highly nontrivial) step. Once this
is achieved, one has an explicit ``dictionary" that translates
concepts that a priori made sense only in the presence of linear
structure to the language of general metric spaces. The next
important step in the Ribe program is to investigate the extent to
which Banach space phenomena, after translation using the new
``dictionary", can be proved for general metric spaces. Remarkably,
over the past decades it turned out that this approach is very
successful, and it uncovers structural properties of metric spaces
that have major impact on areas which do not have any a priori link
to Banach space theory. Examples of such successes of the Ribe
program will be described throughout this article.

A further step in the Ribe program is to investigate the role of the
metric reformulations of Banach space concepts, as provided by the
first step of the Ribe program, in metric space geometry. This step
is not limited to metric analogues of Banach space phenomena, but
rather it aims to use the new ``dictionary" to solve problems that
are inherently nonlinear (examples include the use of nonlinear type
in group theory; see Section~\ref{sec:random walks}). Moreover,
given the realization that insights from Banach space theory often
have metric analogues, the Ribe program aims to uncover metric
phenomena that mirror Banach space phenomena but are not strictly
speaking based on metric reformulations of isomorphic finite
dimensional linear properties. For example, Bourgain's embedding
theorem was discovered due to the investigation of a question raised
by Johnson and Lindenstrauss~\cite{JL84} on a metric analogue of
John's theorem~\cite{Joh48}. Another example is the investigation,
as initiated by Bourgain, Figiel and Milman~\cite{BFM86}, of
nonlinear versions of Dvoretzky's theorem~\cite{Dvo60} (in this
context Milman also asked for a nonlinear version of his Quotient of
Subspace Theorem~\cite{Mil85}, a question that is studied
in~\cite{MN04}). Both of the examples above led to the discovery of
theorems on metric spaces that are truly nonlinear and do not have
immediate counterparts in Banach space theory (e.g., the appearance
of ultrametrics in the context of nonlinear Dvoretzky theory; see
Section~\ref{sec:dvo}), and they had major impact on areas such as
approximation algorithms and  data structures. Yet another example
is Ball's nonlinear version~\cite{Bal92} of Maurey's extension
theorem~\cite{Mau74}, based on nonlinear type and cotype (see
Section~\ref{sec:Mtype}). Such developments include some of the most
challenging and influential aspects of the Ribe program. In essence,
Ribe's theorem pointed the way to a certain analogy between linear
and nonlinear metric spaces. One of the main features of the Ribe
program is that this analogy is a source of new meaningful questions
in metric geometry that  probably would not have been raised if it
weren't for the Ribe program.

\begin{remark}\label{rem:rigidity}
{\em A rigidity theorem asserts that a deformation of a certain
object preserves more structure than one might initially expect. In
other words, equivalence in a weak category implies the existence of
an equivalence in a stronger category. Rigidity theorems are
naturally important since they say much about the structure of the
stronger category (i.e., that it is rigid). However, the point of
view of the Ribe program is that a rigidity theorem opens the door
to a new research direction whose goal is to uncover hidden
structures in the weaker category: perhaps the rigidity exhibited by
the stronger category is actually an indication that concepts and
theorems of the stronger category are ``shadows" of a theory that
has a significantly wider range of applicability? This philosophy
has been very successful in the context of the Ribe program, but
similar investigations were also initiated in response to rigidity
theorems in other disciplines. For example, it follows from the
Mostow rigidity theorem~\cite{Mos68} that if two closed hyperbolic
$n$-manifolds ($n>2$) are homotopically equivalent then they are
isometric. This suggests that the volume of a hyperbolic manifold
may be generalized to a homotopy invariant quantity defined for
arbitrary manifolds: an idea that was investigated by Milnor and
Thurston~\cite{MT77} and further developed by Gromov~\cite{Gro82}
(see also Sections 5.34--5.36 and 5.43 in~\cite{Gro07}). These
investigations led to the notion of {\em simplicial volume}, a
purely topological notion associated to a closed oriented manifold
that remarkably coincides with the usual volume in the case of
hyperbolic manifolds. This notion is very helpful for studying
general continuous maps between hyperbolic manifolds.}
\end{remark}


\noindent{\bf Historical note.} Despite the fact that it was first formulated by Bourgain, the
Ribe program is called this way because it is inspired by Ribe's
rigidity theorem. I do not know the exact origin of this name.
In~\cite{Bou86} Bourgain explains the program and its motivation
from Ribe's theorem, describes the basic ``dictionary" that relates
Banach space concepts to metric space concepts, presents examples of
natural steps of the program, raises some open questions, and proves
his metric characterization of isomorphic uniform convexity as the
first successful completion of a step in the program. Bourgain also
writes in~\cite{Bou86} that ``A detailed exposition of this program
will appear in J. Lindenstrauss's forthcoming survey paper [5]."
Reference [5] in~\cite{Bou86} is cited as ``J. Lindenstrauss, {\em
Topics in the geometry of metric spaces}, to appear."  Probably
referring
 to the same unpublished survey,
in~\cite{Bou85} Bourgain also discusses the Ribe program and writes
``We refer the reader to the survey of J. Lindenstrauss [4] for a
detailed exposition of this theme", where reference [4]
of~\cite{Bou85} is ``J. Lindenstrauss, {\em Proceedings Missouri
Conf., Missouri -- Columbia (1984)}, to appear." Unfortunately,
Lindenstrauss' paper was never published.

\bigskip

This article is intended to serve as an introduction to the Ribe
program, targeted at nonspecialists. Aspects of this research
direction have been previously surveyed
in~\cite{Lin70,Ben85,Pis86,Lin98,BL00,MN08,Kal08,MN08-markov}  and
especially in Ball's Bourbaki expos\'e~\cite{Bal12}. While the
material surveyed here has some overlap with these paper, we cover a
substantial amount of additional topics. We also present sketches of
arguments as an indication of the type of challenges that the Ribe
program raises, and we describe examples of applications to areas
which are far from Banach space theory in order to indicate the
versatility of this approach to metric geometry.

\bigskip \noindent {\bf Asymptotic notation.} Throughout this article we will use the notation $\lesssim,\gtrsim$ to denote the corresponding inequalities up to universal constant factors.
We will also denote equivalence up to universal constant factors by
$\asymp$, i.e., $A\asymp B$ is the same as $(A\lesssim B)\wedge
(A\gtrsim B)$.

\bigskip

\noindent{\bf Acknowledgements.} This article accompanies the 10th
Takagi Lectures delivered by the author at RIMS, Kyoto, on May 26
2012. I am grateful to Larry Guth and Manor Mendel for helpful
suggestions. The research presented here is supported in part by NSF
grant CCF-0832795, BSF grant 2010021, and the Packard Foundation.


\section{Metric type}

Fix a Banach space $(X,\|\cdot\|_X)$. By the triangle inequality we have $\|\e_1 x_1+\ldots+\e_n x_n\|_X\le \|x_1\|_X+\ldots+\|x_n\|_X$ for every $x_1,\ldots,x_n\in X$ and every $\e_1,\ldots,\e_n\in \{-1,1\}$. By averaging this inequality over all possible choices of signs $\e_1,\ldots,\e_n\in \{-1,1\}$ we obtain the following {\em randomized triangle inequality}.
\begin{equation}\label{eq:randomized triangle}
\frac{1}{2^n}\sum_{\e_1,\ldots,\e_n\in \{-1,1\}} \left\|\sum_{i=1}^n\e_ix_i\right\|_X\le \sum_{i=1}^n \|x_i\|_X.
\end{equation}
For $p\ge 1$, the Banach space $X$ is said to have {Rademacher type $p$} if there exists a constant $T\in (0,\infty)$ such that for every $n\in \N$ and every $x_1,\ldots,x_n\in X$ we have
\begin{equation}\label{def rad type}
\frac{1}{2^n}\sum_{\e_1,\ldots,\e_n\in \{-1,1\}} \left\|\sum_{i=1}^n\e_ix_i\right\|_X\le T\left(\sum_{i=1}^n \|x_i\|_X^p\right)^{1/p}.
\end{equation}
It is immediate to check (from the case of collinear $x_1,\ldots,x_n$) that if~\eqref{def rad type} holds then necessarily $p\le 2$. If $p>1$ and~\eqref{def rad type} holds then $X$ is said to have nontrivial type. Note that if this happens then in most cases, e.g. if $x_1,\ldots,x_n$ are all unit vectors, \eqref{def rad type} constitutes an asymptotic improvement of the triangle inequality~\eqref{eq:randomized triangle}. For concreteness, we recall that $L_p(\mu)$ has Rademacher type $\min \{p,2\}$.
\begin{remark}\label{rem:kahane}
{\em A classical inequality of Kahane~\cite{Kah64} asserts that for every $q\ge 1$ we have
\begin{equation*}
\left(\frac{1}{2^n}\sum_{\e_1,\ldots,\e_n\in \{-1,1\}} \left\|\sum_{i=1}^n\e_ix_i\right\|_X^q\right)^{1/q}\le \frac{c(p)}{2^n}\sum_{\e_1,\ldots,\e_n\in \{-1,1\}} \left\|\sum_{i=1}^n\e_ix_i\right\|_X,
\end{equation*}
where $c(p)\in (0,\infty)$ depends on $p$ but not on $n$, the choice of vectors $x_1,\ldots,x_n\in X$, and the Banach space $X$ itself. Therefore the property~\eqref{def rad type} is equivalent to the requirement
\begin{equation}\label{eq:q-rad type}
\left(\frac{1}{2^n}\sum_{\e_1,\ldots,\e_n\in \{-1,1\}} \left\|\sum_{i=1}^n\e_ix_i\right\|_X^q\right)^{1/q}\le T\left(\sum_{i=1}^n \|x_i\|_X^p\right)^{1/p},
\end{equation}
with perhaps a different constant $T\in (0,\infty)$. }
\end{remark}

The improved triangle inequality~\eqref{def rad type} is of profound importance to the study of geometric and analytic questions in Banach space theory and harmonic analysis; see~\cite{Mau03} and the references therein for more information on this topic.

The Ribe theorem implies that the property of having type $p$ is preserved under uniform homeomorphism of Banach spaces. According to the philosophy of the Ribe program, the next goal is to reformulate this property while using only distances between points and making no reference whatsoever to the linear structure of $X$. We shall now explain the ideas behind the known results on this step of the Ribe program as an illustrative example of the geometric and analytic challenges that arise when one endeavors to address such questions.

\subsection{Type for metric spaces}\label{sec:metric type} The basic idea, due to Enflo~\cite{Enf69}, and later to Gromov~\cite{Gro83} and Bourgain, Milman and Wolfson~\cite{BMW86}, is as follows. Given $x_1,\ldots,x_n\in X$ define $f:\{-1,1\}^n\to X$ by
\begin{equation}\label{eq:def linear f}
\forall\, \e=(\e_1,\ldots,\e_n)\in \{-1,1\}^n,\quad f(\e)=\sum_{i=1}^n\e_ix_i.
\end{equation}
With this notation, the definition of Rademacher type appearing in~\eqref{def rad type} is the same as the inequality
\begin{multline}\label{type for linear}
\E_\e\left[\|f(\e)-f(-\e)\|_X\right]\\\le T\left(\sum_{i=1}^n \E_\e\left[\left\|f(\e)-f(\e_1,\ldots,\e_{i-1},-\e_i,\e_{i+1},\ldots,\e_n)\right\|_X^p\right]\right)^{1/p},
\end{multline}
where $\E[\cdot]$ denotes expectation with respect to a uniformly random choice of $\e\in \{-1,1\}^n$.

Inequality~\eqref{type for linear} seems to involve only distances between points, except for the crucial fact that the function $f$ itself is the {\em linear} function appearing in~\eqref{eq:def linear f}. Enflo's (bold) idea~\cite{Enf69} (building on his earlier work~\cite{Enf69,Enf70}) is to drop the linearity requirement of $f$ and to demand that~\eqref{type for linear} holds for {\em all} functions $f:\{-1,1\}^n\to X$. Thus, for $p\ge 1$ we say that a {\em metric space} $(\M,d_\M)$ has type $p$ if there exists a constant $T\in (0,\infty)$ such that for every $n\in \N$ and every $f:\{-1,1\}^n\to \M$,
\begin{multline}\label{def:our metric type}
\E_\e\left[d_\M\left(f(\e),f(-\e)\right)\right]\\\le T\left(\sum_{i=1}^n \E_\e\left[d_\M\left(f(\e),f(\e_1,\ldots,\e_{i-1},-\e_i,\e_{i+1},\ldots,\e_n)\right)^p\right]\right)^{1/p}.
\end{multline}
\begin{remark}\label{rem:enf+BMW}
{\em The above definition of type of a metric space is ad hoc: it was chosen here for the sake of simplicity of exposition. While this definition  is sufficient for the description of the key ideas and it is also strong enough for the ensuing geometric applications, it differs from the standard definitions of type for metric spaces that appear in the literature. Specifically, motivated by the fact that Rademacher type $p$ for a Banach space $(X,\|\cdot\|_X)$ is equivalent to ~\eqref{eq:q-rad type} for any $q\ge 1$, combining the above reasoning with the case $q=p$ in~\eqref{eq:q-rad type} leads to Enflo's original definition: say that a metric space $(\M,d_\M)$ has {\em Enflo type} $p$ if if there exists a constant $T\in (0,\infty)$ such that for every $n\in \N$ and every $f:\{-1,1\}^n\to \M$,
\begin{multline}\label{def:enflo type}
\E_\e\left[d_\M\left(f(\e),f(-\e)\right)^p\right]\\\le T^p\sum_{i=1}^n \E_\e\left[d_\M\left(f(\e),f(\e_1,\ldots,\e_{i-1},-\e_i,\e_{i+1},\ldots,\e_n)\right)^p\right].
\end{multline}
Analogously, by~\eqref{eq:q-rad type} with $q=2$ and H\"older's inequality, if  $(X,\|\cdot\|_X)$ has Rademacher type $p\in [1,2]$ then there exists a constant $T\in (0,\infty)$ such that for every $n\in \N$ and every $x_1,\ldots,x_n\in X$,
$$
\E_\e\left[\left\|\sum_{i=1}^n\e_ix_i\right\|_X^2\right]\le T^2 n^{\frac{2}{p}-1}\sum_{i=1}^n \|x_i\|_i^2.
$$
Hence, following the above reasoning, Bourgain, Milman and Wolfson~\cite{BMW86} suggested the following definition of type of metric spaces, which is more convenient than Enflo type for certain purposes: say that a metric space $(\M,d_\M)$ has {\em BMW type} $p$ if if there exists a constant $T\in (0,\infty)$ such that for every $n\in \N$ and every $f:\{-1,1\}^n\to \M$,
\begin{multline}\label{def:BMW type}
\E_\e\left[d_\M\left(f(\e),f(-\e)\right)^2\right]\\\le T^2n^{\frac{2}{p}-1}\sum_{i=1}^n \E_\e\left[d_\M\left(f(\e),f(\e_1,\ldots,\e_{i-1},-\e_i,\e_{i+1},\ldots,\e_n)\right)^2\right].
\end{multline}
In~\cite{Gro83} Gromov suggested the above definitions of type of metric spaces, but only when $p=2$, in which case~\eqref{def:enflo type} and~\eqref{def:BMW type} coincide.
}
\end{remark}

\begin{remark}\label{rem:improved triangle}
{\em For the same reason that Rademacher type $p>1$ should be viewed as an improved (randomized) triangle inequality, i.e., an improvement over~\eqref{eq:randomized triangle}, the above definitions of type of metric spaces should also be viewed as an improved triangle inequality. Indeed, it is straightforward to check that every metric space $(\M,d_\M)$ satisfies
\begin{multline}\label{def:trivial type 1}
\E_\e\left[d_\M\left(f(\e),f(-\e)\right)\right]\\\le \sum_{i=1}^n \E_\e\left[d_\M\left(f(\e),f(\e_1,\ldots,\e_{i-1},-\e_i,\e_{i+1},\ldots,\e_n)\right)\right]
\end{multline}
for every $n\in \N$ and every $f:\{-1,1\}^n\to \M$. Thus every metric space has type $1$ (equivalently Enflo type $1$) with $T=1$. A similar application of the triangle inequality shows that every metric space has BMW type $1$ with $T=1$. Our definition~\eqref{def:our metric type} of type of a metric space $(\M,d_\M)$ is not formally stronger than~\eqref{def:trivial type 1}, and with this in mind one might prefer to consider the following variant of~\eqref{def:our metric type}:
\begin{multline}\label{def:better our metric type}
\E_\e\left[d_\M\left(f(\e),f(-\e)\right)\right]\\\lesssim \E_\e\left[\left(\sum_{i=1}^n d_\M\left(f(\e),f(\e_1,\ldots,\e_{i-1},-\e_i,\e_{i+1},\ldots,\e_n)\right)^p\right)^{1/p}\right].
\end{multline}
Note that, by Jensen's inequality, \eqref{def:better our metric type} implies~\eqref{def:our metric type}. We chose to work with the definition appearing in~\eqref{def:our metric type} only for simplicity of notation and exposition; the argument below will actually yield~\eqref{def:better our metric type}.}
\end{remark}

\subsection{The geometric puzzle} One would be justified to be concerned about the ``leap of faith" that was performed in Section~\ref{sec:metric type}. Indeed, if a Banach space $(X,\|\cdot\|_X)$ satisfies~\eqref{type for linear} for all linear functions as in~\eqref{eq:def linear f} there is no reason to expect that it actually satisfies~\eqref{type for linear} for all $f:\{-1,1\}^n\to X$ whatsoever. Thus, for the discussion in Section~\ref{sec:metric type} to be most meaningful one needs to prove that if a Banach space has Rademacher type $p$ then it also has type $p$ as a metric space (resp. Enflo type $p$ or BMW type $p$). This question, posed in 1976  by Enflo~\cite{Enf76} (for the case of Enflo type), remains open.
\begin{question}[Enflo's problem]\label{Q:enf}
Is it true that if a Banach space has Rademacher type $p$ then it also has Enflo type $p$?
\end{question}

We will present below  an argument that leads to the following slightly weaker fact: if a Banach space has Rademacher type $p$ then for every $\e\in (0,1)$ it also has type $p-\e$ as a metric space. We will follow an elegant argument of Pisier~\cite{Pis86}, who almost solved Enflo's problem by showing that if a Banach space has Rademacher type $p$ then it also has Enflo type $p-\e$ for every $\e\in (0,1)$. Earlier, and via a different argument, Bourgain, Milman and Wolfson proved~\cite{BMW86} that if a Banach space has Rademacher type $p$ then it also has BMW type $p-\e$ for every $\e\in (0,1)$. More recently, \cite{MN07-scaled} gave a different, more complicated (and less useful), definition of type of a metric space, called {\em scaled Enflo type}, and showed that a Banach space has Rademacher type $p$ if and only if it has scaled Enflo type $p$. This completes the Ribe program for Rademacher type, but it leaves much to be understood, as we conjecture that the answer to Question~\ref{Q:enf} is positive. In~\cite{NS02,KN06,NPSS06,HN12} it is proved that the answer to Question~\ref{Q:enf} is positive for certain classes of Banach spaces (including all $L_p(\mu)$ spaces).

To better understand the geometric meaning of the above problems and results consider the following alternative description of the definition of type of a metric space $(\M,d_\M)$. Call a subset of $2^n$ points in $\M$ that is indexed by $\{-1,1\}^n$  a {\em geometric cube} in $\M$. A diagonal of the geometric cube $\{x_\e\}_{\e\in \{-1,1\}^n}\subseteq \M$ is a pair $\{x_\e,x_\delta\}$ where $\e,\delta\in \{-1,1\}^n$ differ in all the coordinates (equiv. $\delta=-\e$).  An edge of this geometric cube is a pair $\{x_\e,x_\delta\}$ where $\e,\delta\in \{-1,1\}^n$ differ in exactly one coordinate. Then~\eqref{def:our metric type} is the following statement
\begin{equation}\label{eq:diagona-edge}
\frac{\sum diagonal}{2^n}\le T \left(\frac{\sum edge^p}{2^n}\right)^{1/p},
\end{equation}
where in the left hand side of~\eqref{eq:diagona-edge} we have the sum of the lengths of all the diagonals of the geometric cube, and in the right hand side of~\eqref{eq:diagona-edge} we have the sum of the $p$th power of the lengths of all the edges of the geometric cube. The assertion that $(\M,d_\M)$ has type $p$ means that~\eqref{eq:diagona-edge} holds for {\em all} geometric cubes in $\M$.

If $(X,\|\cdot\|_X)$ is a Banach space with Rademacher type $p$ then we know that~\eqref{eq:diagona-edge} holds true for all {\em parallelepipeds} in $X$, as depicted in Figure~\ref{fig:parallel}.

\begin{figure}[ht]
\centering \fbox{
\begin{minipage}{4.7in}
\centering
\includegraphics[height=50mm]{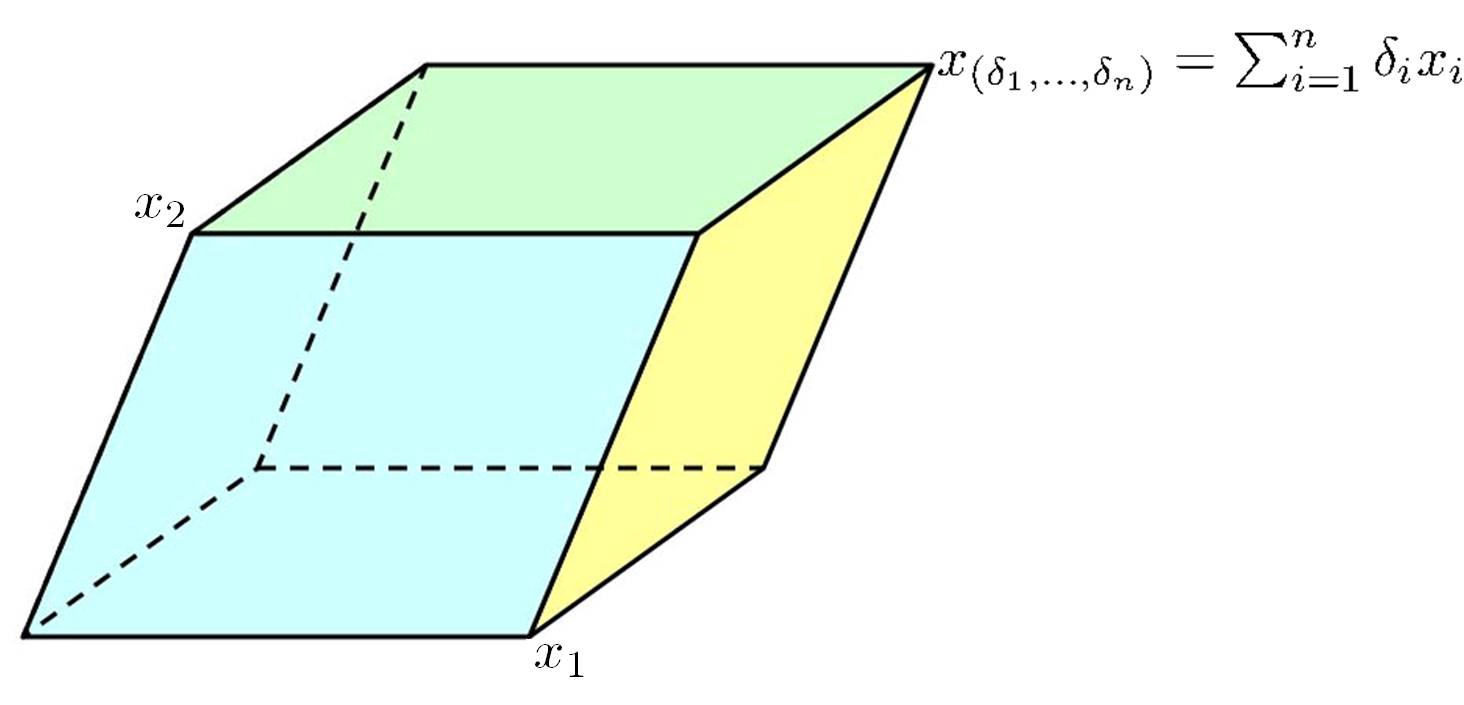}
\caption{$X$ having Rademacher type $p$ is equivalent to the requirement that~\eqref{eq:diagona-edge} holds true for every parallelepiped in $X$, i.e., a set of vectors $\{x_\delta\}_{\delta\in \{0,1\}^n}$ where for some $x_1,\ldots,x_n\in X$ we have $x_{\delta}=\sum_{i=1}^n \delta_ix_i$ for all $\delta=(\delta_1,\ldots,\delta_n)\in \{0,1\}^n$.  }
\label{fig:parallel}
\end{minipage}
}
\end{figure}
The geometric ``puzzle" is therefore to deduce the validity of~\eqref{eq:diagona-edge} for all geometric cubes in $X$ (perhaps with $p$ replaced by $p-\e$) from the assumption that it holds for all parallelepipeds. In other words, given $x_1,\ldots,x_{2^n}\in X$, index these points arbitrarily by $\{-1,1\}^n$. Once this is done, some pairs of these points have been declared as diagonals, and other pairs have been declared as edges, in which case~\eqref{eq:diagona-edge} has to hold true for these pairs; see Figure~\ref{fig:points}.
\begin{figure}[ht]
\centering \fbox{
\begin{minipage}{4.7in}
\centering
\includegraphics[height=60mm]{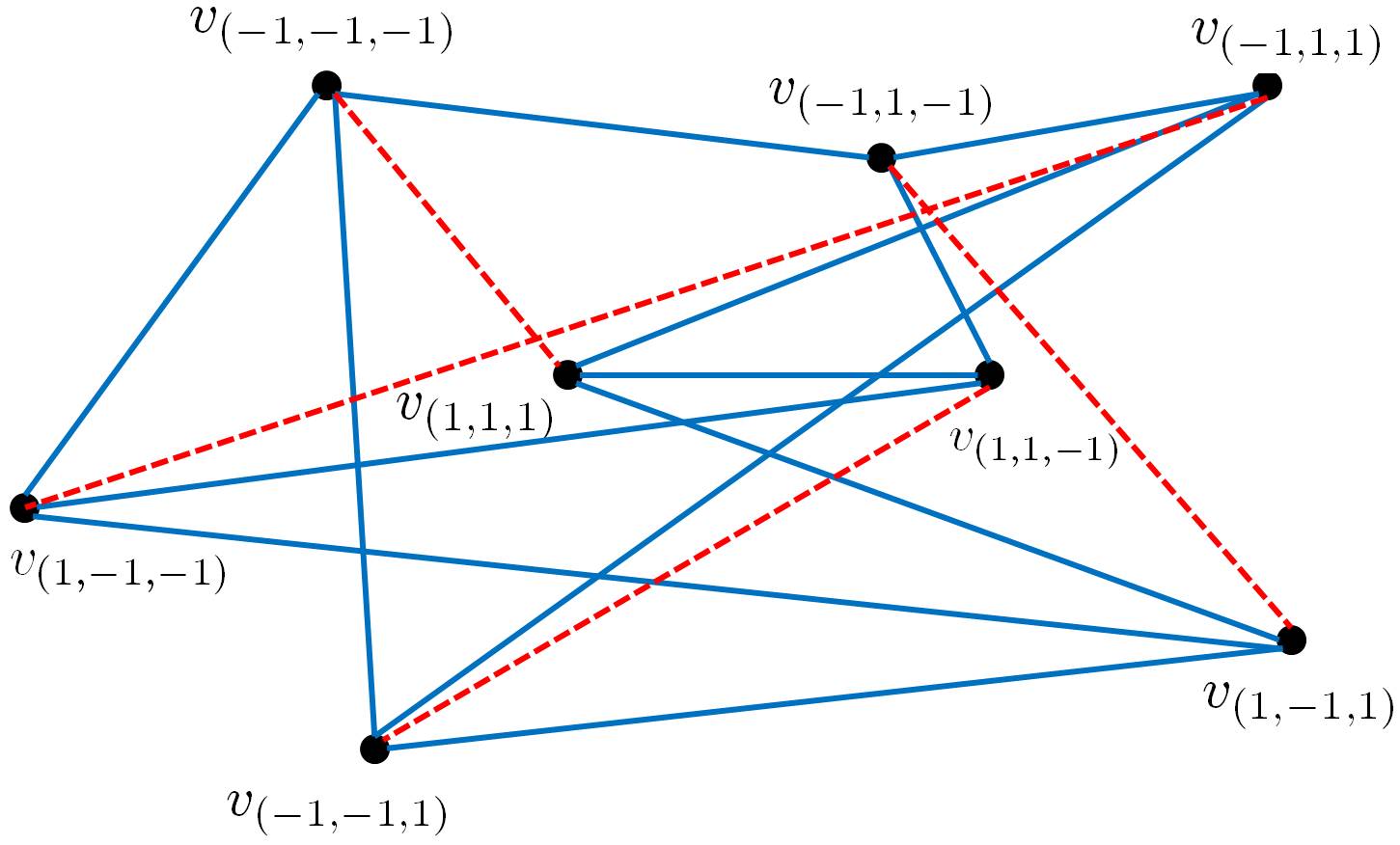}
\caption{A schematic illustration of the problem when $n=3$. Given $x_1,\ldots,x_8\in X$, we index them using the labels $\{(\e_1,\e_2,\e_3):\ \e_1,\e_2,\e_3\in \{-1,1\}\}$ as depicted above. Once this is done, the dotted lines represent diagonals and the full lines represent edges.}
\label{fig:points}
\end{minipage}
}
\end{figure}

\subsection{Pisier's argument} Our goal here is to describe an approach, devised by Pisier in 1986, to deduce metric type from Rademacher type. Before doing so we recall some basic facts related to vector-valued Fourier analysis on $\{-1,1\}^n$. The characters of the group $\{-1,1\}^n$ (equipped with coordinate-wise multiplication) are the Walsh functions $\{W_A\}_{A\subseteq \{1,\ldots,n\}}$, where $
W_A(\e)=\prod_{i\in A} \e_i$.
Fix a Banach space $(X,\|\cdot\|_X)$. Any function $f:\{-1,1\}^n\to X$ has the Fourier expansion
$$
f(\e)=\sum_{A\subseteq \{1,\ldots,n\}} \hat{f}(A)W_A(\e),
$$
where
$$
\hat{f}(A)=\E_\e\left[f(\e)W_A(\e)\right]=\frac{1}{2^n}\sum_{\e=(\e_1,\ldots,\e_n)\in \{-1,1\}^n} f(\e)\prod_{i\in A}\e_i\in X.
$$
For $j\in \{1,\ldots,n\}$ define $\partial_jf:\{-1,1\}^n\to X$ by
\begin{multline}\label{eq:def partial}
\partial_jf(\e)=\frac{f(\e)-f(\e_1,\ldots,\e_{j-1},-\e_j,\e_{j+1},\ldots,\e_n)}{2}\\=\sum_{\substack{A\subseteq \{1,\ldots,n\}\\j\in A}}\hat{f}(A)W_A(\e).
\end{multline}
The hypercube Laplacian of $f$ is given by
$$
\Delta f(\e)=\sum_{j=1}^n\partial_j f(\e)=\sum_{A\subseteq \{1,\ldots,n\}} |A|\hat{f}(A)W_A(\e).
$$
The associated time-$t$   evolute of $f$ under the heat semigroup is
\begin{equation}\label{eq:def heat}
e^{-t\Delta}f(\e)=\sum_{A\subseteq \{1,\ldots,n\}} e^{-t|A|}\hat{f}(A)W_A(\e).
\end{equation}
Since the operator $e^{-t\Delta}$ coincides with convolution with the Riesz kernel $R_t(\e)=\prod_{i=1}^n(1+e^{-t}\e_i)$, which for $t\ge 0$ is the density of a probability measure on $\{-1,1\}^n$, we have by convexity
\begin{equation}\label{eq:heat contracts}
t\ge 0\implies \E_\e\left[\left\|e^{-t\Delta}f(\e)\right\|_X\right]\le \E_\e\left[\|f(\e)\|_X\right].
\end{equation}
It immediately follows from~\eqref{eq:def heat} that
\begin{equation}\label{eq:heat identity}
e^{-t\Delta}\left(W_{\{1,\ldots,n\}}e^{-t\Delta}f\right)=e^{-tn}W_{\{1,\ldots,n\}}f.
\end{equation}
Consequently, we deduce from~\eqref{eq:heat identity} and~\eqref{eq:heat contracts} that
\begin{equation}\label{eq:lower heat}
t\ge 0\implies \E_\e\left[\left\|e^{-t\Delta}f(\e)\right\|_X\right]\ge e^{-nt} \E_\e\left[\|f(\e)\|_X\right].
\end{equation}

Fix $s>0$ that will be determined later. Let $g^*_s:\{-1,1\}^n\to
X^*$ be a  normalizing functional of
$e^{-s\Delta}f-\hat{f}(\emptyset)\in L_1(\{-1,1\}^n,X)$, i.e.,
\begin{equation}\label{eq:g bound}
\forall\, \e \in \{-1,1\}^n,\quad \|g_s^*(\e)\|_{X^*}\le 1,
\end{equation}
and
\begin{multline}\label{eq:g duality identity}
\E_\e\left[\left\|e^{-s\Delta}\left(f(\e)-\hat{f}(\emptyset)\right)\right\|_X\right]=\E_\e\left[g_s^*(\e)\left(e^{-s\Delta}\left(f(\e)-\hat{f}(\emptyset)\right)\right)\right]
\\=\sum_{\substack{A\subseteq \{1,\ldots,n\}\\A\neq\emptyset}}e^{-s|A|}\hat{g_s^*}(A)\left(\hat{f}(A)\right).
\end{multline}

In~\cite{Pis86}, Pisier succeeds to relate general geometric cubes
in $X$ to parallelepipeds in $X$ by interpolating $g_s^*$ between
two hypercubes. Specifically, for every $t\ge 0$ consider the
function
$$
(g_s^*)_t:\{-1,1\}^n\times\{-1,1\}^n\to X^*
$$
given by
\begin{equation}\label{eq:def interpolation}
(g_s^*)_t(\e,\d)=\sum_{A\subseteq \{1,\ldots,n\}} \hat{g_s^*}(A)\prod_{i\in A} \left(e^{-t}\e_i+(1-e^{-t})\d_i\right).
\end{equation}
Equivalently, $(g_s^*)_t(\e,\d)=g_s^*\left(e^{-t}\e+(1-e^{-t})\d\right)$, where we interpret the substitution of the vector $e^{-t}\e+(1-e^{-t})\d\in \R^n$ into the function $g_s^*$, which is  defined a priori only on $\{-1,1\}^n$, by formally substituting this vector into the Fourier expansion of $g_s^*$.

Yet another way to interpret $(g_s^*)_t(\e,\d)$ is to note that for every $A\subseteq \{1,\ldots,n\}$,
\begin{eqnarray}\label{eq:expand prod}
&&\nonumber\!\!\!\!\!\!\!\!\!\!\!\!\!\prod_{i\in A} \left(e^{-t}\e_i+(1-e^{-t})\d_i\right)
=W_A(\e)\prod_{i=1}^n\left(e^{-t}+(1-e^{-t})(\e_i\d_i)^{\1_A(i)}\right)\\\nonumber
&=&W_A(\e)\sum_{B\subseteq \{1,\ldots,n\} } e^{-t|B|}(1-e^{-t})^{n-|B|}W_{A\setminus B}(\e\d)\\&=&
\sum_{B\subseteq \{1,\ldots,n\} } e^{-t|B|}(1-e^{-t})^{n-|B|}W_{A\cap B}(\e)W_{A\setminus B}(\d).
\end{eqnarray}
Hence, by substituting~\eqref{eq:expand prod} into~\eqref{eq:def interpolation}  we have
\begin{multline}\label{eq:for bounding}
(g_s^*)_t(\e,\d)\\=\sum_{B\subseteq \{1,\ldots,n\}} e^{-t|B|}(1-e^{-t})^{n-|B|}g_s^*\left(\sum_{i\in B}\e_ie_i+\sum_{i\in \{1,\ldots,n\}\setminus B} \d_ie_i\right),
\end{multline}
where $e_1,\ldots,e_n$ is the standard basis of $\R^n$. In particular, it follows from~\eqref{eq:g bound} and~\eqref{eq:for bounding} that for every $\e,\d\in \{-1,1\}^n$,
\begin{equation}\label{eq:interpolated bound}
\left\|(g_s^*)_t(\e,\d)\right\|_{X^*}\le \sum_{k=1}^n\binom{n}{k}e^{-kt}(1-e^{-t})^{n-k}=1.
\end{equation}

By directly expanding the products in~\eqref{eq:def interpolation} and collecting the terms that are linear in the variables $(\d_1,\ldots,\d_n)$, we see that
\begin{multline}\label{eq:linear part}
(g_s^*)_t(\e,\d)\\=(e^t-1)\sum_{i=1}^n\d_i\sum_{\substack{A\subseteq\{1,\ldots,n\}\\ i\in A}} e^{-|A|t}\hat{g_s^*}(A)W_{A\setminus \{i\}}(\e)+ \Phi_{s,t}^*(\e,\d),
\end{multline}
where the error term $\Phi_{s,t}^*(\e,\d)\in X^*$ satisfies
\begin{equation}\label{eq:orthogonality error}
\E_\d\left[\Phi_{s,t}^*(\e,\d)\left(\sum_{i=1}^n\d_ix_i\right)\right]=0
\end{equation}
for all $\e\in \{-1,1\}^n$ and all choices of vectors $x_1,\ldots,x_n\in X$. By substituting $x_i=\e_i\partial_if(\e)$ into~\eqref{eq:orthogonality error}, and recalling~\eqref{eq:def partial}, we deduce from~\eqref{eq:linear part} that
\begin{eqnarray}\label{eq:for integration}
&&\!\!\!\!\!\!\!\!\!\!\!\!\!\!\!\!\!\!\!\!\!\nonumber\E_\e\E_\delta\left[(g_s^*)_t(\e,\d)
\left(\sum_{i=1}^n \d_i\e_i\partial_i f(\e)\right)\right]\\&=&\nonumber(e^t-1)
\sum_{i=1}^n\sum_{\substack{A\subseteq \{1,\ldots,n\}\\i\in A}}e^{-t|A|}
\hat{g^*_s}(A)\left(\hat{f}(A)\right)\\&=&(e^t-1)\sum_{A\subseteq\{1,\ldots,n\}}
|A|e^{-t|A|}\hat{g^*_s}(A)\left(\hat{f}(A)\right).
\end{eqnarray}
Recalling~\eqref{eq:interpolated bound} we see that
\begin{equation}\label{eq:use norm bound interpolated}
\E_\e\E_\delta\left[(g_s^*)_t(\e,\d)\left(\sum_{i=1}^n
\d_i\e_i\partial_i f(\e)\right)\right] \le
\E_\e\E_\d\left[\left\|\sum_{i=1}^n \d_i\partial_i
f(\e)\right\|_X\right].
\end{equation}
Hence,
\begin{eqnarray}\label{eq:before using lower heat}
&&\nonumber\!\!\!\!\!\!\!\!\!\!\!\!\!\!\!\!\!\!\!\!\!\!\!\!\!\!\!\!\!\E_\e\left[\left\|e^{-s\Delta}\left(f(\e)-\E_\d\left[f(\d)\right]\right)\right\|_X\right]
\stackrel{\eqref{eq:g duality identity}}{=}\sum_{\substack{A\subseteq\{1,\ldots,n\}\\A\neq\emptyset}}
e^{-s|A|}\hat{g^*_s}(A)\left(\hat{f}(A)\right)\\&=&\nonumber
\int_s^\infty \left(\sum_{A\subseteq\{1,\ldots,n\}}
|A|e^{-t|A|}\hat{g^*_s}(A)\left(\hat{f}(A)\right)\right)dt\\&\stackrel{\eqref{eq:for integration}\wedge\eqref{eq:use norm bound interpolated}}{\le}&\nonumber
\left(\int_s^\infty\frac{dt}{e^t-1}\right)\E_\e\E_\d\left[\left\|\sum_{i=1}^n
\d_i\partial_i f(\e)\right\|_X\right]\\&=&\log\left(\frac{e^s}{e^s-1}\right)\E_\e\E_\d\left[\left\|\sum_{i=1}^n
\d_i\partial_i f(\e)\right\|_X\right].
\end{eqnarray}
Recalling~\eqref{eq:lower heat}, it follows from~\eqref{eq:before
using lower heat} that
\begin{multline}\label{eq:to choose s}
\E_\e\left[\left\|f(\e)-\E_\d[f(\d)]\right\|_X\right]\\\le
e^{ns}\log\left(\frac{e^s}{e^s-1}\right)\E_\e\E_\d\left[\left\|\sum_{i=1}^n
\d_i\partial_i f(\e)\right\|_X\right].
\end{multline}
By choosing $s\asymp \frac{\log\log n}{n\log n}$ so as to minimize the right hand side
of~\eqref{eq:to choose s},
\begin{multline}\label{eq:log log}
\E_\e\left[\left\|f(\e)-\E_\d[f(\d)]\right\|_X\right]\\\le
\left(\log n+O(\log\log n)\right)\E_\e\E_\d\left[\left\|\sum_{i=1}^n
\d_i\partial_i f(\e)\right\|_X\right].
\end{multline}

If $X$ has Rademacher type $p>1$, i.e., it satisfies~\eqref{def rad type},
then
\begin{eqnarray}\label{eq:type used}
&&\nonumber\!\!\!\!\!\!\!\!\!\!\!\!\!\!\!\!\!\!\!\E_\e\left[\left\|f(\e)-f(-\e)\right\|_X\right]\le
2\E_\e\left[\left\|f(\e)-\E_\d[f(\d)]\right\|_X\right]\\\nonumber&\lesssim&
T(\log n)\E_\e\left[\left(\sum_{i=1}^n
\|\partial_if(\e)\|_X^p\right)^{1/p}\right]\\&\lesssim& T(\log
n)\left(\sum_{i=1}^n \E_\e\left[\left\|
f(\e)-f(\e_1,\ldots,-\e_i,\ldots,\e_n)\right\|_X^p\right]\right)^{1/p}.
\end{eqnarray}
This proves that if $X$ has Rademacher type $p$ then it almost has
type $p$ as a metric space: inequality~\eqref{def:our metric type}
holds with an additional logarithmic factor. We have therefore
managed to deduce the fully metric ``diagonal versus edge"
inequality~\eqref{eq:diagona-edge} from the corresponding inequality
for parallelepipeds, though with a (conjecturally) redundant factor
of $\log n$. Using similar ideas, for every $\e\in (0,1)$ one can
also deduce the validity of the Enflo type $p$ condition~\eqref{def:enflo type} without the $\log
n$ term but with $p$ replaced by $p-\e$ and the implied constant
depending on $\e$. See Pisier's paper~\cite{Pis86} for the proof of
this alternative tradeoff. A similar tradeoff was previously proved for BMW type
using a different method by Bourgain, Milman and Wolfson~\cite{BMW86}.

\begin{remark}\label{lem:Pisier's inequality}
{\em An inspection of the above argument reveals that there exists a universal constant $C\in (0,\infty)$ such that for every Banach space $(X,\|\cdot\|_X)$, every $q\in [1,\infty]$, every $n\in \N$, and every $f:\{-1,1\}^n\to X$ we have
\begin{multline}\label{eq:pisier's inequality}
\left(\E_\e\left[\left\|f(\e)-\E_\d[f(\d)]\right\|_X^q\right]\right)^{1/q}\\\le
C(\log n) \left(\E_\e\E_\d\left[\left\|\sum_{i=1}^n
\d_i\partial_i f(\e)\right\|_X^q\right]\right)^{1/q}.
\end{multline}
Inequality~\eqref{eq:pisier's inequality} was proved in 1986 by Pisier~\cite{Pis86}, and is known today as {\em Pisier's inequality}. Removal of the $\log n$ factor from~\eqref{eq:pisier's inequality} for Banach spaces with nontrivial Rademacher type would yield a positive solution Enflo's problem (Question~\ref{Q:enf}). Talagrand proved~\cite{Tal93} that there exist Banach spaces for which the $\log n$ term in~\eqref{eq:pisier's inequality} cannot be removed, but we conjecture that if $(X,\|\cdot\|_X)$ has Rademacher type $p>1$ then the $\log n$ term in~\eqref{eq:pisier's inequality} can be replaced by a universal constant (depending on the geometry of $X$). In~\cite{Tal93} it was shown that the $\log n$ term in~\eqref{eq:pisier's inequality} can be replaced by a universal constant if $X=\R$, and in~\cite{Wag00} it was shown that this is true for a general Banach space $X$ if $q=\infty$. In~\cite{NS02,HN12} it is shown that the $\log n$ term in~\eqref{eq:pisier's inequality} can be replaced by a universal constant for certain classes of Banach spaces that include all $L_p(\mu)$ spaces, $p\in (1,\infty)$.}
\end{remark}

\subsection{Unique obstructions to type}\label{sec:obstruction type} There is an obvious obstruction preventing a Banach space $(X,\|\cdot\|_X)$ from having any Rademacher type $p>1$: if $X$ contains well-isomorphic copies of $\ell_1^n=(\R^n,\|\cdot\|_1)$ for all $n\in \N$ then its Rademacher type must be trivial. Indeed, assume that $(X,\|\cdot\|_X)$ satisfies~\eqref{def rad type} and for $n\in \N$ and $D\in (0,\infty)$ suppose that there exists a linear operator $A:\ell_1^n\to X$ satisfying $\|x\|_1\le \|Ax\|_X\le D\|x\|_1$ for all $x\in \ell_1^n$. Letting $\e_1,\ldots,\e_n$ be the standard basis of $\R^n$, it follows that for $x_i=Ae_i$ we have $\|x_i\|_X\le D$ and $$\forall\, \e\in \{-1,1\}^n,\quad \|\e_1x_i+\ldots+\e_n x_n\|_X=\|A(\e_1e_1+\cdots+\e_ne_n)\|_X\ge n.$$ These facts are in conflict with~\eqref{def rad type}, since they force  the constant $T$ appearing in~\eqref{def rad type} to satisfy
\begin{equation}\label{eq:T lower linear}
T\ge \frac{n^{1-\frac{1}{p}}}{D}.
\end{equation}
Pisier proved~\cite{Pis73} that the well-embeddability of
$\{\ell_1^n\}_{n=1}^\infty$ is the {\em only obstruction} to
nontrivial Rademacher type: a Banach space $(X,\|\cdot\|_X)$ fails
to have nontrivial type if and only if for every $\e\in (0,1)$ and
every $n\in \N$ there exists a linear operator $A:\ell_1^n\to X$
satisfying $\|x\|_1\le \|Ax\|_X\le (1+\e)\|x\|_1$ for all $x\in
\ell_1^n$. In other words, once we know that $X$ does not contain
isomorphic copies of $\{\ell_1^n\}_{n=1}^\infty$ we immediately
deduce that the norm on $X$ must satisfy the asymptotically stronger
randomized triangle inequality~\eqref{def rad type}.

As one of the first examples of the applicability of Banach space insights to general metric spaces, Bourgain, Milman and Wolfson~\cite{BMW86} proved the only obstruction preventing a {\em metric space} $(\M,d_\M)$ from having any BMW type $p>1$ is that $\M$ contains bi-Lipschitz copies of the Hamming cubes $\{(\{-1,1\}^n,\|\cdot\|_1)\}_{n=1}^\infty$.

To make this statement precise it would be useful to recall the following standard notation from bi-Lipschitz embedding theory: given two metric space $(\M,d_\M)$ and $(\mathcal N,d_\mathcal{N})$, denote by
\begin{equation}\label{eq:def distortion}
c_{(\M,d_\M)}(\mathcal{N},d_\mathcal{N})
 \end{equation}
(or $c_\M(\mathcal{N})$ if the metrics are clear from the context) the infimum over those $D\in [1,\infty]$ for which there exists $f:\mathcal{N}\to \M$ and a scaling factor $\lambda\in (0,\infty)$ satisfying
$$
\forall\, x,y\in \mathcal{N},\quad \lambda d_{\mathcal{N}}(x,y)\le d_{\M}(f(x),f(y))\le D\lambda d_{\mathcal{N}}(x,t).
$$
This parameter is called the $\M$ distortion of $\mathcal{N}$. When $\M$ is a Hilbert space, this parameter is called the Euclidean distortion of $\mathcal{N}$.

Suppose that $p>1$ and $(\M,d_\M)$ satisfies any of the type $p$ inequalities~\eqref{def:our metric type}, \eqref{def:enflo type} or~\eqref{def:BMW type} (i.e., our definition of metric type, Enflo type, or BMW type, respectively). If $c_\M(\{-1,1\}^n,\|\cdot\|_1)<D$ then there exists $f:\{-1,1\}^n\to \M$ and $\lambda>0$ such that $$\forall\, \e,\d\in \{-1,1\}^n,\quad \lambda\|\e-\d\|_1\le d_\M(f(\e),f(\d))\le D\lambda \|\e-\d\|_1.$$
It follows that $d_\M(f(\e),f(\e_1,\ldots,\e_{i-1},\e_i,\e_{i+1},\ldots,\e_n)\le 2D\lambda$ for all $\e\in \{-1,1\}^n$ and $i\in \{1,\ldots,n\}$. Also, $d_\M(f(\e),f(-\e))\ge 2n\lambda$ for all $\e\in \{-1,1\}^n$. Hence any one of the nonlinear type conditions~\eqref{def:our metric type}, \eqref{def:enflo type} or~\eqref{def:BMW type} implies that
\begin{equation}\label{eq:lower cube distortion type}
c_\M(\{-1,1\}^n,\|\cdot\|_1)\ge \frac{n^{1-\frac{1}{p}}}{T}.
\end{equation}
Bourgain, Milman and Wolfson proved~\cite{BMW86} (see also the exposition in~\cite{Pis86}) that a metric space $(\M,d_\M)$ fails to satisfy the improved randomized triangle inequality~\eqref{def:BMW type} if and only if $c_\M(\{-1,1\}^n,\|\cdot\|_1)=1$ for all $n\in \N$. It is open whether the same ``unique obstruction" result holds true for Enflo type as well.

We note in passing that it follows from~\eqref{eq:type used} and~\eqref{eq:lower cube distortion type} that if $(X,\|\cdot\|_X)$ is a normed space with type $p>1$ then
\begin{equation}\label{eq:dist cube log}
c_X(\{-1,1\}^n,\|\cdot\|_1)\gtrsim_X \frac{n^{1-\frac{1}{p}}}{\log n},
\end{equation}
where the implied constant may depend on the geometry of $X$ but not
on $n$. In combination with~\eqref{eq:T lower linear}, we deduce
that $c_X(\ell_1^n)$ and $c_X(\{-1,1\}^n)$ have the same asymptotic
order of magnitude, up to a logarithmic term which we conjecture can
be removed. This logarithmic term is indeed not needed if $X$ is an
$L_p(\mu)$ space, as shown by Enflo~\cite{Enf69} for $p\in (1,2]$
and in~\cite{NS02} for $p\in (2,\infty)$ (alternative proofs are
given in~\cite{KN06,NPSS06}). It is tempting to guess that
$(\{-1,1\}^n,\|\cdot\|_1)$ has (up to constant factors) the largest
$\ell_p$ distortion among all subsets of $\ell_1$ of size $2^n$.
This stronger statement remains a challenging open problem; it has
been almost solved (again, up to a logarithmic factor) only for
$p=2$ in~\cite{ALN08}.

\section{Metric cotype}\label{sec:cotype}

The natural ``dual" notion to Rademacher type, called {\em
Rademacher cotype}, arises from reversing the inequalities
in~\eqref{def rad type} or~\eqref{eq:q-rad type} (formally, duality
is a subtle issue in this context; see~\cite{MP76,Pis82}).
Specifically, say that a Banach space $(X,\|\cdot\|_X)$ has
Rademacher cotype $q\in [1,\infty]$ if there exists a constant $C\in
(0,\infty)$ such that for every $n\in \N$ and every
$x_1,\ldots,x_n\in X$ we have
\begin{equation}\label{eq:def rad cotype}
\left(\sum_{i=1}^n\|x_i\|_X^q\right)^{1/q}\le C\E\left[\left\|\sum_{i=1}^n \e_i x_i\right\|_X\right].
\end{equation}
It is simple to check that if~\eqref{eq:def rad cotype} holds then
necessarily $q\in [2,\infty]$, and that every Banach space has
Rademacher cotype $\infty$ (with $C=1$). As in the case of
Rademacher type, the notion of Rademacher cotype is of major
importance to Banach space theory; e.g. it affects the dimension of
almost spherical sections of convex bodies~\cite{FLM77}. For more
information on the notion of Rademacher cotype (including a
historical discussion), see the survey~\cite{Mau03} and the
references therein.

As explained in Remark~\ref{rem:kahane}, Kahane's inequality implies
that the requirement~\eqref{eq:def rad cotype} is equivalent (with a
different constant $C$) to the requirement
\begin{equation}\label{eq:q-def rad cotype}
\sum_{i=1}^n\|x_i\|_X^q\le C^q\E\left[\left\|\sum_{i=1}^n \e_i x_i\right\|_X^q\right].
\end{equation}
For simplicity of notation we will describe below metric variants
of~\eqref{eq:q-def rad cotype}, though the discussion carries over
mutatis mutandis also to the natural analogues of~\eqref{eq:def rad
cotype}.

In Banach spaces it is very meaningful  to reverse the inequality in
the definition of Rademacher type, but in metric spaces reversing
the the inequality in the definition of Enflo type results in a
requirement that no metric space can satisfy unless it consists of a
single point (the same assertion holds true for our definition of
metric type~\eqref{def:our metric type} and BMW type~\eqref{def:BMW
type}, but we will only discuss Enflo type from now on). Indeed,
assume that a metric space $(\M,d_\M)$ satisfies
\begin{multline}\label{def:rule out enflo cotype}
\sum_{i=1}^n
\E_\e\left[d_\M\left(f(\e),f(\e_1,\ldots,\e_{i-1},-\e_i,\e_{i+1},\ldots,\e_n)\right)^q\right]\\\le
C^q\E_\e\left[d_\M\left(f(\e),f(-\e)\right)^q\right].
\end{multline}
For all $f:\{-1,1\}^n\to X$.  If $\M$ contains two distinct point
$x_0,y_0$ then apply~\eqref{def:rule out enflo cotype} to a function
$f:\{-1,1\}^n\to \{x_0,y_0\}$ chosen uniformly at random from the
$2^{2^n}$ possible functions of this type. The right hand side
of~\eqref{def:rule out enflo cotype} will always be bounded by
$C^qd_\M(x_0,y_0)^q$, while the expectation over the random function
$f$ of the left hand side of~\eqref{def:rule out enflo cotype} is
$\frac{n}{2}d_\M(x_0,y_0)^q$. Thus necessarily $C\gtrsim n^{1/q}$.

In~\cite{MN08} the following definition of {\em metric cotype} was
introduced. A metric space $(\M,d_\M)$ has metric cotype $q$ if
there exists a constant $C\in (0,\infty)$ such that for every $n\in
\N$ there exists an even integer $m\in\N$ such that  every
$f:\Z_m^n\to \M$ satisfies
\begin{multline}\label{eq:def metric cotype}
\sum_{j=1}^n\sum_{x\in \Z_m^n}
d_\M\left(f\left(x+\frac{m}{2}e_j\right),f(x)\right)^q\\\le
\frac{(Cm)^q}{3^n}\sum_{\e\in \{-1,0,1\}^n}\sum_{x\in \Z_m^n}
d_\M(f(x+\e),f(x))^q.
\end{multline}
Here $e_1,\ldots,e_n$ are the standard basis of the discrete torus
$\Z_m^n$ and addition is performed modulo $m$. The average over
$\e\in \{-1,0,1\}^n$ on the right hand side of~\eqref{eq:def metric
cotype} is natural here, as it corresponds to the $\ell_\infty$
edges of the discrete torus.

It turns out that it is possible to complete the step of the Ribe
program corresponding to Rademacher cotype via the above definition
of metric cotype. Specifically,  the following theorem was proved
in~\cite{MN08}.
\begin{theorem}\label{thm:cotype equiv}
A Banach space $(X,\|\cdot\|_X)$ has Rademacher cotype $q$ if and
only if it has metric cotype $q$.
\end{theorem}

The definition of metric cotype stipulates that for every $n\in \N$
there exists an even integer $m\in \N$ for which~\eqref{eq:def
metric cotype} holds true, but for certain applications it is
important to have good bounds on $m$. The argument that was used
above to rule out~\eqref{def:rule out enflo cotype} shows that if
$(\M,d_\M)$ contains at least two points then the validity
of~\eqref{eq:def metric cotype} implies that $m\gtrsim n^{1/q}$.
In~\cite{MN08} it was proved that one can ensure that $m$ has this
order of magnitude if $X$ is Banach space with nontrivial Rademacher
type.

\begin{theorem}\label{thm:K convex bound}
Let $(X,\|\cdot\|_X)$ be a Banach space with Rademacher cotype
$q<\infty$ and Rademacher type $p>1$. Then~\eqref{eq:def metric
cotype} holds true for some even integer $m\le \kappa n^{1/q}$,
where $\kappa\in (0,\infty)$ depends on the geometry of $X$ but not
on $n$.
\end{theorem}

As an example of an application of Theorem~\ref{thm:K convex bound},
the following characterization of the values of $p,q\in [1,\infty)$
for which $L_p[0,1]$ is uniformly homeomorphic to a subset of
$L_q[0,1]$ was obtained in~\cite{MN08}, answering a question posed
by Enflo~\cite{Enf76} in 1976.

\begin{theorem}\label{thm:LpLq}
Fix $p,q\in [1,\infty)$. Then $L_p[0,1]$ is uniformly homeomorphic
to a subset of $L_q[0,1]$ if and only if either $p\le q$ or $p,q\in
[1,2]$.
\end{theorem}
An analogous result was proved for coarse embeddings in~\cite{MN08}
and for quasisymmetric embeddings in~\cite{Nao06}, answering a question posed by
V\"ais\"al\"a~\cite{Vai99}. The link between Theorem~\ref{thm:K
convex bound} and these results is that one can argue that if
$(\M,d_\M)$ satisfies~\eqref{eq:def metric cotype} with $m\lesssim
n^{1/q}$ then any Banach space that embeds into $\M$ in one of these
senses inherits the cotype of $\M$. Thus, metric cotype (with
appropriate dependence of $m$ on $n$) is an obstruction to a variety
of weak notions of metric embeddings. The following natural open
question is of major importance.
\begin{question}\label{Q:m on n}
Is it possible to obtain the conclusion of Theorem~\ref{thm:K convex
bound} without the assumption that $X$ has nontrivial Rademacher
type? In other words, is it true that any Banach space
$(X,\|\cdot\|_X)$ with Rademacher cotype $q<\infty$
satisfies~\eqref{eq:def metric cotype} with $m\lesssim_X n^{1/q}$?
\end{question}
We conjecture that the answer to Question~\ref{Q:m on n} is
positive, in which case metric cotype itself, without additional
assumptions, would be an invariant for uniform, coarse and
quasisymmetric embeddings. For a general Banach space
$(X,\|\cdot\|_X)$ of Rademacher cotype $q$ the best known bound on
$m$ in terms of $n$ in~\eqref{eq:def metric cotype}, due
to~\cite{GMN11}, is $m\lesssim n^{1+1/q}$.

There are additional applications of metric cotype for which the
dependence of $m$ on $n$ in~\eqref{eq:def metric cotype} has no
importance. In analogy to the discussion in
Section~\ref{sec:obstruction type}, it was proved by Maurey and
Pisier~\cite{MP76} that the only obstruction that can prevent a
Banach space $(X,\|\cdot\|_X)$ from having finite Rademacher cotype
is the presence of well-isomorphic copies of
$\{\ell_\infty^n\}_{n=1}^\infty$. In~\cite{MN08} a variant of the
definition of metric cotype was given, in analogy to the
Bourgain-Milman-Wolfson variant of Enflo type, and it was shown that
a {\em metric space} has finite metric cotype in this sense if and
only if $c_\M(\{1,\ldots,m\}^n,\|\cdot\|_\infty)=1$ for every
$m,n\in \N$. This nonlinear Maurey-Pisier theorem was used
in~\cite{MN08} to prove the following dichotomy result for general
metric spaces, answering a question posed by Arora, Lov\'asz,
Newman, Rabani, Rabinovich and Vempala~\cite{ALNRRV06} and improving
a Ramsey-theoretical result of Matou\v{s}ek~\cite{Mat92}.

\begin{theorem}[General metric dichotomy~\cite{MN08}]\label{thm:dich
MN} Let $\F$ be a family of metric spaces. Then one of the following
dichotomic possibilities must hold true.
\begin{itemize}
\item For every finite metric space $(\M,d_\M)$ and for every $\e\in
(0,\infty)$ there exists $\mathcal{N}\in \F$ such that
$$c_{\mathcal{N}}(\M)\le 1+\e.$$
\item There exists $\alpha(\F),\kappa(\F)\in (0,\infty)$
and for each $n\in \N$ there exists an $n$-point metric space
$(\M_n,d_{\M_n})$ such that for every $\mathcal{N}\in \F$ we
have
$$
c_{\mathcal{N}}(\M_n)\ge \kappa(\F)(\log n)^{\alpha(\F)}.
$$
\end{itemize}
\end{theorem}
We refer to~\cite[Sec.~1.1]{MN08-markov} and \cite{MN11-dich}, as
well as the survey paper~\cite{Men09}, for more information on the
theory of metric dichotomies. Theorem~\ref{thm:dich MN} leaves the
following fundamental question open.

\begin{question}[Metric cotype dichotomy problem~\cite{MN08,MN11-dich}]
Can one replace the constant $\alpha(\F)$ of Theorem~\ref{thm:dich
MN} by a constant $\alpha\in (0,\infty)$ that is independent of the
family $\F$? It isn't even known if one can take $\alpha(\F)=1$ for
all families of metric spaces $\F$.
\end{question}

\section{Markov type and cotype}\label{sec:Mtype}

As part of his investigation of the Lipschitz extension
problem~\cite{Bal92}, K. Ball introduced a stronger version of type
of metric spaces called {\em Markov type}. Other than its
applications to Lipschitz extension, the notion of Markov type has
found many applications in embedding theory, some of which will be
described  in Section~\ref{sec:random walks}.

Recall that a stochastic process $\{Z_t\}_{t=0}^\infty$ taking
values in $\{1,\ldots,n\}$ is called a {\em stationary reversible
Markov chain} if there exists an $n$ by $n$ stochastic matrix
$A=(a_{ij})$ such that for every $t\in \N\cup\{0\}$ and every
$i,j\in \{1,\ldots,n\}$ we have
$\Pr\left[Z_{t+1}=j|Z_t=i\right]=a_{ij}$, for every $i\in
\{1,\ldots,n\}$ the probability $\pi_i=\Pr[Z_t=i]$ does not depend
on $t$, and $\pi_i a_{ij}=\pi_ja_{ji}$ for all $i,j\in
\{1,\ldots,n\}$.

A metric space $(\M,d_\M)$ is said to have Markov type $p\in
(0,\infty)$ with constant $M\in (0,\infty)$ if for every $n\in \N$,
every stationary reversible Markov chain on $\{1,\ldots,n\}$, every
$f:\{1,\ldots,n\}\to \M$ and every time $t\in  \N$ we have
\begin{equation}\label{eq:def Mtype}
\E\left[d_\M(f(Z_t),f(Z_0))^p\right]\le M^pt\E\left[d_\M(f(Z_1),f(Z_0))^p\right].
\end{equation}

Note that the triangle inequality implies that every metric space
has Markov type $1$ with constant $1$. Ball proved~\cite{Bal92} that
if $p\in [1,2]$ then any $L_p(\mu)$ space has Markov type $p$ with
constant $1$. Thus, while it is well-known that the standard random
walk on the integers is expected to be at distance at most
$\sqrt{t}$ from the origin after $t$ steps, Ball established the
less well-known fact that {\em any} stationary reversible random
walk in Hilbert space has this property. If a metric space has
Markov type $p$ then it also has Enflo type $p$, as proved
in~\cite{NS02}. In essence, Enflo type $p$ corresponds
to~\eqref{eq:def Mtype} in the special case when the Markov chain is
the standard random walk on the Hamming cube $\{-1,1\}^n$. Thus the
Markov type $p$ condition is a strengthening of Enflo type, its
power arising in part from the flexibility to choose any stationary
reversible Markov chain whatsoever.

\begin{remark}\label{rem:edge Markov}
{\em We do not know to what extent Enflo type $p>1$ implies Markov
type $p$ (or perhaps Markov type $q$ for some $1<q<p$). When the
metric space $(\M,d_\M)$ is an unweighted graph equipped with the
shortest path metric (as is often the case in applications), it is
natural to introduce an intermediate notion of Markov type in which
the Markov chains are only allowed to ``move" along edges, i.e., by
considering~\eqref{eq:def Mtype} under the additional restriction
that if $a_{ij}>0$ then $\{f(i),f(j)\}$ is an edge. Call this notion
``edge Markov type $p$". For some time it was unclear whether edge
Markov type $p$ implies Markov type $p$. However, in~\cite{NP08} it
was shown that there exists a Cayley graph with edge Markov type $p$
for every $1<p<\frac43$ that does not have nontrivial Enflo type. It
is unknown whether a similar example exists with edge Markov type
$2$.}
\end{remark}

In~\cite{NPSS06} it was shown that for $p\in [2,\infty)$ any
$L_p(\mu)$ space has Markov type $2$ (with constant $M\asymp
\sqrt{p})$. More generally, it is proved in~\cite{NPSS06} that any
$p$-uniformly smooth Banach space has Markov type $p$. Uniform
smoothness, and its dual notion uniform convexity, are defined as
follows. Let $(X,\|\cdot\|_X)$ be a normed space with unit sphere
$S_X=\{x\in X:\ \|x\|_X=1\}$. The {\em modulus of uniform convexity}
of $X$ is defined for $\e\in [0,2]$ as
\begin{equation}\label{def:convexity}
\delta_X(\e)=\inf\left\{ 1-\frac{\|x+y\|_X}{2}:\
x,y\in S_X,\ \|x-y\|_X=\e\right\}.
\end{equation}
$X$ is said to be {\em uniformly convex} if $\delta_X(\e)>0$ for all
$\e\in (0,2]$. $X$ is said to have modulus of uniform convexity of
power type $q$ if there exists a constant $c\in (0,\infty)$ such
that $\delta_X(\e)\ge c\,\e^q$ for all $\e\in [0,2]$. It is
straightforward to check that in this case necessarily $q\ge 2$. The
{\em modulus of uniform smoothness} of $X$ is define for $\tau\in
(0,\infty)$ as
\begin{equation}\label{eq:def smoothness}
\rho_X(\tau)\eqdef \left\{\frac{\|x+\tau y\|_X+\|x-\tau
y\|_X}{2}-1:\ x,y\in S_X\right\}.
\end{equation}
$X$ is said to be {\em uniformly smooth} if $\lim_{\tau\to
0}\rho_X(\tau)/\tau=0$. $X$ is said to have modulus of uniform
smoothness of power type $p$ if there exists a constant $C\in
(0,\infty)$ such that $\rho_X(\tau)\le C\tau^p$ for all $\tau\in
(0,\infty)$. It is straightforward to check that in this case
necessarily $p\in [1,2]$.

For concreteness, we recall~\cite{Han56} (see also~\cite{BCL94}) that if $p\in (1,\infty)$ then $\delta_{\ell_p}(\e)\gtrsim_p \e^{\max\{p,2\}}$ and $\rho_{\ell_p}(\tau)\lesssim_p \tau^{\min\{p,2\}}$.  The moduli appearing
in~\eqref{def:convexity} and~\eqref{eq:def smoothness} relate to
each other via the following classical duality formula of
Lindenstrauss~\cite{Lin63}:
\begin{equation}\label{eq:lindenstrauss duality}
\rho_{X^*}(\tau)=\sup\left\{\frac{\tau\e}{2}-\delta_X(\e):\ \e\in
[0,2]\right\}.
\end{equation}
An important theorem of Pisier~\cite{Pis75-martingales} asserts that
$X$ admits an equivalent uniformly convex norm if and only if it
admits an equivalent norm whose modulus of uniform convexity is of
power type $q$ for some $q\in [2,\infty)$. Similarly, $X$ admits an
equivalent uniformly smooth norm if and only if it admits an
equivalent norm whose modulus of uniform smoothness is of power type
$p$ for some $p\in (1,2]$.

We will revisit these notions later, but at this point it suffices
to say that, as proved in~\cite{NPSS06}, any Banach space that
admits an equivalent norm whose  modulus of uniform smoothness is of
power type $p$ also has Markov type $p$. The relation between
Rademacher type $p$ and Markov type $p$ is unclear. While for every
$p\in (1,2]$ there exist Banach spaces with Rademacher $p$ that do
not admit any equivalent uniformly smooth norm~\cite{Jam78,PX87},
the following question remains open.

\begin{question}\label{Q:Mtype}
Does there exists a Banach space $(X,\|\cdot\|_X)$ with Markov type
$p>1$ yet $(X,\|\cdot\|_X)$ does not admit a uniformly smooth norm?
\end{question}

In addition to uniformly smooth Banach spaces, the Markov type of
several spaces of interest has been computed. For example, the
following classes of metric spaces are known to have Markov type
$2$: weighted graph theoretical trees~\cite{NPSS06}, series parallel
graphs~\cite{BKL07}, hyperbolic groups~\cite{NPSS06}, simply
connected Riemmanian manifolds with pinched negative sectional
curvature~\cite{NPSS06}, Alexandrov spaces of nonnegative
curvature~\cite{Oht09}. Also, the Markov type of certain
$p$-Wasserstein spaces was computed in~\cite{ANN10}.

Recall that a metric space $(M,d_\M)$ is {\em doubling} if there
exists $K\in \N$ such that for every $x\in \M$ and $r\in (0,\infty)$
there exist $y_1,\ldots,y_K\in \M$ such that $B(x,r)\subseteq
B(y_1,r/2)\cup\ldots\cup B(y_K,r/2)$, i.e., every ball in $\M$ can
be covered by $K$ balls of half the radius.  Here, and in what
follows, $B(z,\rho)=\{w\in \M:\ d_\M(z,w)\le \rho\}$ for all $z\in
\M$ and $\rho\ge 0$. The parameter $K$ is called a {\em doubling
constant} of $(\M,d_\M)$.
\begin{question}\label{Q:doubling}
Does every doubling metric space have Markov type 2? Specifically,
does the Heisenberg group have Markov type 2?
\end{question}
Assouad's embedding theorem~\cite{Ass83} says that if $(\M,d_\M)$ is
doubling then the metric space $(\M,d_\M^{1-\e})$ admits a
bi-Lipschitz embedding into Hilbert space for every $\e\in (0,1)$.
As observed in~\cite{NPSS06}, this implies that if $(\M,d_\M)$ is
doubling then it has Markov type $p$ for all $p<2$. It was also
shown in~\cite{NPSS06} that if $(\M,d_\M)$ is doubling with constant
$K\in (1,\infty)$ then for every $n\in \N$, every stationary
reversible Markov chain on $\{1,\ldots,n\}$, every
$f:\{1,\ldots,n\}\to \M$ and every time $t\in  \N$,
\begin{multline}\label{eq:weak Mtype}
\forall\, u>0,\quad \Pr\left[d_\M(f(Z_t),f(Z_0))\ge
u\sqrt{t}\right]\\\le \frac{O((\log
K)^2)}{u^2}\E\left[d_\M(f(Z_1),f(Z_0))^2\right].
\end{multline}
Thus, one can say that doubling spaces have ``weak Markov type $2$".
Using the method of~\cite{Rab08} it is also possible to show that
doubling spaces have Enflo type $2$.

Further support of a positive answer to Question~\ref{Q:doubling}
was obtained in~\cite{NPSS06}, where it was shown that the Laakso
graphs $\{G_k\}_{k=0}^\infty$ have Markov type $2$. These graphs are
defined~\cite{Laa02} iteratively by letting $G_0$ be a single edge
and $G_{i+1}$ is obtained by replacing the middle third of each edge
of $G_i$ by a quadrilateral; see Figure~\ref{fig:lang}.
\begin{figure}[ht]
\centering \fbox{
\begin{minipage}{4.7in}
\centering
\includegraphics{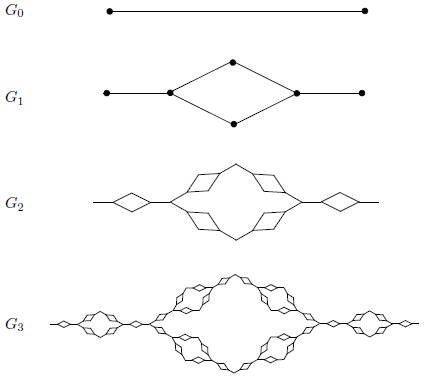}
 \caption{The first four Laakso graphs.} \label{fig:geronimus}
  \end{minipage}
}\label{fig:lang}
\end{figure}
Equipped with their shortest path metric, each Laakso graph $G_k$ is
doubling with constant $6$ (see the proof of~\cite[Thm.~2.3]{LP01}),
yet, as proved by Laakso~\cite{Laa02}, we have $\lim_{k\to \infty}
c_{\ell_2}(G_k)=\infty$ (in fact~\cite[Thm.~2.3]{LP01} asserts that
$c_{\ell_2}(G_k)\gtrsim \sqrt{k}$). The graphs
$\{G_k\}_{k=0}^\infty$ are among the standard examples of doubling
spaces that do not well-embed into Hilbert space, yet, as proved
in~\cite{NPSS06}, they do have Markov type $2$. The Heisenberg
group, i.e., the group of all $3$ by $3$ matrices generated by the
set
\begin{equation*}
S=\left\{ \begin{pmatrix}
  1 & 1&0 \\
   0 & 1&0\\
   0&0&1
   \end{pmatrix},  \begin{pmatrix}
  1 & -1&0 \\
   0 & 1&0\\
   0&0&1
   \end{pmatrix},\begin{pmatrix}
  1 & 0&0 \\
   0 & 1&1\\
   0&0&1
   \end{pmatrix},\begin{pmatrix}
  1 & 0&0 \\
   0 & 1&-1\\
   0&0&1
   \end{pmatrix}\right\},
   \end{equation*}
and equipped with the associated word metric, is another standard
example of a doubling space that does not admit a bi-Lipschitz
embedding into Hilbert space~\cite{Pan89,Sem96}. However, as
indicated in Question~\ref{Q:doubling}, the intriguing problem
whether the Heisenberg group has Markov type $2$ remains open.

Note that by the nonlinear Maurey-Pisier theorem~\cite{MN08}, as
discussed in Section~\ref{sec:cotype}, a doubling metric space must
have finite metric cotype. The Laakso graphs $\{G_k\}_{k=0}^\infty$,
being examples of series parallel graphs, admit a bi-Lipschitz
embedding into $\ell_1$ with distortion bounded by a constant
independent of $k$, as proved in~\cite{GNRS04}. Since $\ell_1$ has
Rademacher cotype $2$, it follows from Theorem~\ref{thm:cotype equiv} that the Laakso graphs have metric
cotype $2$ (with the constant $C$ in~\eqref{eq:def metric cotype}
taken to be independent of $k$). We do not know if all doubling
metric spaces have metric cotype $2$. The Heisenberg group is a
prime example for which this question remains open. Note that the
Heisenberg group does not embed into any $L_1(\mu)$
space~\cite{CK10}. Therefore the above reasoning for the Laakso
graphs does not apply to the Heisenberg group.

Metric trees and the Laakso graphs are nontrivial examples of planar
graphs that have Markov type $2$. This result of~\cite{NPSS06} was
extended to all series parallel graphs in~\cite{BKL07}. It was also
shown in~\cite{NPSS06} that any planar graph satisfies the weak
Markov type $2$ inequality~\eqref{eq:weak Mtype}, and
using~\cite{Rab08} one can show that planar graphs have Enflo type
$2$. It remains open whether all planar graphs have Markov type $2$.

\subsection{Lipschitz extension via Markov type and cotype} Here we
explain Ball's original motivation for introducing Markov type.

Ball also introduced in~\cite{Bal92} a {\em linear} property of
Banach spaces that he called Markov cotype $2$, and he indicated a
two-step definition that could be used to extend this notion to
general metric spaces. Motivated by Ball's ideas, the following
variant of his definition was introduced in~\cite{MN10-calculus}. A
metric space $(\M,d_\M)$ has {\em metric Markov cotype} $q\in
(0,\infty)$ with constant $C\in (0,\infty)$ if  for every $m,n\in
\N$, every $n$ by $n$ symmetric stochastic matrix $A=(a_{ij})$, and
every $x_1,\ldots,x_n\in \M$, there exist $y_1,\ldots,y_n\in \M$
satisfying
\begin{multline}\label{def:markov cotype}
\sum_{i=1}^n d_\M(x_i,y_i)^q+m\sum_{i=1}^n\sum_{j=1}^n a_{ij}
d_\M(y_i,y_j)^q\\\le C^q\sum_{i=1}^n\sum_{j=1}^n
\left(\frac{1}{m}\sum_{t=0}^{m-1}A^t\right)_{ij}d_\M(x_i,x_j)^q.
\end{multline}

To better understand the meaning of~\eqref{def:markov cotype},
observe that the Markov type $p$ condition for $(\M,d_\M)$ implies
that
\begin{equation}\label{eq:to reverse type}
\sum_{i=1}^n\sum_{j=1}^n (A^m)_{ij}d_\M(x_i,x_j)^p\le M^p\sum_{i=1}^n\sum_{j=1}^n a_{ij}d_\M(x_i,x_j)^p.
\end{equation}
Thus~\eqref{def:markov cotype} aims to reverse the direction of the
inequality in~\eqref{eq:to reverse type}, with the following
changes. One is allowed to pass from the initial points
$x_1,\ldots,x_n\in \M$ to new points $y_1,\ldots,y_m\in \M$. The
first summand in the left hand side of~\eqref{def:markov cotype}
ensures that on average $y_i$ is close to $x_i$. The remaining terms
in~\eqref{def:markov cotype} correspond to the reversal
of~\eqref{eq:to reverse type}, with $\{x_i\}_{i=1}^n$ replaced by
$\{y_i\}_{i=1}^n$  in the left hand side, and the power $A^m$
replaced by the Ces\`aro average $\frac{1}{m}\sum_{t=0}^{m-1}
A^{t}$.

Although~\eqref{def:markov cotype} was inspired by Ball's ideas, the
formal relation between the above definition of metric Markov cotype
and Ball's original definition in~\cite{Bal92} is unclear. We chose
to work with the above definition since it suffices for the purpose
of Ball's original application, and in addition it can be used for
other purposes. Specifically, metric Markov cotype is key to the
development of calculus for nonlinear spectral gaps and the
construction of super-expanders; an aspect of the Ribe program that
we will not describe here for lack of space
(see~\cite{MN10-calculus}).

For $q\in [1,\infty)$, a metric space $(\M,d_\M)$ is called
$W_q$-barycentric with constant $\Gamma\in (0,\infty)$ if for every
finitely supported probability measure $\mu$ on $\M$ there exists a
point $\beta_\mu\in \M$ (a barycenter of $\mu$) such that
$\beta_{\d_x}=x$ for all $x\in X$ and for every two finitely
supported probability measures $\mu,\nu$ we have
$d_\M(\beta_\mu,\beta_\nu)\le \Gamma W_q(\mu,\nu)$, where
$W_q(\cdot,\cdot)$ denotes the $q$-Wasserstein  metric
(see~\cite[Sec.~7.1]{Vil03}). Note that by convexity every Banach
space is $W_q$-barycentric with constant $1$.

The following theorem from~\cite{MN12-ext} is a metric space variant
of Ball's Lipschitz extension theorem~\cite{Bal92} (the proof
follows the same ideas as in~\cite{Bal92} with some technical
differences of lesser importance).

\begin{theorem}\label{thm:ball ext}
Fix $q\in (0,\infty)$ and let $(\M,d_\M)$,
$(\mathcal{N},d_\mathcal{N})$ be two metric spaces. Assume that $\M$
has Markov type $q$ with constant $M$ and $\mathcal{N}$ has metric
Markov cotype $q$ with constant $C$. Assume also that $\mathcal{N}$
is $W_q$-barycentric with constant $\Gamma$. Then for every
$A\subseteq \M$, every finite $S\subseteq \M\setminus A$, and every
Lipschitz mapping $f:A\to \N$ there exists $F:A\cup S\to \N$
satisfying $F(x)=f(x)$ for all $x\in A$ and
$$
\|F\|_{\Lip}\lesssim_{\Gamma,M,C} \|f\|_{\Lip},
$$
where the implied constant depends only on $\Gamma, M, C$.
\end{theorem}

Ball proved Theorem~\ref{thm:ball ext} when  $\mathcal{N}$ is a
Banach space, $q=2$, and the metric Markov cotype assumption is
replaced by his linear notion of Markov cotype. He proved that every
Banach space that admits an equivalent norm with modulus of uniform
convexity of power type $2$ satisfies his notion of Markov cotype
$2$. In combination with~\cite{NPSS06}, it follows that the
conclusion of Theorem~\ref{thm:ball ext} holds true if $\M$ is a
Banach space that admits an equivalent norm with modulus of uniform
smoothness of power type $2$ and $\mathcal{N}$ is a Banach space
that admits an equivalent norm with modulus of uniform convexity of
power type $2$. In particular, for $1<q\le 2\le p<\infty$ we can
take $\M=\ell_p$ and $\mathcal{N}=\ell_q$. This answers positively a
1983 conjecture of Johnson and Lindenstrauss~\cite{JL84}. The
motivation of the question of Johnson and Lindenstrauss belongs to
the Ribe program (see also~\cite{MP84}): to obtain a metric analogue
of a classical theorem of Maurey~\cite{Mau74} that implies this
result for linear operators, i.e., in Maurey's setting $A$ is a
closed linear subspace and $f$ is a linear operator, in which case
the conclusion is that $F:\M\to \mathcal{N}$ is a bounded linear
operator with $\|F\|\lesssim_{M,C} \|f\|$. Examples of applications
of Ball's extension theorem can be found in~\cite{Nao01,MN06-ball}.

In~\cite{MN12-ext} it is shown that if $(X,\|\cdot\|_X)$ is a Banach
space that admits an equivalent norm with modulus of uniform
convexity of power type $q$ then it has metric Markov cotype $q$.
Also, it is shown in~\cite{MN12-ext} that certain barycentric metric
spaces have metric Markov cotype $q$; this is true in particular for
$CAT(0)$ spaces, and hence also all simply connected manifolds of
nonpositive sectional curvature (see~\cite{BH99}). These facts, in
conjunction with Theorem~\ref{thm:ball ext}, yield new Lipschitz
extension theorems; see~\cite{MN12-ext}.

For Banach
spaces the notion of metric Markov cotype $q$ does not coincide with
Rademacher cotype: one can deduce from a clever construction of Kalton~\cite{Kal11}
 that there exists a closed linear subspace $X$ of
$L_1$ (hence $X$ has Rademacher cotype $2$) that does not have
metric Markov cotype $q$ for any $q<\infty$. The following natural
question remains open.
\begin{question}\label{Q:L1}
Does $\ell_1$ have metric Markov cotype 2?
\end{question}

By Theorem~\ref{thm:ball ext}, a positive solution of
Question~\ref{Q:L1} would answer a well known question of
Ball~\cite{Bal92}, by showing that every Lipschitz function from a
subset of $\ell_2$ to $\ell_1$ can be extended to a Lipschitz
function defined on all of $\ell_2$. See~\cite{MM10} for
ramifications of this question in theoretical computer science.

\section{Markov convexity}\label{sec:Mconvexity}

Deep work of James~\cite{Jam72-self,Jam72} and Enflo~\cite{Enf72}
implies that a Banach space $(X,\|\cdot\|_X)$ admits an equivalent
uniformly convex norm if and only if it admits an equivalent
uniformly smooth norm, and these properties are equivalent to the
assertion that any Banach space $(Y,\|\cdot\|_Y)$ that is finitely
representable in $X$ must be reflexive. Such spaces are called {\em
superreflexive} Banach spaces. The Ribe program suggests that
superreflexivity has a purely metric reformulation. This is indeed
the case, as proved by Bourgain~\cite{Bou86}.

For $k,n\in \N$ let $T^k_n$ denote the complete $k$-regular tree of
depth $n$, i.e., the finite unweighted rooted tree such that the
length of any root-leaf path equals $n$ and every non-leaf vertex
has exactly $k$ adjacent vertices. We shall always assume that
$T^k_n$ is equipped with the shortest path metric
$d_{T_r^k}(\cdot,\cdot)$, i.e., the distance between any two
vertices is the sum of their distances to their least common
ancestor. Bourgain's characterization of
superreflexivity~\cite{Bou86} asserts that a Banach space
$(X,\|\cdot\|_X)$ admits an equivalent uniformly convex norm if and
only if for all $k\ge 3$ we have \begin{equation}\label{eq:bourgain
superreflexiv char} \lim_{n\to\infty}c_X(T_n^k)=\infty.
\end{equation}
Bourgain's proof also yields the following asymptotic computation of
the Euclidean distortion of $T_n^k$:
\begin{equation}\label{eq:bourgain tree}
k\ge 3\implies c_{\ell_2}\left(T^k_n\right)\asymp \sqrt{\log n}.
\end{equation}
All known proofs of the lower bound $c_{\ell_2}\left(T^k_n\right)\gtrsim
\sqrt{\log n}$ are non-trivial (in addition to the original proof
of~\cite{Bou86}, alternative proofs appeared
in~\cite{Mat99,LS03,LNP09}). In this section we will describe a
proof of~\eqref{eq:bourgain tree} from the viewpoint of random
walks.

It is a nontrivial consequence of the work of
Pisier~\cite{Pis75-martingales} that the Banach space property of
admitting an equivalent norm whose modulus of uniform convexity has
power type $p$ is an isomorphic local linear property. As such,
 the Ribe program calls for a purely metric
reformulation of this property. Since Pisier
proved~\cite{Pis75-martingales} that a Banach space is
superrreflexive if and only if it admits an equivalent norm whose
modulus of uniform convexity has power type $p$ for some $p\in
[2,\infty)$, this question should be viewed as asking for a
quantitative refinement of Bourgain's metric characterization of
superreflexivity.

The following definition is due to~\cite{LNP09}. Let $\{Z_t\}_{t\in
\Z}$ be a Markov chain on a state space $\Omega$. Given integers
$k,s\ge 0$,  denote by $\{\widetilde Z_t(s)\}_{t\in \Z}$ the process
that equals $Z_t$ for time $t\le s$, and evolves independently (with
respect to the same transition probabilities) for time $t > s$. Fix
$p>0$. A metric space $(\M,d_\M)$ is called {Markov $p$-convex with
constant $\Pi$} if for every Markov chain $\{Z_t\}_{t\in \Z}$ on a
state space $\Omega$, and every mapping $f : \Omega \to \M$,
\begin{multline}\label{eq:def-mconvex}
\sum_{s=0}^{\infty}\sum_{t\in \Z}\frac{\E\left[
d_\M\left(f(Z_t),f\left(\widetilde
Z_t\left(t-2^{s}\right)\right)\right)^p\right]}{2^{sp}}
\\\le \Pi^p
\cdot \sum_{t\in \Z}\E \big[d_\M(f(Z_t),f(Z_{t-1}))^p\big].
\end{multline}
The infimum over those $\Pi\in [0,\infty]$ for
which~\eqref{eq:def-mconvex} holds for all Markov chains is called
the Markov $p$-convexity constant of $\M$, and is denoted
$\Pi_p(\M)$. We say that $(\M,d_\M)$ is Markov $p$-convex if
$\Pi_p(\M) < \infty$.

We will see in a moment how to work with~\eqref{eq:def-mconvex}, but
we first state the following theorem, which constitutes a completion
of the step of the Ribe program that corresponds to the Banach space
property of admitting an equivalent norm whose modulus of uniform
convexity has power type $p$. The ``only if" part of this statement
is due to~\cite{LNP09} and the ``if" part is due
to~\cite{MN08-markov}.

\begin{theorem}\label{thm:Mconv}
Fix $p\in [2,\infty)$. A Banach space $(X,\|\cdot\|_X)$ admits an
equivalent norm whose modulus of uniform convexity has power type
$p$ if and only if $(X,\|\cdot\|_X)$ is Markov $p$-convex.
\end{theorem}

The meaning of~\eqref{eq:def-mconvex} will become clearer once we
examine the following example. Fix an integer $k\ge 3$ and let
$\{Z_t\}_{t\in \Z}$ be the following Markov chain whose state space
is $T_n^k$. $Z_t$ equals the root of $T_n^k$ for $t\le 0$, and
$\{Z_t\}_{t\in \N}$ is the standard {\em outward} random walk (i.e.,
if $0\le t<n$ then $Z_{t+1}$ is distributed uniformly over the $k-1$
neighbors of $Z_t$ that are further away from the root than $Z_t$),
with absorbing states at the leaves. Suppose that $(\M,d_\M)$ is a
metric space that is Markov $p$-convex with constant $\Pi$, and for
some $\lambda,D\in (0,\infty)$ we are given an embedding $f:T_n^k\to
\M$ that satisfies $\lambda d_{T_n^k}(x,y)\le d_\M(f(x),f(y))\le
D\lambda d_{T_n^k}(x,y)$ for all $x,y\in T_n^k$. For every $s,t\in
\N$ such that $2^s\le t\le n$, with probability at least $1-1/(k-1)$
the vertices $Z_{t-2^s+1}$ and $\widetilde Z_{t-2^s+1}(t-2^s)$ are
distinct, in which case $d_{T_n^k}(Z_t,\widetilde
Z_{t}(t-2^s))=2^{s+1}$. It therefore follows
from~\eqref{eq:def-mconvex} that
\begin{multline*}
\lambda^pn\log n\lesssim \sum_{s=0}^{\infty}\sum_{t\in \Z}\frac{\E\left[
d_\M\left(f(Z_t),f\left(\widetilde
Z_t\left(t-2^{s}\right)\right)\right)^p\right]}{2^{sp}}
\\\le \Pi^p
\cdot \sum_{t\in \Z}\E \big[d_\M(f(Z_t),f(Z_{t-1}))^p\big]\le \Pi^pD^p\lambda^pn.
\end{multline*}
Consequently, $$c_\M(T_n^k)\gtrsim \frac{1}{\Pi_p(\M)}(\log
n)^{1/p}.$$ In particular, when $\M=\ell_2$ this
explains~\eqref{eq:bourgain tree}.

A different choice of Markov chain can be used in combination with
Markov convexity to compute the asymptotic behavior of the Euclidean
distortion of the {\em lamplighter group} over $\Z_n$;
see~\cite{LNP09,ANV10}. Similar reasoning also applies to the Laakso
graphs $\{G_k\}_{k=0}^\infty$, as depicted in Figure~\ref{fig:lang}.
In this case let $\{Z_t\}_{t=0}^\infty$ be the Markov chain that
starts at the leftmost vertex of $G_k$ (see
Figure~\ref{fig:lang}), and at each step moves to the right. If
$Z_t$ is a vertex of degree $3$ then $Z_{t+1}$ equals one of the two
vertices on the right of $Z_t$, each with probability $\frac12$. An
argument along the above lines (see~\cite[Sec.~3]{MN08-markov})
yields
$$
c_\M(G_k)\gtrsim \frac{1}{\Pi_p(\M)}k^{1/p}.
$$
This estimate is sharp when $\M=\ell_q$ for all $q\in (1,\infty)$.
Note that since the Laakso graphs are doubling, they do not contain
bi-Lipschitz copies of $T_n^3$ with distortion bounded independently
of $n$. Thus the Markov convexity invariant applies equally well to
trees and Laakso graphs, despite the fact that these examples are
very different from each other as metric spaces. Recently Johnson
and Schechtman~\cite{JS09} proved that if for a  Banach space
$(X,\|\cdot\|_X)$ we have $\lim_{k\to \infty} c_X(G_k)=\infty$ then
$X$ is superreflexive. Thus the nonembeddability of the Laakso
graphs is a
 metric characterization of superreflexivity that is different from
 Bourgain's characterization~\eqref{eq:bourgain superreflexiv char}.

In addition to uniformly convex Banach spaces, other classes of
metric spaces for which Markov convexity has been computed include
Alexandrov spaces of nonnegative curvature~\cite{AN10} (they are
Markov $2$-convex) and the Heisenberg group (it is Markov
$4$-convex, as shown by Sean Li). Markov convexity has several
applications to metric geometry, including a characterization of
tree metrics that admit a bi-Lipschitz embedding into Euclidean
space~\cite{LNP09}, a polynomial time approximation algorithm to
compute the $\ell_p$ distortion of tree metrics~\cite{LNP09}, and
applications to the theory of Lipschitz
quotients~\cite{MN08-markov}.

\section{Metric smoothness?} Since a Banach space admits an equivalent
uniformly convex norm if and only if it admits an equivalent
uniformly smooth norm, Bourgain's characterization of
superreflexivity implies that, for every $k\ge 3$, a Banach space
$X$ admits an equivalent uniformly smooth norm if and only if
$\lim_{n\to \infty} c_X(T_n^k)=\infty$. Nevertheless, a subtlety of
this problem appears if one is interested in equivalent norms whose
modulus of uniform smoothness has a given power type. Specifically,
a Banach space $X$ admits an equivalent norm whose modulus of
uniform smoothness has power type $p$ if and only if $X^*$ admits an
equivalent norm whose modulus of uniform convexity has power type
$p/(p-1)$; this is an immediate consequence
of~\eqref{eq:lindenstrauss duality}. Despite this fact, and in
contrast to Theorem~\ref{thm:Mconv}, we do not know how to complete
the Ribe program for the property of admitting an equivalent norm
whose modulus of uniform smoothness has power type $p$. The presence
of Trees and Laakso graphs is a natural obstruction to uniform
convexity, but it remains open to isolate a natural (and useful)
family of metric spaces whose presence is an obstruction to uniform
smoothness of power type $p$.

\section{Bourgain's discretization problem}\label{sec:bourgain
discretization} Let $(X,\|\cdot\|_X)$ and $(Y,\|\cdot\|_Y)$ be
normed spaces with unit balls $B_X$ and $B_Y$, respectively. For
$\e\in (0,1)$ let $\d_{X\hookrightarrow Y}(\e)$ be the supremum over
those $\d\in (0,1)$ such that every $\d$-net $\mathcal N_\d$ in
$B_X$ satisfies $c_Y(\mathcal{N}_\d)\ge (1-\e) c_Y(X)$.
$\d_{X\hookrightarrow Y}(\cdot)$ is called the {\em discretization
modulus} corresponding to $X, Y$. Ribe's theorem follows from the
assertion that if $\dim(X)<\infty$ then $\d_{X\hookrightarrow
Y}(\e)>0$ for all $\e\in (0,1)$. This implication follows from the
classical observation~\cite{CK63} that uniformly continuous mappings
on Banach spaces are bi-Lipschitz for large distances, and a $w^*$
differentiation argument of  Heinrich and Mankiewicz~\cite{HM82};
see~\cite{GNS11} for the details.

In~\cite{Bou87} Bourgain found a new proof of Ribe's theorem that
furnished an explicit bound on $\d_{X\hookrightarrow Y}(\cdot)$.
Specifically, if $\dim(X)=n$ then
\begin{equation}\label{eq:Bourgain's discretization bound}
\forall\, \e\in (0,1),\quad \d_{X\hookrightarrow Y}(\e)\ge e^{-(n/\e)^{Cn}},
\end{equation}
where $C\in (0,\infty)$ is a universal constant.
\eqref{eq:Bourgain's discretization bound} should be viewed as a
quantitative version of Ribe's theorem, and it yields an abstract
and generic way to obtain a  family of finite metric spaces that
serve as obstructions whose presence  characterizes the failure of
any given isomorphic finite dimensional linear property of Banach
spaces; see~\cite{Ost12-forms,Ost12-test}.

In light of the Ribe program it would be of great interest to
determine the asymptotic behavior in $n$ of, say,
$\d_{X\hookrightarrow Y}(1/2)$. However, the
bound~\eqref{eq:Bourgain's discretization bound} remains the best
known estimate, while the known (simple) upper bounds on
$\d_{X\hookrightarrow Y}(1/2)$ decay like a power of $n$;
see~\cite{GNS11}. This question is of interest even when $X,Y$ are
restricted to certain subclasses of Banach spaces, in which the
following improvement is known~\cite{GNS11}: for all $p\in
[1,\infty)$ we have $\d_{X\hookrightarrow {L_p(\mu)}}(1/2)\gtrsim
(\dim(X))^{-5/2}$ (the implied constant is universal). We refer
to~\cite{GNS11} for a more general statement along these lines, as
well as to~\cite{LN12,HLN12} for alternative approaches to this
question.

\section{Nonlinear Dvoretzky theorems}\label{sec:dvo}

 A classical theorem of Dvoretzky~\cite{Dvo60} asserts, in confirmation of a conjecture of Grothendieck~\cite{Gro53-dvo},  that for every $k\in
\N$ and $D>1$ there exists $n=n(k,D)\in \N$ such that every
$n$-dimensional normed space has a $k$-dimensional linear subspace
that embeds into Hilbert space with distortion $D$;
see~\cite{Mil71,MS99,Sch06} for the best known bounds on $n(k,D)$.
In accordance with the Ribe program, Bourgain, Figiel and Milman
asked in 1986 if there is an analogue of the Dvoretzky phenomenon
which holds for general metric spaces. Specifically, they
investigated the largest $m\in \N$ such that {\em any} finite metric
space $(\M,d_\M)$ of cardinality $n$ has a subset $S\subseteq \M$
with $|S|\ge m$ such that the metric space $(S,d_\M)$ embeds with
distortion $D$ into Hilbert space. Twenty years later,  Tao asked an
analogous question in terms of Hausdorff dimension: given $\alpha>0$
and $D>1$, what is the supremum over those $\beta\ge 0$ such that
every compact metric space $\M$ with $\dim_H(\M)\ge \alpha$ has a
subset $S\subseteq \M$ with $\dim_H(\M)\ge \beta$ that embeds into
Hilbert space with distortion $D$? Here $\dim_H(\cdot)$ denotes
Hausdorff dimension.

A pleasing aspect of the Ribe program is that sometimes we get more
than we asked for. In our case, we asked for almost Euclidean subsets, but
the known answers to the above questions actually provide subsets
that are even more structured: they are approximately ultrametric.
Before describing these answers to the above questions, we therefore
first discuss the structure of ultrametric spaces, since this
additional structure is crucial for a variety of applications.

\subsection{The structure of ultrametric spaces}\label{sec:ult strcut} Let $(\M,d_\M)$ be
an ultrametric space, i.e.,
\begin{equation}\label{eq:ultratriangle}
\forall\, x,y,z\in \M,\quad d_\M(x,y)\le \max\left\{d_\M(x,z),d_\M(y,z)\right\}.
\end{equation}
In the discussion below, assume for simplicity that $\M$ is finite:
this case contains all the essential ideas, and the natural
extensions to infinite ultrametric spaces can be found in
e.g.~\cite{Hug04,MN11-skel,KMZ12}. Define an equivalence relation
$\sim$ on $\M$ by
$$
\forall\, x,y\in \M,\quad x\sim y\iff d_\M(x,y)<\diam(\M)=\max_{z,w\in \M} d_\M(z,w).
$$
Observe that it is the ultra-triangle
inequality~\eqref{eq:ultratriangle} that makes $\sim$ be indeed an
equivalence relation. Let $A_1,\ldots,A_k$ be the corresponding
equivalence classes. Thus $d_\M(x,y)<\diam(\M)$ if $(x,y)\in
\bigcup_{i=1}^k A_i\times A_i$ and $d_\M(x,y)=\diam(\M)$ if
$(x,y)\in \M\setminus \bigcup_{i=1}^k A_i\times A_i$.

By applying this construction to each equivalence class separately,
and iterating, one obtains a sequence of partitions
$\P_0,\ldots,\P_n$ of $\M$ such that $\P_0=\{\M\}$,
$\P_n=\{\{x\}\}_{x\in \M}$, and $\P_{i+1}$ is a refinement of $\P_i$
for all $i\in \{0,\ldots,n-1\}$. Moreover, for every $x,y\in \M$, if
we let $i\in \{0,\ldots,n\}$ be the maximal index such that $x,y\in
A$ for some $A\in \P_i$, then $d_\M(x,y)=\diam(A)$. Alternatively,
consider the following graph-theoretical tree whose vertices are
labeled by subsets of $\M$. The root is labeled by $\M$ and the
$i$th level of the tree is in one-to-one correspondence with the
elements of the partition $\P_i$. The descendants of an $i$ level
vertex whose label is  $A\in \P_i$ are declared to be the $i+1$
level vertices whose labels are $\{B\in \P_{i+1}:\ B\subseteq A\}$.
With this combinatorial picture in mind, $\M$ can be identified as
the leaves of the tree and the metric on $\M$ has the following
simple description: the distance between any two leaves is the
diameter of the set corresponding to their least common ancestor in
the tree. This simple combinatorial structure of ultrametric spaces
will be harnessed extensively in the ensuing discussion.
See~\cite{Hug04,MN11-skel,KMZ12} for an extension of this picture to
infinite compact ultrametric spaces (in which case the points of
$\M$ are in one-to-one correspondence with the ends of an infinite
tree).

We record two more consequences of the above discussion. First of
all, by considering the natural lexicographical order that is
induced on the leaves of the tree, we obtain a linear order $\prec$
on $\M$ such that if $x,y\in \M$ satisfy $x\preceq y$ then
\begin{equation}\label{eq:linear order}
\diam([x,y])=\diam(\{z\in \M:\ x\preceq z\preceq y\})=d_\M(x,y).
\end{equation}
See~\cite{KMZ12} for a proof of the existence of a linear order
satisfying~\eqref{eq:linear order} for every compact ultrametric
space $(\M,d_\M)$, in which case the order interval $[x,y]$ is
always a Borel subsets of $\M$.

The second consequence that we wish to record here is that
$(\M,d_\M)$ admits an isometric embedding into the sphere of radius
$\diam(\M)/\sqrt{2}$ of Hilbert space. This is easily proved by
induction on $\M$ as follows. Letting $A_1,\ldots,A_k$ be the equivalence
classes as above, by the induction hypothesis there exist isometric
embeddings $f_i:A_i\to H_i$, where $H_1,\ldots,H_k$ are Hilbert
spaces and $\|f_i(x)\|_{H_i}=\diam(A_i)/\sqrt{2}$ for all $i\in
\{1,\ldots,k\}$. Now define
$$
f:\M\to \left(\bigoplus_{i=1}^kH_i\right)\oplus \ell_2^k\eqdef H
$$
by
$$
x\in A_i\implies f(x)=f_i(x)+\sqrt{\frac{\diam(\M)^2-\diam(A_i)^2}{2}}e_i,
$$
where $e_1,\ldots,e_k$ is the standard basis of
$\ell_2^k=(\R^k,\|\cdot\|_2)$. One deduces directly from this
definition, and the fact that $d_\M(x,y)=\diam(\M)$ if $(x,y)\in
\M\setminus \bigcup_{i=1}^k A_i\times A_i$, that
$\|f(x)\|_H=\diam(\M)/\sqrt{2}$ for all $x\in \M$ and
$\|f(x)-f(y)\|_H=d_\M(x,y)$ for all $x,y\in \M$. See~\cite{VT79} for more information on Hilbertian isometric embeddings of ultrametric spaces.

Thus, the reader should keep the following picture in mind when
considering a finite ultrametric space $(\M,d_\M)$: it corresponds
to the leaves of a tree that are isometrically embedded in Hilbert
space. Moreover, for every node of the tree the distinct
subtrees that are rooted at its children are, after translation,
mutually orthogonal.

\subsection{Ultrametric spaces are ubiquitous}\label{sec:skel}

The following theorem is equivalent to the main result
of~\cite{MN11-ultra}, the original formulation of which will not be
stated here; the formulation below is due to~\cite{MN11-skel}.

\begin{theorem}[Ultrametric skeleton theorem]\label{thm:skel}
For every $\e\in (0,1)$ there exists $c_\e\in [1,\infty)$ with the
following property. Let $(\M,d_\M)$ be a compact metric space and
let $\mu$ be a Borel probability measure on $\M$. Then there exists
a compact subset $S\subseteq \M$ and a Borel probability measure
$\nu$ that is supported on $S$, such that $(S,d_\M)$ embeds into an
ultrametric space with distortion at most $9/\e$ and
\begin{equation}\label{eq:x,r}
\forall (x,r)\in \M\times [0,\infty),\quad \nu\left(B(x,r)\cap S\right)\le \left(\mu\left(B(x,c_\e r)\right)\right)^{1-\e}.
\end{equation}
\end{theorem}

The subset $S\subseteq \M$ of Theorem~\ref{thm:skel} is called an
{\em ultrametric skeleton} of $\M$ since, as we shall see below and
is explained further in~\cite{MN11-skel}, it must be ``large" and
``spread out", and, more importantly, its main use is to deduce
global information about the initial metric space $(\M,d_\M)$.

By Theorem~\ref{thm:skel} we know that despite the fact that
ultrametric spaces have a very restricted structure, every metric measure
space has an ultrametric skeleton. We will now describe several
consequences of this fact. Additional examples of consequences of
Theorem~\ref{thm:skel}  are contained in
Sections~\ref{sec:majorizing}, \ref{sec:keleti}, \ref{sec:oracle}
below.

Our first order of business is to relate Theorem~\ref{thm:skel} to
the above nonlinear Dvoretzky problems. Theorem~\ref{thm:skel} was discovered in the context of
investigations on nonlinear Dvoretzky theory, and as such it
constitutes another example of a metric space phenomenon that was
uncovered due to the Ribe program.

\begin{theorem}\label{thm:finite dvo}
For every $\e\in (0,1)$ and $n\in \N$, any $n$-point metric space has a subset of
size at least $n^{1-\e}$ that embeds into an ultrametric space with
distortion $O(1/\e)$.
\end{theorem}

\begin{proof}
This is a simple corollary of the ultrametric skeleton theorem,
which does not use its full force. Specifically, right now we will
only care about the case $r=0$ in~\eqref{eq:x,r}, though later we
will need~\eqref{eq:x,r} in its entirety. So, let $(\M,d_\M)$ be an
$n$-point metric space and let $\mu$ be the uniform probability
measure on $\M$. An application of Theorem~\ref{thm:skel} to the
metric measure space $(\M,d_\M,\mu)$ yields an ultrametric skeleton
$(S,\nu)$. Thus $(S,d_\M)$ embeds into an ultrametric space with
distortion $O(1/\e)$. Since $\nu$ is a probability measure that is
supported on $S$, there must exist a point  $x\in S$ with
$\nu(\{x\})\ge 1/|S|$. By~\eqref{eq:x,r} (with $r=0)$ we have
$\nu(\{x\})\le \mu(\{x\})^{1-\e}=1/n^{1-\e}$. Thus $|S|\ge
n^{1-\e}$,
\end{proof}

Theorem~\ref{thm:finite dvo} was first proved in~\cite{MN07}, as a
culmination of the investigations
in~\cite{BFM86,KKR94,BKRS00,BBM,BLMN05}. The best known bound
for this problem is due to~\cite{NT10}, where it is shown that if
$\e\in (0,1)$ then any $n$-point metric space has a subset of size
$n^{1-\e}$ that embeds into an ultrametric space with distortion at
most
\begin{equation}\label{eq:NT distortion}
D(\e)=\frac{2}{\e(1-\e)^{\frac{1-\e}{\e}}}.
\end{equation}

Theorem~\ref{thm:finite dvo} belongs to the nonlinear Dvoretzky
framework of Bourgain, Figiel and Milman because we have seen that ultrametric spaces admit an isometric embedding into Hilbert space. Moreover, the following matching impossibility result was proved in~\cite{BLMN05}.

\begin{theorem}\label{thm:expander lower}
There exist  universal constants $K,\kappa\in (0,\infty)$ and for
every $n\in \N$ there exists an $n$-point metric space
$(\M_n,d_{\M_n})$ such that for every $\e\in (0,1)$ we have
$$
\forall\, S\subseteq \M_n, \quad |S|\ge Kn^{1-\e}\implies c_{\ell_2}(S,d_{\M_n})\ge \frac{\kappa}{\e}.
$$
\end{theorem}
In addition to showing that Theorem~\ref{thm:finite dvo}, and hence also Theorem~\ref{thm:skel}, is asymptotically sharp, Theorem~\ref{thm:expander lower} establishes that, in general, the best way (up to
constant factors) to find a large approximately Euclidean subset is to actually find a subset satisfying the more stringent requirement of being almost ultrametric.

Turning to the Hausdorff dimensional nonlinear Dvoretzky problem, we have the following consequence of the ultrametric skeleton theorem due to~\cite{MN11-ultra}.

\begin{theorem}\label{thm:hausdorff}
For every $\e\in (0,1)$ and $\alpha\in (0,\infty)$, any compact metric space of Hausdorff dimension greater than $\alpha$ has a closed subset of Hausdorff dimension greater than $(1-\e)\alpha$ that embeds into an ultrametric space with distortion $O(1/\e)$.
\end{theorem}

\begin{proof}
Let $(\M,d_\M)$ be a compact metric space with $\dim_H(\M)>\alpha$.
By the Frostman lemma (see~\cite{How95,Mattila}) it follows that
there exists a Borel probability measure $\mu$ on $\M$ and  $K\in
(0,\infty)$ such that
\begin{equation}\label{eq:frostman condition}
\forall (x,r)\in \M\times [0,\infty),\quad \mu(B(x,r))\le Kr^\alpha.
\end{equation}
An application of Theorem~\ref{thm:skel} to the metric measure space
$(\M,d_\M,\mu)$ yields an ultrametric skeleton $(S,\nu)$. If
$\{B(x_i,r_i)\}_{i=1}^\infty$ is a collection of balls that covers
$S$ then
\begin{multline*}
 1=\nu(S)=\nu\left(\bigcup_{i=1}^\infty
B(x_i,r_i)\right)\le
\sum_{i=1}^\infty\nu(B(x_i,r_i))\\\stackrel{\eqref{eq:x,r}}{\le}
\sum_{i=1}^\infty\mu\left(B(x_i,c_\e
r_i)\right)^{1-\e}\stackrel{\eqref{eq:frostman condition}}{\le}
K^{1-\e}c_\e^{(1-\e)\alpha}\sum_{i=1}^\infty r_i^{(1-\e)\alpha}.
\end{multline*}
Having obtained an absolute positive lower bound on
$\sum_{i=1}^\infty r_i^{(1-\e)\alpha}$ for all the covers of $S$ by
balls $\{B(x_i,r_i)\}_{i=1}^\infty$, we conclude the desired
dimension lower bound $\dim_H(S)\ge (1-\e)\alpha$.
\end{proof}

\begin{remark}\label{re:hausdorff sharpness}
{\em It is also proved in~\cite{MN11-ultra} that there is a
universal constant $\kappa\in (0,\infty)$ such that for every
$\alpha>0$ and $\e\in (0,1)$ there exists a compact metric space
$(\M,d_\M)$ with $\dim_H(\M)=\alpha$ such that
$$
\forall\, S\subseteq
\M,\quad  \dim_H(S)\ge (1-\e)\alpha \implies c_{\ell_2}(S,d_\M)\ge
\frac{\kappa}{\e}.
$$
 Therefore, as in the case of the nonlinear Dvoretzky
problem for finite metric spaces, the question of finding in a
general metric space a high-dimensional subset which is
approximately Euclidean is the same (up to constants) as the
question of finding a high-dimensional subset which is approximately
an ultrametric space. This phenomenon helps explain how
investigations that originated in Dvoretzky's theorem led to a
theorem such as~\ref{thm:skel} whose conclusion seems to be far from
its initial Banach space motivation: the Ribe program indicated a natural question to
ask, but the answer itself turned out to be a truly nonlinear
phenomenon involving subsets which are approximately ultrametric
spaces; a (perhaps unexpected) additional feature that is more
useful than just the extraction of approximately Euclidean
subsets.}
\end{remark}

\begin{remark}\label{rem:small distortion}
{\em As mentioned above, the best known distortion bound in
Theorem~\ref{thm:finite dvo} is given in~\eqref{eq:NT distortion}.
When $\e\to 1$ this bound tends to $2$ from above. Distortion $2$ is
indeed a barrier here: the nonlinear Dvoretzky problem exhibits a
phase transition at distortion $2$ between power-type and
logarithmic behavior of the largest Euclidean subset that can be
extracted in general metric spaces of cardinality $n$. This
phenomenon was discovered in~\cite{BLMN05}; see
also~\cite{BLMN05-dic,BLMN05-low,CK05} for related threshold
phenomena. In their original paper~\cite{BFM86} that introduced the
nonlinear Dvoretzky problem, Bourgain Figiel and Milman proved that
for every $D>1$ any $n$-point metric space has a subset of size at
least $c(D)\log n$ that embeds with distortion $D$ into Hilbert
space. They also proved that there exists constants $D_0=1.023...$,
$\kappa\in (0,\infty)$ and for every $n\in \N$ there exists an
$n$-point metric space $(\M_n,d_{\M_n})$ such that every $S\subseteq
\M_n$ with $|S|\ge \kappa \log n$ satisfies
$c_{\ell_2}(S,d_{\M_n})\ge D_0$. In~\cite{BLMN05} this impossibility
result was extended to any distortion in $(1,2)$, thus establishing
the above phase transition phenomenon. The asymptotic behavior of
the nonlinear Dvoretzky problem at distortion $D=2$ remains unknown.
For the Hausdorff dimensional version of this question the phase
transition at distortion $2$ becomes more extreme: for every $\d\in
(0,1/2)$ one can obtain~\cite{MN11-ultra} a version of
Theorem~\ref{thm:hausdorff} with the resulting subset $S$ having
ultrametric distortion $2+\d$ and $\dim_H(S)\gtrsim
\frac{\d}{\log(1/\d)}\alpha$. In contrast,  for every $\alpha\in
(0,\infty)$ there exists~\cite{MN11-ultra} a compact metric space
$(\M,d_\M)$ of Hausdorff dimension $\alpha$ such that if $S\subseteq
\M$ embeds into Hilbert space with distortion strictly smaller than
$2$ then $\dim_H(S)=0$.}
\end{remark}

\section{Examples of applications}\label{sec:apps}

Several applications of the Ribe program have already been discussed
throughout this article. In this section we describe some additional
 applications of this type. We purposefully chose examples of
applications to areas which are far from Banach space theory, as an
indication of the relevance of the Ribe program to a variety of
fields.

\subsection{Majorizing measures}\label{sec:majorizing}

A (centered) Gaussian process is a family of random variables
$\{G_x\}_{x\in X}$, where $X$ is an abstract index set and for every
$x_1,\ldots,x_n\in X$ and $s_1,\ldots,s_n\in \R$ the random variable
$\sum_{i=1}^ns_iG_{x_i}$ is a mean zero Gaussian random variable. To
avoid technicalities that will obscure the key geometric ideas we
will assume throughout the ensuing discussion that $X$ is finite.

Given a  centered Gaussian process $\{G_x\}_{x\in X}$, it is of
great interest to compute (or estimate up to constants) the quantity
$\E\left[\max_{x\in X} G_x\right]$. The process induces the metric
$d(x,y)=\sqrt{\E\left[(G_x-G_y)^2\right]}$ on $X$, and this metric
determines $\E\left[\max_{x\in X} G_x\right]$. Indeed,  if
$X=\{x_1,\ldots,x_n\}$ then consider the $n$ by $n$
 matrix $D=(d(x_i,x_j)^2)$ and observe that  $D$
 is negative semidefinite on the subspace $\left\{x\in \R^n:\
\sum_{i=1}^nx_i=0\right\}$ of $\R^n$. Then,
\begin{multline*}\label{eq:formula}
\E\left[\max_{i\in \{1,\ldots,n\}}
G_{x_i}\right]\\=\frac{1}{(2\pi)^{n/2}} \int_{\left\{x\in \R^n:\
\sum_{i=1}^nx_i=0\right\}} \left(\max_{i\in \{1,\ldots,n\}}
\left(\sqrt{-D}x\right)_i\right)e^{-\frac12\|x\|_2^2}dx.
\end{multline*}
More importantly, $\E\left[\max_{x\in X} G_x\right]$ is well-behaved
under bi-Lipschitz deformations of $(X,d)$: by the classical Slepian
lemma (see e.g.~\cite{Fer74,Tal87}), if $\{G_x\}_{x\in X}$ and
$\{H_x\}_{x\in X}$ are  Gaussian processes satisfying
$$\alpha \sqrt{\E\left[(G_x-G_y)^2\right]}\le
\sqrt{\E\left[(H_x-H_y)^2\right]}\le \beta
\sqrt{\E\left[(G_x-G_y)^2\right]}$$ for all $x,y\in X$, then
$$\alpha\E\left[\max_{x\in X} G_x\right]\le \E\left[\max_{x\in X}
H_x\right]\le \beta\E\left[\max_{x\in X} G_x\right].$$

These facts suggest that one could ``read" the value of
$\E\left[\max_{x\in X} G_x\right]$ (up to universal constant
factors) from the geometry of the metric space $(X,d)$. How to do
this explicitly has been a long standing mystery until Talagrand
proved~\cite{Tal87} in 1987 his celebrated {\em majorizing measure
theorem}, which solved this question and, based on his
investigations over the ensuing two decades, led to a systematic
geometric method to estimate $\E\left[\max_{x\in X} G_x\right]$,
with many important applications (see the
books~\cite{LT91,Tal05,Tal11} and the references therein). We will
now explain the majorizing measure theorem itself, and how it is a
consequence of the ultrametric skeleton theorem; this deduction is
due to~\cite{MN11-skel}.

For a finite metric space $(X,d)$ let $\Prob(X)$ denote the space of
all probability measures on $X$. Consider the quantity
$$
\gamma_2(X,d)=\inf_{\mu\in \Prob(X)}\sup_{x\in X} \int_0^\infty \sqrt{\log\left(\frac{1}{\mu(B(x,r))}\right)}dr.
$$
The parameter $\gamma_2(X,d)$ should be viewed as a Gaussian version
of a covering number. Indeed, the integral $\int_0^\infty
\sqrt{\log\left(1/\mu(B(x,r))\right)}dr$ is large if $\mu$ has a
small amount of mass near $x$, so $\gamma_2(X,d)$ measures the
extent to which one can spread unit mass over $X$ so that all the
points are ``close" to this mass distribution in the sense that
$\max_{x\in X} \int_0^\infty
\sqrt{\log\left(1/\mu(B(x,r))\right)}dr$ is as small as possible.

Fernique introduced $\gamma_2(X,d)$ in~\cite{Fer74}, where he proved
that every Gaussian process $\{G_x\}_{x\in X}$ satisfies
$\E\left[\sup_{x\in X}G_x\right]\lesssim \gamma_2(X,d)$. Under
additional assumptions, he also obtained a matching lower bound
$\E\left[\sup_{x\in X}G_x\right]\gtrsim \gamma_2(X,d)$. Notably,
Fernique proved in 1975 (see~\cite{Fer76} and
also~\cite[Thm.~1.2]{Fer78}) that if the metric
$d(x,y)=\sqrt{\E\left[(G_x-G_y)^2\right]}$ happens to be an
ultrametric then $\E\left[\sup_{x\in X}G_x\right]\asymp
\gamma_2(X,d)$. By the Slepian lemma, the same conclusion holds true
also if $(X,d)$ embeds with distortion $O(1)$ into an ultrametric
space.

It is simple to see how the ultrametric structure is relevant to
such probabilistic considerations: in Section~\ref{sec:ult strcut}
we explained that an ultrametric space can be represented as a
subset of Hilbert space corresponding to leaves of a tree in which
the subtrees rooted at a given vertex are mutually orthogonal. In
the setting of Gaussian processes orthogonality is equivalent to
(stochastic) independence, so the geometric assumption of
ultrametricity in fact has strong probabilistic ramifications.
Specifically, the problem reduces to the estimation of the expected
supremum of the following special type of Gaussian process, indexed
by leaves of a graph theoretical tree $T=(V,E)$: to each edge $e\in
E(T)$ we associated a mean zero Gaussian random variable $H_e$, the
variables $\{H_e\}_{e\in E}$ are independent, and for every leaf $x$
we have $G_x=\sum_{e\in E(P_x)} H_e$, where $P_x$ is the unique path
joining $x$ and the root of $T$. This additional independence that
the ultrametric structure provides allowed Fernique to directly prove that
$\E\left[\sup_{x\in X}G_x\right]\gtrsim \gamma_2(X,d)$.

Due in part to the above evidence, Fernique conjectured in 1974 that
$\E\left[\sup_{x\in X}G_x\right]\asymp \gamma_2(X,d)$ for every
Gaussian process $\{G_x\}_{x\in X}$.  Talagrand's majorizing measure
theorem~\cite{Tal87} is the positive resolution of this conjecture.
By Fernique's work as described above, this amounts to the assertion
that $\E\left[\sup_{x\in X}G_x\right]\gtrsim \gamma_2(X,d)$ for
every Gaussian process $\{G_x\}_{x\in X}$. Talagrand's strategy was
to show that there is $S\subseteq X$ that embeds into an ultrametric
space with distortion $O(1)$, and $\gamma_2(S,d)\gtrsim
\gamma_2(X,d)$. It would then follow from Fernique's original proof
of the majorizing measure theorem for ultrametric spaces that
$\E\left[\sup_{x\in S}G_x\right]\gtrsim \gamma_2(S,d)\gtrsim
\gamma_2(X,d)$. Since trivially $\E\left[\sup_{x\in X}G_x\right]\ge
\E\left[\sup_{x\in S}G_x\right]$, this strategy will indeed prove
the majorizing measures theorem.

Consider the following quantity
$$
\delta_2(X,d)=\sup_{\mu\in \Prob(X)}\inf_{x\in X} \int_0^\infty \sqrt{\log\left(\frac{1}{\mu(B(x,r))}\right)}dr.
$$
For the same reason that $\gamma_2(X,d)$ is in essence a Gaussian
covering number, $\delta_2(X,d)$  should be viewed as a Gaussian
version of a packing number. A short argument (see~\cite{MN11-skel})
shows that $\delta_2(X,d)\asymp \gamma_2(X,d)$ for every finite
metric space $(X,d)$.

Take $\mu\in \Prob(X)$ at which $\delta_2(X,d)$ is attained, i.e.,
for every $x\in X$ we have $\int_0^\infty
\sqrt{\log\left(1/\mu(B(x,r))\right)}dr\ge \delta_2(X,d)$ . An
application of the ultrametric skeleton theorem to the metric
measure space $(X,d,\mu)$ with, say, $\e=3/4$, yields an ultrametric
skeleton $(S,\nu)$. Thus $S\subseteq X$ embeds into an ultrametric
space with distortion $O(1)$ and $\nu\in \Prob(S)$ satisfies
$\nu(B(x,r))\le\sqrt[4]{\mu(B(x,Cr))}$ for all $x\in X$ and $r>0$,
where $C>0$ is a universal constant. It follows that for every $x\in
S$ the integral $\int_0^\infty
\sqrt{\log\left(1/\nu(B(x,r))\right)}dr$ is at least
$\frac12\int_0^\infty \sqrt{\log\left(1/\mu(B(x,Cr))\right)}dr$,
which by a change of variable equals $\frac{1}{2C}\int_0^\infty
\sqrt{\log\left(1/\mu(B(x,r))\right)}dr$. But  $\int_0^\infty
\sqrt{\log\left(1/\nu(B(x,r))\right)}dr\gtrsim \delta_2(X,d)$ by our
choice of $\mu$. By the definition of $\delta_2(S,d)$ we have
$\delta_2(S,d)\ge \int_0^\infty
\sqrt{\log\left(1/\nu(B(x,r))\right)}dr$, so $\delta_2(S,d)\gtrsim
\delta_2(X,d)$. Since $\delta_2(\cdot)\asymp\gamma_2(\cdot)$, the
proof is complete.

\begin{remark}
{\em The use of ultrametric constructions in metric spaces in order
to prove maximal inequalities is a powerful paradigm in analysis.
The original work of Fernique and Talagrand on majorizing measures
is a prime example of the success of such an approach, and methods
related to (parts of the proof of) the ultrametric skeleton theorem
have been used in the context of certain maximal inequalities
in~\cite{MN10-max,NT10-max}. Other notable examples of related ideas
include~\cite{Chr90,Bar98,NTV03,FRT04}.}
\end{remark}

\subsection{Lipschitz maps onto cubes}\label{sec:keleti}  Keleti, M\'ath\'e and
Zindulka~\cite{KMZ12} proved the following theorem using the
nonlinear Dvoretzky theorem for Hausdorff dimension
(Theorem~\ref{thm:hausdorff}), thus answering a question of
Urba\'nski~\cite{Urb09}.
\begin{theorem}\label{thm:KMZ}
Fix $n\in \N$ and let $(\M,d_\M)$ be a compact metric space of
Hausdorff dimension bigger than $n$. Then there exists a Lipschitz
mapping from $\M$ onto the cube $[0,1]^n$.
\end{theorem}

If, in addition to the assumptions of Theorem~\ref{thm:KMZ},
$(\M,d_\M)$ is an ultrametric space, then Theorem~\ref{thm:KMZ} is
proved as follows. By Frostman's lemma there exists a Borel
probability measure $\mu$ on $\M$ and $K\in (0,\infty)$ such that
$\mu(A)\le K(\diam(A))^n$ for all Borel $A\subseteq \M$. Moreover, as
explained in Section~\ref{sec:ult strcut}, there exists a linear
order $\prec$ on $\M$ satisfying~\eqref{eq:linear order}. Define
$\f:\M\to [0,1]$ by $\f(x)=\mu(\{y\in \M:\ y\prec x\})$. Then
$|\f(x)-\f(y)|\le Kd_\M(x,y)^n$ for all $x,y\in X$. Thus $\f$ is
continuous, and since $\mu$ is atom-free and $\M$ is compact, it
follows that $\f(\M)=[0,1]$. Letting $P$ be a $1/n$-H\"older Peano
curve from $[0,1]$ onto $[0,1]^n$ (see e.g.~\cite{Sag94}), the
mapping $f=P\circ \f$ has the desired properties.

To prove Theorem~\ref{thm:KMZ}, start with a general compact metric
space $(\M,d_\M)$ with $\dim_H(\M)>n$. By
Theorem~\ref{thm:hausdorff} there exists a compact subset
$S\subseteq \M$ with $\dim_H(S)>n$ that admits a bi-Lipschitz
embedding into an ultrametric space. By the above reasoning there
exists a Lipschitz mapping $f$ from $S$ onto $[0,1]^n$. We now
conclude the proof of Theorem~\ref{thm:KMZ} by extending $f$ to a
Lipschitz mapping $F:\M\to [0,1]$ (e.g. via the nonlinear
Hahn-Banach theorem~\cite[Lem.~1.1]{BL00}).

The above reasoning exemplifies the role of ultrametric skeletons:
$S$ was used as a tool, but the conclusion makes no mention of
ultrametric spaces. Moreover, $S$ itself admits an $n$-H\"older
mapping onto $[0,1]$, something which is impossible to do for
general $\M$. Only after composition with a Peano curve do we get a
Lipschitz mapping to which the nonlinear Hahn-Banach theorem
applies, allowing us to deduce a theorem about $\M$ with no mention
of the ultrametric skeleton $S$.

\subsection{Approximate distance oracles and approximate
ranking}\label{sec:oracle} Here we explain applications of nonlinear
Dvoretzky theory to computer science. By choosing to discuss only a
couple examples we are doing an injustice to the impact that the
Ribe program has had on theoretical computer science. We refer
to~\cite{LLR95,AR98,Lin02,Mat02,Nao10,WS11} for a more thorough (but
still partial) description of the role of ideas that are motivated
by the Ribe program in approximation algorithms. Even if we only
focus attention
 on nonlinear Dvoretzky theorems, the
full picture is omitted below: Theorem~\ref{thm:finite dvo} also
yields the best known lower bound~\cite{BBM,BLMN05} on the
competitive ratio of the randomized $k$-server problem; a central
question in the field of online algorithms.

An $n$-point metric space $(X,d_X)$ is completely determined by the
numbers $\{d_X(x,y)\}_{x,y\in X}$. One can therefore store
$\binom{n}{2}$ numbers, so that when one is asked the distance
between two points $x,y\in X$ it is possible to output the number
$d_X(x,y)$ in constant time\footnote{For the sake of the discussion
in this survey one should think of ``time" as the number of
locations in the data structure that are probed plus the number of
arithmetic operations that are performed. ``Size" refers to the
number of floating point numbers that are stored. The computational
model in which we will be working is the RAM model, although weaker
computational models such as the ``Unit cost floating-point word RAM
model" will suffice. See~\cite{H-PM06,MN07} for a discussion of
these computational issues. The preprocessing algorithms below are
 randomized, in which case ``preprocessing time"  refers to
``expected preprocessing time". All other algorithms are
deterministic.}. The {\em approximate distance oracle} problem asks
for a way to store $o(n^2)$ numbers so that given (a distance query)
$x,y\in X$ one can quickly output a number that is guaranteed to be
within a prescribed factor of the true distance $d_X(x,y)$. The
following theorem was proved in~\cite{MN07} as a consequence of the
nonlinear Dvoretzky theorem~\ref{thm:finite dvo}.

\begin{theorem}\label{thm:oracle} Fix $D>1$. Every $n$-point metric space $(\{1,\ldots,n\},d)$ can be preprocessed
in time $O\left(n^2\right)$ to yield a data structure of size
$O(n^{1+O(1/D)})$ so that given $i,j\in \{1,\ldots,n\}$ one can
output in $O(1)$ time a number $E(i,j)$ that is guarantied to
satisfy \begin{equation}\label{eq:stretch D} d(i,j)\le E(i,j)\le
Dd(i,j).
\end{equation}
\end{theorem}
Here, and in what follows, all the implied constants in the
$O(\cdot)$ notation are universal constants. The preprocessing time
of Theorem~\ref{thm:oracle} is due to Mendel and Schwob~\cite{MS09},
improving over the original preprocessing time of $O(n^{2+O(1/D)})$
that was obtained in~\cite{MN07}.

In their important paper~\cite{TZ05}, Thorup and Zwick constructed
approximate distance oracles as in Theorem~\ref{thm:oracle}, but
with query time $O(D)$. Their preprocessing time is $O(n^2)$, and
the size of their data structure is $O(Dn^{1+2(1+O(1/D))/D})$.  The
key feature of~\ref{thm:oracle} is that it yields constant query
time, i.e., a true oracle. In addition, the proof of
Theorem~\ref{thm:oracle} is via a new geometric method that we will
sketch below, based on nonlinear Dvoretzky theory.

Note that the exponent of $n$ in the size of the Thorup-Zwick oracle
is at most $1+2(1+o(1))/D$, while in Theorem~\ref{thm:oracle} it is
$1+C/D$ for some universal constant $C$ (which can be shown to be at
most $20$). This difference in constants can be important for
applications, but recently Wulff-Nilsen proved~\cite{Wul12} that one
can use the oracle of Theorem~\ref{thm:oracle} as a black box
(irrespective of the constant $C$) to construct an oracle of size
$O(n^{1+2(1+\e)/D})$ whose query time depends only on $\e$. The
significance of the constant $2$ here is that~\cite{TZ05}
establishes that it is sharp conditioned on the validity of a
positive solution to a certain well-known combinatorial open
question of Erd{\H{o}}s~\cite{Erd64}.

Sommer, Verbin and Yu~\cite{SVY09} have shown that
Theorem~\ref{thm:oracle} is sharp in the sense of the following
lower bound in the cell-probe model\footnote{See~\cite{Mil99} for
more information on the cell probe computational model. It suffice
to say here that it is a weak model, so cell probe lower bounds
should be viewed as strong impossibility results.}. Any data
structure that, given a query $i,j\in \{1,\ldots,n\}$, outputs in
time $t$ a number $E(i,j)$ satisfying~\eqref{eq:stretch D} must have
size at least $n^{1+c/(tD)}/\log n$. This lower bound works even
when the oracle's performance is measured only on metric spaces
corresponding to sparse graphs. The fact that the query time $t$ of
Theorem~\ref{thm:oracle} is a universal constant thus makes this
theorem asymptotically sharp. Nonlinear Dvoretzky theory is the only
currently known method that yields such sharp results.

It turns out that the proof of Theorem~\ref{thm:finite dvo}
in~\cite{MN07} furnishes a randomized polynomial time algorithm
that, given an $n$-point metric space $(X,d_X)$, outputs a subset
$S\subseteq X$ with $|S|\ge n^{1-\e}$ such that $(S,d_X)$ embeds
into an ultrametric space with distortion $O(1/\e)$. Moreover, we
can ensure that there exists an ultrametric $\rho$ on $X$ such that
for every $x\in X$ and $s\in S$ we have $d_X(x,s)\le \rho(x,s)\le
\frac{c}{\e}d_X(x,s)$, where $c\in (0,\infty)$ is a universal
constant. The latter statement follows from the following general
{\em ultrametric extension lemma}~\cite{MN07}, though the proof of
Theorem~\ref{thm:finite dvo} in~\cite{MN07} actually establishes
this fact directly without invoking Lemma~\ref{lem:ext ult} below
(this is important if one cares about constant factors).

\begin{lemma}[Extension lemma for approximate ultrametrics]\label{lem:ext ult}
Let $(X,d_X)$ be a finite metric space and fix $S\subseteq X$ and
$D\ge 1$. Suppose that that $\rho_0:S\times S\to [0,\infty)$ is an
ultrametric on $S$ satisfying $d_X(x,y)\le \rho_0(x,y)\le Dd_X(x,y)$
for all $x,y\in S$. Then there exists an ultrametric $\rho:X\times
X\to [0,\infty)$ such that $\rho(x,y)=\rho_0(x,y)$ if $x,y\in S$,
for every $x,y\in X$ we have $\rho(x,y)\ge d_X(x,y)/3$, and for
every $x\in X$ and $y\in S$ we have $\rho(x,y)\le 2Dd_X(x,y)$.
\end{lemma}

We are now in position to apply Theorem~\ref{thm:finite dvo} iteratively as follows. Set $S_0=\emptyset$ and let $S_1\subseteq X$ be the subset whose existence is stipulated in Theorem~\ref{thm:finite dvo}. Thus there exists an ultrametric $\rho_1$ on $X$ satisfying $d_X(x,y)\le \rho_1(x,y)\le \frac{c}{\e}d_X(x,y)$ for all $x\in X$ and $y\in S_1$. Apply the same procedure to $X\setminus S_1$, and continue inductively until the entire space $X$ is exhausted. We obtain a partition $\{S_1,\ldots,S_m\}$ of $X$ with the following properties holding for every $k\in \{1,\ldots,m\}$.
\begin{itemize}
\item $|S_{k}|\ge \left(n-\sum_{j=0}^{k-1}|S_j|\right)^{1-\e}$.
\item  There exists an ultrametric $\rho_k$ on $X\setminus \bigcup_{j=0}^{k-1}S_j$ satisfying $$d_X(x,y)\le \rho_k(x,t)\le \frac{c}{\e}d_X(x,y)$$ for all $x\in X\setminus \bigcup_{j=0}^{k-1}S_j$ and $y\in S_k$.
\end{itemize}

As we have seen in Section~\ref{sec:ult strcut}, for every $k\in
\{1,\ldots,m\}$ the ultrametric $\rho_k$ corresponds to a
combinatorial tree whose leaves are $X\setminus
\bigcup_{j=0}^{k-1}S_j$ and each vertex of which is labeled by a
nonegative number such that for $x,y\in X\setminus
\bigcup_{j=0}^{k-1}S_j$ the label of their least common ancestor is
exactly $\rho_k(x,y)$. A classical theorem of Harel and
Tarjan~\cite{HT84} (see also~\cite{BF-C00}) states that any
$N$-vertex tree can be preprocessed in time $O(N)$ so as to yield a
data structure of size $O(N)$ which, given two nodes as a query,
returns their least common ancestor in time $O(1)$. By applying the
Harel-Tarjan data structure to each of the trees corresponding to
$\rho_k$ we obtain an array of data structures (see
Figure~\ref{fig:oracles}) that can answer distance queries as
follows. Given distinct $x,y\in X$ let $k\in \{1,\ldots,m\}$ be the
minimal index for which $\{x,y\}\cap S_k\neq \emptyset$. Thus
$x,y\in X\setminus \bigcup_{j=0}^{k-1}S_j$, and, using the
Harel-Tarjan data structure corresponding to $\rho_k$, output in
$O(1)$ time the label of the least common ancestor of $x,y$ in the
tree corresponding to $\rho_k$. This output equals $\rho_k(x,y)$,
which, since $\{x,y\}\cap S_k\neq \emptyset$, satisfies $d_X(x,y)\le
\rho_k(x,t)\le \frac{c}{\e}d_X(x,y)$. Setting $D=c/\e$ and analyzing
the size of the data structure thus obtained (using the recursion
for the cardinality of $S_k$), yields Theorem~\ref{thm:oracle}; the
details of this computation can be found in~\cite{MN07}.

\begin{figure}[ht]
\centering \fbox{
\begin{minipage}{4.7in}
\centering \includegraphics[scale=0.7]{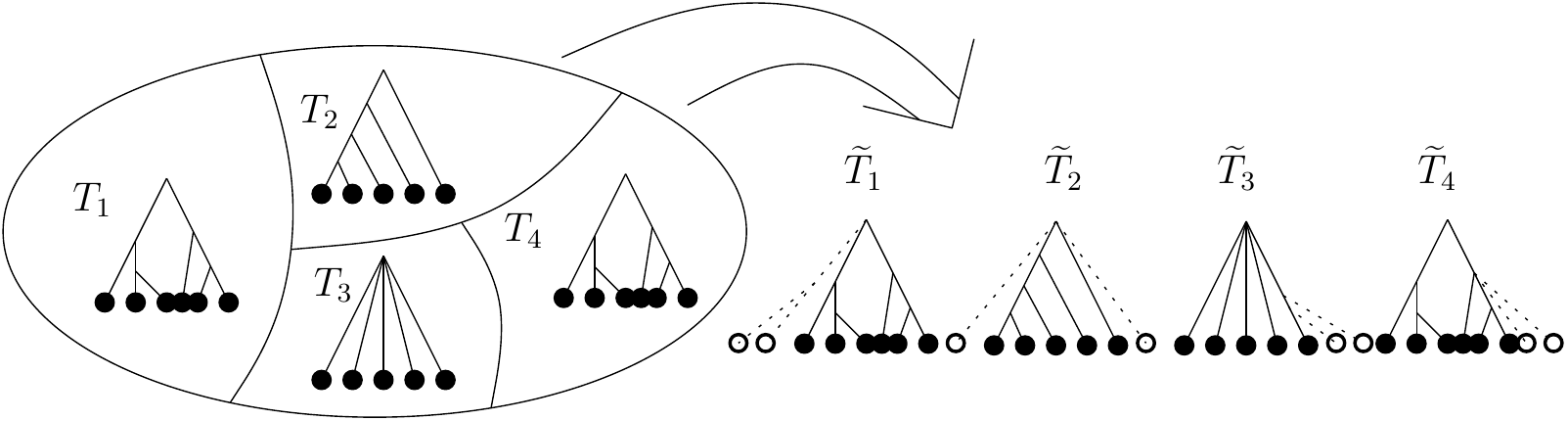}
 \caption{ In the approximate distance oracle problem an iterative application of Theorem~\ref{thm:finite dvo} yields an array of trees, which are then transformed into an array of Harel-Tarjan data structures. For the approximate ranking problem we also need to extend each tree to a tree whose leaves are the entire space $X$ using Lemma~\ref{lem:ext ult}. The nodes that were added to these trees
are illustrated by empty circles, and the dotted lines are their connections to the original tree.
 } \label{fig:oracles}
   \end{minipage}
}
\end{figure}

The ideas presented above are used in~\cite{MN07} to solve additional data structure problems. For example, we have the following theorem that addresses the {\em approximate ranking problem}, in which the goal is to compress the natural  ``$n$ proximity orders" (or ``rankings") induced on each of the points in an $n$-point metric space (i.e., each $x\in X$ orders the points of $X$ by increasing distance from itself).

\begin{theorem}\label{thm:ranking}
Fix $D>1$, $n\in \N$ and an $n$-point metric space $(X,d_X)$. Then there exists
a data structure which can be preprocessed in time $O\left(D
n^{2+O(1/D)}\log n\right)$, has size $O\left(D n^{1+O(1/D)}\right)$, and supports the following type of queries. Given
$x\in X$, have ``fast access" to a bijection $\pi^{(x)}:\{1,\ldots,n\}\to X$
satisfying
$$\forall\, 1\le i<j\le n,\quad d_X\left(x,\pi^{(x)}(i)\right)\leq  D
d_X\left(x,\pi^{(x)}(j)\right).$$ By ``fast access" to $\pi^{(x)}$ we
mean that we can do the following in $O(1)$ time:
\begin{compactenum}
\item Given $x\in X$ and $i\in \{1,\ldots,n\}$ output $\pi^{(x)}(i)$.
\item Given $x,u\in X$  output $j \in \{1,\ldots,n\}$ satisfying $\pi^{(x)}(j)=u$.
\end{compactenum}
\end{theorem}
The proof of Theorem~\ref{thm:ranking} follows the same procedure as above, with the following differences: at each stage we extend the ultrametric $\rho_k$ from $X\setminus \bigcup_{j=0}^{k-1}S_j$ to $X$ using Lemma~\ref{lem:ext ult}, and we replace the Harel-Tarjan data structure by a new data structure that is custom-made for the approximate ranking problem. The details are contained in~\cite{MN07}.

\subsection{Random walks and quantitative
nonembeddability}\label{sec:random walks}

While Ball introduced the notion of Markov type in order to
investigate the Lipschitz extension problem, this notion has proved
to be a versatile tool for the purpose of proving nonembeddability
results. The use of Markov type in the context of embedding problems
was introduced in~\cite{LMN02}, and this method has been
subsequently developed in~\cite{BLMN05,NPSS06,ANP09,NP08,NP11}.
Somewhat curiously, Markov type can also be used as a tool to prove
Lipschitz non-extendability results; see~\cite{NP11}. Markov type is
therefore a good example of the impact of ideas originating in the
Ribe program on metric geometry.

In this section we illustrate how one can use the notion of Markov
type to reason that certain metric spaces must be significantly
distorted in any embedding into certain Banach spaces. Since our
goal here is to explain in the simplest possible terms this way of
thinking about nonembeddability, we will mostly deal with model
problems, which might not necessarily be the most
general/difficult/important problems of this type. For example, we
will almost always state our results for embeddings into Hilbert
space, though it will be obvious how to extend our statements to
general target spaces with Markov type $p\in (1,\infty)$. Also, we
will present proofs in the case of finite graphs with large girth.
While these geometric objects are somewhat exotic, they serve as a
suitable model case for other spaces of interest, to which Markov
type techniques also apply (e.g., certain Cayley graphs, including
the discrete hypercube), since the large girth assumption simplifies
the arguments, while preserving the essential ideas. We stress,
however, that finite graphs with large girth are interesting
geometric objects in their own right. Their existence is established
with essentially complete freedom in the choice of certain governing
parameters (such as the girth and degree; see~\cite{Sac63}), yet
understanding their geometry is difficult: this is
illustrated by the fact that several basic problems on the
embeddability properties of such graphs remain open. We will present
some of these open problems later.

Fix an integer $k\ge 3$. Let $G=(V,E)$ be an $n$-vertex
$k$-regular connected graph, equipped with its associated
 shortest path metric $d_G$. Let $g$ be the girth of $G$, i.e., the length of the shortest closed cycle in $G$. Fix an integer
$r<\frac{g}{4}$.
 For any ball $B$ of radius $r$ in $G$, the metric space $(B,d_G)$ is isometric to
 $\left(T_r^k,d_{T_r^k}\right)$ (the tree $T_r^k$ is defined in Section~\ref{sec:Mconvexity}); see Figure~\ref{fig:girth}. Thus Bourgain's lower bound~\eqref{eq:bourgain tree} implies that
\begin{equation}\label{eq:girth bourgain}
c_{\ell_2}(G)\gtrsim \sqrt{\log g}.
\end{equation}

Can we do better than~\eqref{eq:girth bourgain}? It seems reasonable
to expect that we should be able to say more about the geometry of
$G$ than that it contains a large tree. When one tries to imagine
what does a finite graph with large girth look like, one quickly
realizes that it must be a complicated object: while it is true that
small enough balls in such a graph are trees, these local trees must
somehow be glued together to create a finite $k$-regular graph. It
seems natural to expect that the interaction between these local
trees induces a geometry which is far more complicated than what is
suggested by the lower bound~\eqref{eq:girth bourgain}. This
question was raised in 1995 by Linial, London and
Rabinovich~\cite{LLR95}. Our ultimate goal is to argue that {\em
all} large enough subsets of $G$ must be significantly distorted
when embedded into Hilbert space, but as a warmup we will start with
an argument of~\cite{LMN02} which shows how the fact that Hilbert
space has Markov type $2$ easily implies the following exponential
improvement to~\eqref{eq:girth bourgain}:
\begin{equation}\label{eq:LMN}
c_{\ell_2}(G)\gtrsim \sqrt{g}.
\end{equation}

To prove~\eqref{eq:LMN} we shall use the fact that $G$ has large
girth as follows: it isn't only the case that $G$ contains large
trees, in fact {\em every} small enough ball in $G$  is isometric to
a tree. This information can be harnessed to our advantage as
follows. Let $\{Z_t\}_{t=0}^\infty$ be the standard random walk on
$G$, i.e., $Z_0$ is uniformly distributed on $V$ and $Z_{t+1}$
conditioned on $Z_t$ is uniformly distributed on the $k$-neighbors
of $Z_t$. Then $\{Z_t\}_{t=0}^\infty$ is a stationary reversible Markov chain on
$V$. We claim that for every $t<\frac{g}{2}-1$ we have
\begin{equation}\label{eq:drift}
\E\left[d_G(Z_{t},Z_0)\right]\gtrsim t. \end{equation}

\begin{figure}[ht]
\centering \fbox{
\begin{minipage}{4.7in}
\centering
\includegraphics[height=80mm]{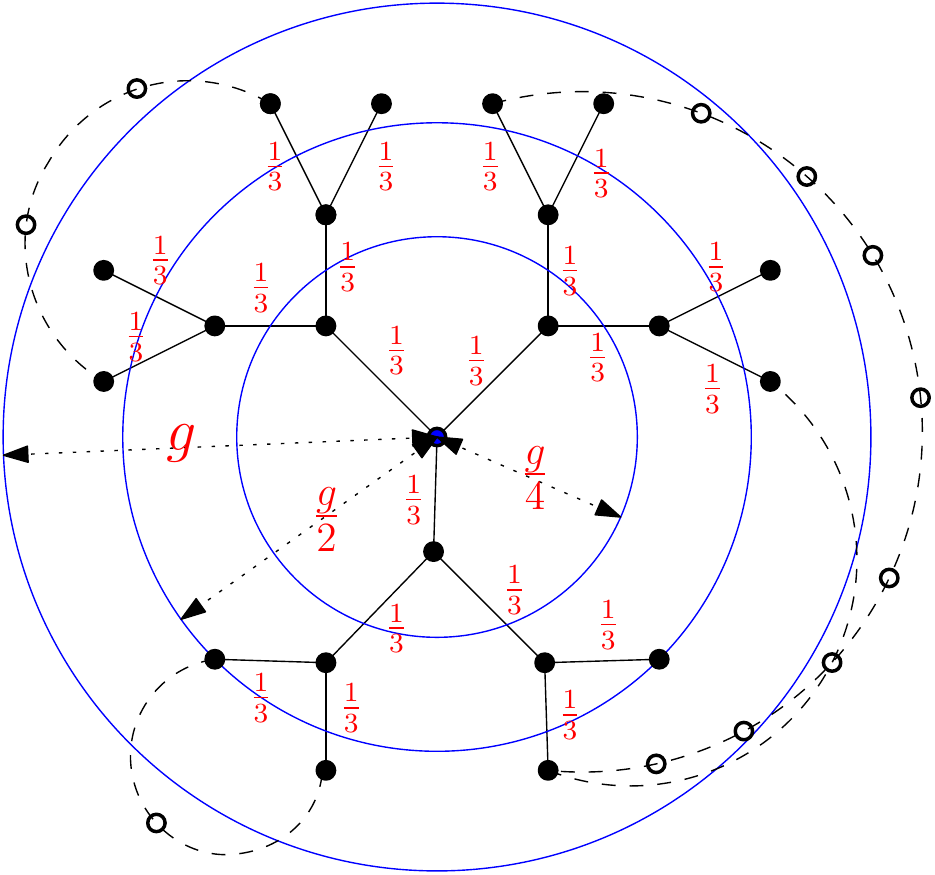}
 \caption{A $3$-regular graph with girth $g$. Balls of radius $\frac{g}{2}$ look like trees in the sense
 that the distance of any vertex to the center of the ball is the same as the
 corresponding distance in a $3$-regular tree rooted at the center
 of the ball. Any ball of radius $\frac{g}{4}$ is isometric to a $3$-regular tree. If we pick the center of the ball uniformly at random, and then perform a standard random walk,
 then up to time $<\frac{g}{2}$, at each step there is probability $\frac23$ to step further away from the center in
 the next step.} \label{fig:girth}
 \end{minipage}
}
\end{figure}

The proof of~\eqref{eq:drift} is simple. $Z_0$ is chosen uniformly
among the vertices of $G$. But, once $Z_0$ has been chosen, the walk
$\left\{Z_s:\ s<\frac{g}{2}-1\right\}$ is simply the standard walk
on a $k$-regular tree starting from its root. At each step of this
walk, if $Z_s\neq Z_0$ then with probability $1-\frac{1}{k}$ the
vertex $Z_{s+1}$ is one of the $k-1$ neighbors of $Z_s$ which are
further away from $Z_{0}$ than $Z_s$, and with probability
$\frac{1}{k}$ the vertex $Z_{s+1}$ is the unique neighbor of $Z_s$
that lies on the (unique) path joining $Z_s$ and $Z_0$. If it
happens to be the case that $Z_s=Z_0$, then $Z_{s+1}$ is further
away from $Z_0$ than $Z_s$ with probability $1$. Since
$1-\frac{1}{k}>\frac{1}{k}$, we see that even though $\left\{Z_s:\
s<\frac{g}{2}-1\right\}$ is a stationary reversible Markov chain, in
terms of the distance from $Z_0$ it is effectively a one dimensional
random walk with positive drift, implying the required lower
bound~\eqref{eq:drift}.

Suppose that $f:V\to L_2$ satisfies
\begin{equation}\label{eq:dist assumption}
\forall\, x,y\in V,\quad d_G(x,y)\le \|f(x)-f(y)\|_2\le D d_G(x,y).
\end{equation}
Our goal is to bound $D$ from below. The fact that Hilbert space has
Markov type $2$  implies that for all times $t<\frac{g}{2}-1$ we
have
\begin{multline}\label{eq;for contradiction}
t^2\stackrel{\eqref{eq:drift}}{\lesssim}
\left(\E\left[d_G(Z_{t},Z_0)\right]\right)^2\le\E\left[d_G(Z_t,Z_0)^2\right]\stackrel{\eqref{eq:dist
assumption}}{\le}
\E\left[\|f(Z_t)-f(Z_0)\|_2^2\right]\\\stackrel{\eqref{eq:def
Mtype}}{\le}
t\E\left[\|f(Z_1)-f(Z_0)\|_2^2\right]\stackrel{\eqref{eq:dist
assumption}}{\le} tD^2\E\left[d_G(Z_1,Z_0)^2\right]=tD^2.
\end{multline}
 Taking $t\asymp g$ in~\eqref{eq;for contradiction}
yields~\eqref{eq:LMN}.

The above argument can be extended to the case when $G$ is not
necessarily a regular graph. All we need is that the {\em average
degree} of $G$ is greater than $2$. Recall that the average degree
of $G$ is $$\frac{1}{|V|}\sum_{x\in V} \deg_G(x)=\frac{2|E|}{|V|},$$
where $\deg_G(x)$ denotes the number of edges in $E$ emanating from
$x$. Since we will soon be forced to deal with graphs of large girth
which are not necessarily regular, we record here the following
lemma from~\cite{BLMN05}:
\begin{lemma}\label{lem:not regular}
Let $G=(V,E)$ be a connected graph with girth $g$ and average degree
$k$. Then
\begin{equation}\label{eq:non regular}
c_{\ell_2}(G)\gtrsim \left(1-\frac{2}{k}\right)\sqrt{g}.
\end{equation}
\end{lemma}
The proof of Lemma~\ref{lem:not regular} follows the lines of the
above proof of~\eqref{eq:LMN}, with the following changes. For $x\in
V$ define
\begin{equation}\label{eq:def pi}
\pi(x)=\frac{\deg_G(x)}{\sum_{y\in V}
\deg_G(y)}=\frac{\deg_G(x)}{2|E|}.
\end{equation}
Now, let $\{Z_t\}_{t=0}^\infty$ be the standard random walk on $G$,
where $Z_0$ is distributed on $V$ according to the probability
distribution $\pi$. Then $\{Z_t\}_{t=0}^\infty$ is a stationary
reversible Markov chain on $V$, so that the Markov type $2$
inequality still applies to it. A short computation now
yields~\eqref{eq:non regular}; the details are contained in Theorem
6.1 of~\cite{BLMN05}.


This type of use of random walks is quite flexible. For example,
consider the case of the Hamming cube $(\{-1,1\}^n,\|\cdot\|_1)$.
Let $\{Z_t\}_{t=0}^\infty$ be the standard random walk on
$\{-1,1\}^n$, where $Z_0$ is distributed uniformly on $\{-1,1\}^n$.
At each step, one of the $n$ coordinates of $Z_t$ is chosen
uniformly at random, and its sign is flipped. For $t<\frac{n}{2}$ we
have $\E\left[\|Z_t-Z_0\|_1\right]\gtrsim t$, since at each step
with probability at least $\frac12$ the coordinate being flipped has
not been flipped in any previous step of the walk. As we have argued
above, this implies that $c_{\ell_2}\left(\{-1,1\}^n\right)\gtrsim
\sqrt{n}$. This lower bound is sharp up to the implied
multiplicative constant; in fact, a classical result of
Enflo~\cite{Enf69} states that
$c_{\ell_2}\left(\{-1,1\}^n\right)=\sqrt{n}$. Enflo's proof of this
fact uses a tensorization argument (i.e., induction on dimension
while relying on the product structure of the Hamming cube). Another
proof~\cite{KN06} of Enflo's theorem can be deduced from a Fourier
analytic argument (both known proofs of the equality
$c_{\ell_2}\left(\{-1,1\}^n\right)=\sqrt{n}$ are nicely explained in
the book~\cite{Mat02}). These proofs rely heavily on the structure
of the Hamming cube, while, as we shall see below, the random walk
proof that we presented here is more robust: e.g. it applies to
negligibly small subsets of the Hamming cube which may be highly
unstructured.

Before passing to a more sophisticated application of Markov type,
we recall the following interesting open question~\cite{LMN02}.

\begin{question}\label{Q:euclidean girth}
Let $c_2(g)$ be the infimum of $c_{\ell_2}(G)$ over all finite
$3$-regular connected graphs $G$ with girth $g$. What is the growth
rate of $c_2(g)$ as $g\to \infty$? In particular, does $c_2(g)$ grow
asymptotically faster than $\sqrt{g}$?
\end{question}
In order to prove that $\lim_{g\to \infty}c_2(g)/\sqrt{g}=\infty$
(if true), we would need to use more about the structure of $G$ than
the fact that a ball of radius $\asymp g$ around each vertex is
isometric to  a $3$-regular tree. One would need to understand the
complicated regime in which these local trees interact. Our
understanding of the geometry of these interactions is currently
quite poor, which is why Question~\ref{Q:euclidean girth} is
meaningful. On the other hand, if for arbitrarily large $g\in \N$
there were $3$-regular graphs $G$ of girth $g$ with
$c_{\ell_2}(G)\lesssim \sqrt{g}$, this would also have interesting
consequences, as explained in~\cite{LMN02}. Note that one could also
ask a variant of Question~\ref{Q:euclidean girth}, when $g$ depends
on the cardinality of $V$. The case $g\asymp \log |V|$ is of
particular importance (see~\cite{LMN02}).

Letting $c_1(g)$ denote the infimum of $c_{\ell_1}(G)$ over all
finite $3$-regular connected graphs $G$ with girth $g$, it was also
asked in~\cite{LMN02} whether or not $c_1(g)$ tends to $\infty$ with
$g$. This question was recently solved by Ostrovskii~\cite{Ost12},
who showed that for arbitrarily large $n\in \N$ there exists a
$3$-regular graph $G_n$ of girth at least a constant multiple of
$\log\log n$ yet $c_{\ell_1}(G_n)=O(1)$. Since trees admit an
isometric embedding into $\ell_1$, such questions address the issue
of how the local geometry of a metric space affects its global
geometry (see~\cite{ALNRRV06,CMM07,KS09,RS09} for related
investigations along these lines). It remains an interesting open
question whether there exist arbitrarily large graphs of logarithmic
girth that admit a bi-Lipschitz embedding into $\ell_1$;
see~\cite{LMN02} for ramifications of this question.

\subsubsection{Impossibility results for nonlinear Dvoretzky problems}\label{sec:ramsey}

Our goal here is to explain the relevance of Markov type techniques
to proving impossibility results for nonlinear Dvoretzky problems,
i.e., to show that certain metric spaces cannot have large subsets
that well-embed into Hilbert space (or a metric space with
nontrivial Markov type). Everything presented here is part of the
investigation in~\cite{BLMN05} of the nonlinear Dvoretzky problem in
concrete examples. Additional results of this type are contained
in~\cite{BLMN05}.

We have already seen that $c_{\ell_2}\left(\{-1,1\}^n\right)\gtrsim
\sqrt{n}$. Assume now that we are given a subset $S\subseteq
\{-1,1\}^n$.  If we only knew that the cardinality of $S$ is large,
would it then be possible to show that $c_{\ell_2}(S)$ is also
large? It is not clear how to proceed if $|S|=o(2^n)$ (this isn't
clear even when $|S|$ is, say, one tenth of the cube). The random
walk technique turns out to be robust enough to yield almost sharp
bounds on the Euclidean distortion of a large subset of the Hamming
cube, without any a priori assumption on the structure of the
subset. Namely, it was proved in~\cite{BLMN05} that for every
$S\subseteq \{-1,1\}^n$ we have
\begin{equation}\label{eq:ramsey cube}
c_{\ell_2}(S)\gtrsim \sqrt{\frac{n}{1+\log\left(\frac{2^n}{|S|}\right)}}.
\end{equation}
Thus, in particular, if
$|S|=2^{n(1-\e)}=\left|\{-1,1\}^n\right|^{1-\e}$,
then~\eqref{eq:ramsey cube} becomes
\begin{equation*}\label{eq:ramsey cube epsilon}
c_{\ell_2}(S)\gtrsim \min\left\{\frac{1}{\sqrt{\e}},\sqrt n\right\}.
\end{equation*}
This bound is tight up to logarithmic factors: it was shown
in~\cite{BLMN05} that for every $\e\in (0,1)$ there exists $S\subset
\{-1,1\}^n$ with $|S|\ge 2^{n(1-\e)}$ and
$$
c_{\ell_2}(S)\lesssim \sqrt{\frac{1+\log(1/\e)}{\e}}.
$$

The proof of~\eqref{eq:ramsey cube} uses Markov type in a crucial
way. Here, in order to illustrate the main ideas, we will deal with
the analogous problem for subsets of graphs with large girth.
Namely, let $G=(V,E)$ be a finite $k$-regular ($k\ge 3$) connected
graph with girth $g$. Assume that $S\subseteq V$ is equipped with
the metric $d_G$ inherited from $G$. We will prove the following
lower bound on $c_{\ell_2}(S)$, which is also due to~\cite{BLMN05}:
\begin{equation}\label{eq:ramsey girth}
c_{\ell_2}(S)\gtrsim \sqrt{\frac{g}{1+\log_k\left(\frac{|V|}{|S|}\right)}}.
\end{equation}

Note that when $S=V$ we return to~\eqref{eq:LMN}, but the proof
of~\eqref{eq:ramsey girth} is more subtle than the proof
of~\eqref{eq:LMN}. This proof uses more heavily the fact that
in~\eqref{eq:def Mtype} we are free to choose the stationary
reversible Markov chain as we wish. Our plan is to construct a
special stationary reversible Markov chain on $S$, which in
conjunction with the Markov type $2$ property of Hilbert space, will
establish~\eqref{eq:ramsey girth}.

Ideally, we would like our Markov chain to be something like the
standard random walk on $G$, restricted to $S$. Lemma~\ref{lem:not
regular} indicates that for this approach to work we need $S$ to
have large average degree, or equivalently to contain many edges of
$G$. But, $S$ might be very small, and need not contain any edge of
$G$. We will overcome this problem by considering a different set of
edges $E'$ on $V$, which is nevertheless closely related
to the geometry of $G$, such that $S$ contains sufficiently many
edges from $E'$. Before proceeding to carry out this plan,
we therefore need to make a small digression which explains a
spectral method for showing that a subset of a graph contains many
edges.

\subsubsection{$\lambda_n$ and self mixing}
Let $H=(\{1,\ldots,n\},E_H)$ be a $d$-regular loop-free graph on
$\{1,\ldots,n\}$. We denote by $A_H=(a_{ij})$ its adjacency matrix,
i.e., the $n\times n$ matrix whose entries are in $\{0,1\}$, and
$a_{ij}=1$ if and only if $ij\in E_H$. Let $\lambda_1(H)\ge
\lambda_2(H)\ge \cdots\ge \lambda_n(H)$ be the eigenvalues of $A_H$.
Thus $\lambda_1(H)=d$, and since the diagonal entries of $H$ vanish,
$\mathrm{trace}(A_H)=\sum_{i=1}^n \lambda_i(H)=0$. In particular we
are ensured that $\lambda_n(H)$ is negative.

Let $\{v_1,\ldots,v_n\}$ be an eigenbasis of $A_H$, which is
orthonormal with respect to the standard scalar product
$\langle\cdot,\cdot\rangle$ on $\R^n$. We can choose the labeling so
that $v_1=\frac{1}{\sqrt{n}}\1_{\{1,\ldots,n\}}$, and the eigenvalue
corresponding to $v_i$ is $\lambda_i(H)$. For every $S\subseteq
\{1,\ldots,n\}$ let $E_H(S)$ denote the number of edges in $E_H$
that are incident to two vertices in $S$. Observe that
\begin{multline*}
\langle A_H\1_S,\1_S\rangle=\sum_{i=1}^n\lambda_i(H)\langle
v_i,\1_S\rangle^2=\frac{d|S|^2}{n}+\sum_{i=2}^n \lambda_i(H)\langle
v_i,\1_S\rangle^2\\\ge \frac{d|S|^2}{n}+\lambda_n(H)\sum_{i=2}^n
\langle
v_i,\1_S\rangle^2=\frac{d|S|^2}{n}+\lambda_n(H)\left(|S|-\frac{|S|^2}{n}\right).
\end{multline*}
Thus, since $\lambda_n(H)<0$,
\begin{equation}\label{eq:self mixing}
2E_H(S)=\langle A_H\1_S,\1_S\rangle\ge \frac{d|S|^2}{n}+\lambda_n(H)|S|.
\end{equation}

We can use~\eqref{eq:self mixing} to deduce that $E_H(S)$ is large
provided that $\lambda_n(H)$ is not too negative (in~\cite{BLMN05}
such a bound is called a {\em self mixing inequality}). The bound
in~\eqref{eq:self mixing} is perhaps less familiar than Cheeger's
inequality~\cite{Che70,AM85}, which relates the number of edges
joining $S$ and its complement to $\lambda_2(H)$, but these two
inequalities are the same in spirit. We refer to the
survey~\cite{HLW06} for more information on the connection between
the second largest eigenvalue and graph expansion. While bounds on
$\lambda_2(H)$ would have been very useful for us to have in the
ensuing argument to prove~\eqref{eq:ramsey girth} (and the
corresponding proof of~\eqref{eq:ramsey cube} in~\cite{BLMN05}), we
will only obtain bounds on $|\lambda_n(H)|$ (for an appropriately
chosen graph $H$), which will nevertheless suffice for our purposes.

\subsubsection{The spectral argument in the case of large girth}

Returning to the proof of~\eqref{eq:self mixing}, let $G=(V,E)$ be
an $n$-vertex $k$-regular connected graph ($k\ge 3$) with girth $g$.
We assume throughout that $G$ is loop-free and contains no multiple
edges. As before, the shortest path metric on $G$ is denoted by
$d_G$. Fix $m\in \N$ and let $G^{(m)}=(V,E_{G^{(m)}})$ denote the
distance $m$ graph of $G$, i.e., the graph on $V$ in which two
vertices $u,v\in V$ are joined by an edge if and only if
$d_G(u,v)=m$.

Recall that $A_{G^{(m)}}$ denotes the adjacency matrix of $G^{(m)}$.
Thus we have $A_{G^{(0)}}=I_V$ (the identity matrix on $V$) and
$A_{G^{(1)}}=A_G$. Moreover, $A_G^2=kI_V+A_{G^{(2)}}$, and
$A_GA_{G^{(m-1)}}=(k-1)A_{G^{(m-2)}}+A_{G^{(m)}}$ for all
$2<m<\frac{g}{2}$. Indeed, write
$\left(A_GA_{G^{(m-1)}}\right)_{uv}=\sum_{w\in V}
\left(A_G\right)_{uw}\left(A_{G^{(m-1)}}\right)_{wv}$ for all
$u,v\in V$. There are only two types of possible contributions to
this sum: either $d_{G}(u,v)=m$ and $w$ is on the unique path
joining $u$ and $v$ such that $uw\in E$, or $d_{G}(u,v)=m-2$ and $w$
is one of the neighbors of $v$ which is not on the path joining $u$
and $v$ (the number of such $w$ equals $k$ if $m=2$, and equals
$k-1$ if $m>2$).

The above discussion shows that if we define a sequence of
polynomials $\{P_m^k(x)\}_{m=0}^\infty$ by
\begin{equation}\label{eq:start recursion} P_0^k(x)=1,\quad
P_1^k(x)=x,\quad P_2^k(x)=x^2-k,
\end{equation}
and recursively,
\begin{equation}\label{eq:geronimus}
P_m^k(x)=xP_{m-1}^k(x)-(k-1)P_{m-2}^k(x),
\end{equation}
then for all integers $0\le m<\frac{g}{2}$,
\begin{equation}\label{eq:geon G}
A_{G^{(m)}}=P_m^k\left(A_G\right).
\end{equation}

The polynomials $\{P_m^k(x)\}_{m=0}^\infty$ are known as the
Geronimus polynomials (see~\cite{Sol92} and the references therein).
By~\eqref{eq:geon G}, when $m<\frac{g}{2}$ the eigenvalues of
$A_{G^{(m)}}$ are
$\left\{P_m^k\left(\lambda_i(A_G)\right)\right\}_{i=1}^n$. For the
purpose of bounding the negative number
$\lambda_n\left(A_{G^{(m)}}\right)$ from below, it therefore
suffices to use the bound
\begin{equation}\label{eq:negative part}
\lambda_n\left(A_{G^{(m)}}\right)\ge \min_{x\in \R  }P_m^k(x).
\end{equation}

A simple induction shows that $P_m^k(x)$ is a polynomial of  degree
$m$ with leading coefficient $1$, and it is an even function for
even $m$, and an odd function for odd $m$. Moreover, we have the
following trigonometric identity (see~\cite{Sol92}):
\begin{multline}\label{eq:trigonometry}
P_m^k\left(2\sqrt{k-1}\cos\theta\right)\\=
(k-1)^{\frac{m}{2}-1}\cdot\frac{(k-1)\sin((m+1)\theta)-\sin((m-1)\theta)}{\sin
\theta}.
\end{multline}
The proof of~\eqref{eq:trigonometry} is a straightforward induction:
check the validity of~\eqref{eq:trigonometry} for $m=1,2$
using~\eqref{eq:start recursion}, and verify by induction
that~\eqref{eq:trigonometry} holds using the
recursion~\eqref{eq:geronimus}.

Define $\theta_q=\frac{\frac{\pi}{2}+q\pi}{m+1}$. For every $q\in
\{0,\ldots,m\}$ we have $\theta_q\in (0,\pi)$, and the sign of
$P_m^k\left(2\sqrt{k-1}\cos\theta_q\right)$ is equal to the sign of
\begin{multline*}
(k-1)\sin((m+1)\theta_q)-\sin((m-1)\theta_q)\\=(-1)^q(k-1)-
\sin\left(\frac{m-1}{m+1}\left(\frac{\pi}{2}+q\pi\right)\right).
\end{multline*}
Thus for every $q\in \{0,\ldots,m\}$ the value
$P_m^k\left(2\sqrt{k-1}\cos\theta_q\right)$ is positive if $q$ is
even, and negative if $q$ is odd. It follows that $P_m^k$ must have
a zero in each of the $m$ intervals
$\left\{\left[2\sqrt{k-1}\cos\theta_q,2\sqrt{k-1}\cos\theta_{q+1}\right]\right\}_{q=0}^{m-1}$.
Since $P_m^k$ is a polynomial of degree $m$, we deduce that the
zeros of $P_m^k$ are contained in the interval
$\left[-2\sqrt{k-1},2\sqrt{k-1}\right]$. In particular, if $m$ is
even, since $P_m^k(x)$ is an even function which tends to $\infty$
as $x\to \infty$, it can take negative values only in the interval
$\left[-2\sqrt{k-1},2\sqrt{k-1}\right]$. See
Figure~\ref{fig:geronimus} for the case $k=3$, $m=8$.

\begin{figure}[ht]
\centering \fbox{
\begin{minipage}{4.7in}
\centering
\includegraphics[scale=0.47]{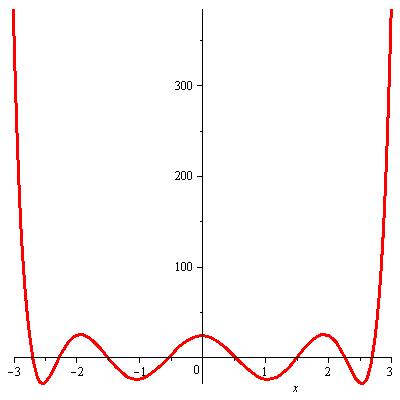}
 \caption{A plot of the polynomial $$P_8^3(x)=x^8-15x^6+70x^4-104x^2+24.$$} \label{fig:geronimus}
  \end{minipage}
}
\end{figure}

It follows from the above discussion, combined with~\eqref{eq:negative part} and~\eqref{eq:trigonometry}, that for every even integer $0<m<\frac{g}{2}$ we have

\begin{eqnarray}\label{eq:lambda_m bound}
&&\!\!\!\!\!\!\!\!\!\!\nonumber\lambda_n\left(A_{G^{(m)}}\right)\\&\ge&\nonumber (k-1)^{\frac{m}{2}-1}\min_{\theta\in [-\pi,\pi]}\frac{(k-1)\sin((m+1)\theta)-\sin((m-1)\theta)}{\sin \theta}\\ \nonumber &=& (k-1)^{\frac{m}{2}-1}\min_{\theta\in [-\pi,\pi]}\left((k-1)e^{-m\theta i}\sum_{r=0}^m e^{2\theta r i}-e^{-(m-2)\theta i} \sum_{r=0}^{m-2}e^{2\theta r i}\right)\\ \nonumber
&\ge& -(k-1)^{\frac{m}{2}-1}\left((k-1)(m+1)+m-1\right)\\&\ge& -(k-1)^{\frac{m}{2}-1} k(m+1).
\end{eqnarray}
Since the degree of $G^{(m)}$ is $k(k-1)^{m-1}$, the following corollary is a combination of~\eqref{eq:self mixing} and~\eqref{eq:lambda_m bound}.

\begin{corollary}\label{coro:self mixing}
Let $G=(V,E)$ be a $k$-regular graph with girth $g$. Then for all
even integers $0< m<\frac{g}{2}$ and for all $S\subseteq V$, the
average degree in the graph induced by $G^{(m)}$ on $S$ satisfies
$$
\frac{2E_{G^{(m)}}(S)}{|S|}\ge \frac{|S|}{n}k(k-1)^{m-1}-(k-1)^{\frac{m}{2}-1} k(m+1).
$$
In particular, if
\begin{equation}\label{eq:S condiiton}
\frac{|S|}{n}\ge \frac{2m+2}{(k-1)^{\frac{m}{2}}},
\end{equation}
then
\begin{equation}\label{eq:easier to read}
\frac{2E_{G^{(m)}}(S)}{|S|}\ge k(k-1)^{m-1}\frac{|S|}{2n}.
\end{equation}
\end{corollary}

\subsubsection{Completion of the proof of~\eqref{eq:ramsey
girth}}\label{sec:m-jump}

Corollary~\ref{coro:self mixing}, in combination with
Lemma~\ref{lem:not regular}, suggests that we should consider the
stationary reversible random walk on the graph induced by $G^{(m)}$
on $S$. We will indeed do so, and by judiciously choosing $m$,
\eqref{eq:ramsey girth} will follow.

For each $v\in S$ we denote by $\deg_{G^{(m)}[S]}(v)$ its degree in
the graph induced by $G^{(m)}$ on $S$, i.e., the number of vertices
$u\in S$ that are at distance $m$ from $v$, where the distance is
measured according to the original shortest path metric on $G$. As
in~\eqref{eq:def pi}, for $v\in S$ we write
\begin{equation}\label{eq:def pi girth}
\pi(v)=\frac{\deg_{G^{(m)}[S]}(v)}{2E_{G^{(m)}}(S)}.
\end{equation}
Let
$\{Z_t\}_{t=0}^\infty$ be the following Markov chain on $S$: $Z_0$
is distributed according to $\pi$, and $Z_{t+1}$ is distributed
uniformly on the $\deg_{G^{(m)}[S]}(Z_t)$ vertices of $S$ at
distance $m$ from $Z_t$ (note that $\deg_{G^{(m)}[S]}(Z_t)>0$, since
$Z_t$ is distributed only on those $v\in S$ for which $\pi(v)>0$).

At time $t\in \N$ we clearly have $d_G(Z_0,Z_t)\le tm$. In order to remain in the local ``tree range",
we will therefore impose the assumption
\begin{equation}\label{eq:tree range assumption}
tm<\frac{g}{4}.
\end{equation}
Assume from now on that $m$ is divisible by $6$. We first observe that for $t$ as
in~\eqref{eq:tree range assumption}, the number of neighbors $w\in V$ of $Z_{t-1}$
in the graph $G^{(m)}$ which satisfy $d_G(w,Z_{0})<d_G(Z_0,Z_{t-1})+\frac{m}{3}$
is at most $(k-1)^{\frac{2m}{3}-1}$. Indeed, we may assume that $d_G(Z_0,Z_{t-1})>\frac{m}{3}$,
 since otherwise for any such $w$ we have
\begin{multline*}
d_G(w,Z_0)\ge
d_G(w,Z_{t-1})-d_G(Z_0,Z_{t-1})\\=m-d_G(Z_0,Z_{t-1})\ge
d_G(Z_0,Z_{t-1})+\frac{m}{3}.
\end{multline*}
 So, assuming
$d_G(Z_0,Z_{t-1})>\frac{m}{3}$ and
$d_G(w,Z_{0})<d_G(Z_0,Z_{t-1})+\frac{m}{3}$, let $v$ be the point on
the unique path joining $Z_0$ and $Z_{t-1}$ such that
$d_G(v,Z_{t-1})=\frac{m}{3}+1$. The path in $G$ (whose length is
$m$) joining $w$ and $Z_{t-1}$ must pass through $v$. See
Figure~\ref{fig:fork} for an explanation of this simple fact. Note
that $d_G(w,v)=\frac{2m}{3}-1$, and hence the number of such $w$ is
at most $(k-1)^{\frac{2m}{3}-1}$.
\begin{figure}[ht]
\centering \fbox{
\begin{minipage}{4.7in}
\centering
\includegraphics[scale=0.57]{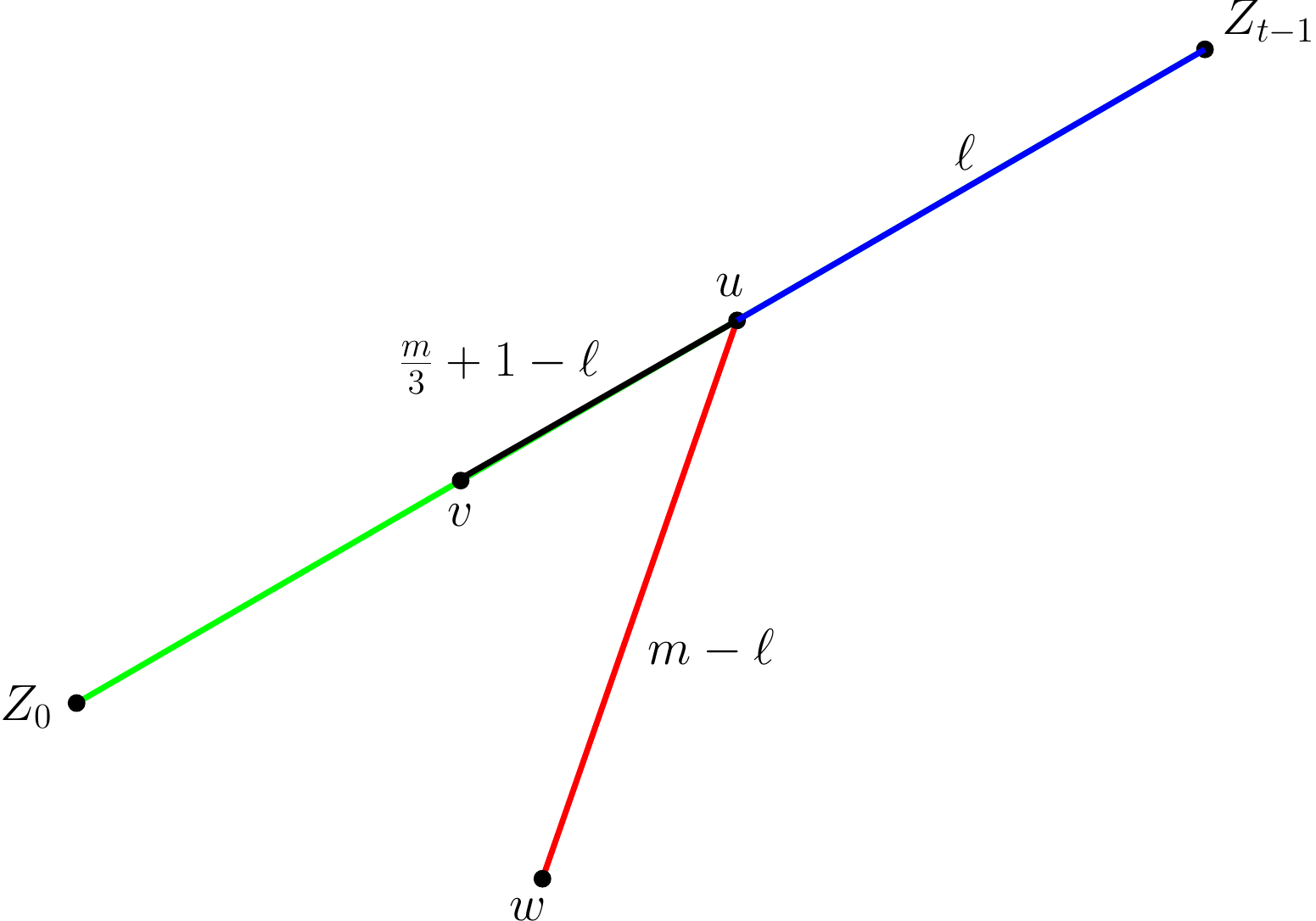}
 \caption{If the path joining $w$ and $Z_{t-1}$ does not pass
 through $v$ then it must touch the path joining $Z_0$ and $Z_{t-1}$ at a vertex $u$,
 as depicted above. Denoting $\ell= d_G(u,Z_{t-1})$, we have $\ell\le \frac{m}{3}$.
 Hence, since $d_{G}(w,Z_0)=\left(d_G(Z_0,Z_{t-1})-\ell\right)+(m-\ell)$, we have $ d_{G}(w,Z_0)\ge d_G(Z_0,Z_{t-1})+\frac{m}{3}.$}  \label{fig:fork}
 \end{minipage}
}
\end{figure}

Let $N(Z_{t-1})$ denote the number of $w\in S$ with
$d_{G}(w,Z_{t-1})=m$ and
$d_G(w,Z_{0})<d_G(Z_0,Z_{t-1})+\frac{m}{3}$. Then,
\begin{eqnarray}\label{eq:N}
\nonumber&&\!\!\!\!\!\!\!\!\!\!\!\!\!\!\!\!\!\!\!\!\!\!\!\E\left[d_G(Z_0,Z_{t})\right]\\&\ge&\nonumber
\E\left[\frac{\deg_{G^{(m)}[S]}(Z_{t-1})-N(Z_{t-1})}{\deg_{G^{(m)}[S]}(Z_{t-1})}
\left(d_G(Z_0,Z_{t-1})+\frac{m}{3}\right)\right.\\&\phantom{\le}&\quad\left.+
\frac{N(Z_{t-1})}{\deg_{G^{(m)}[S]}(Z_{t-1})}
\left(d_G(Z_0,Z_{t-1})-m\right)\right]\nonumber\\
&=& \E\left[d_G(Z_0,Z_{t-1})\right]+\frac{m}{3}-\frac{4m}{3}
\E\left[\frac{N(Z_{t-1})}{\deg_{G^{(m)}[S]}(Z_{t-1})}\right].
\end{eqnarray}
We  will estimate the last term appearing in~\eqref{eq:N} via the
point-wise bound $N(Z_{t-1})\le  (k-1)^{\frac{2m}{3}-1}$ that we
proved above, together with~\eqref{eq:easier to read}, for which we
need to assume~\eqref{eq:S condiiton}.
\begin{multline}\label{eq:use stationary}
\E\left[\frac{N(Z_{t-1})}{\deg_{G^{(m)}[S]}(Z_{t-1})}\right]\le
(k-1)^{\frac{2m}{3}-1}\sum_{\substack{v\in S\\
\deg_{G^{(m)}[S]}(v)>0}}\frac{\pi(v)}{\deg_{G^{(m)}[S]}(v)}\\\stackrel{\eqref{eq:def pi girth}}{\le}(k-1)^{\frac{2m}{3}-1}\frac{|S|}{2E_{G^{(m)}}(S)}
\stackrel{\eqref{eq:easier to
read}}{\le}\frac{1}{k(k-1)^{\frac{m}{3}}}\cdot\frac{2n}{|S|}.
\end{multline}
By combining~\eqref{eq:N} and~\eqref{eq:use stationary} we get the bound
\begin{multline}\label{eq:increment}
\E\left[d_G(Z_0,Z_{t})\right]\ge
\E\left[d_G(Z_0,Z_{t-1})\right]+\frac{m}{3}-\frac{8mn}{3k(k-1)^{\frac{m}{3}}|S|}\\\ge
\E\left[d_G(Z_0,Z_{t-1})\right]+\frac{m}{6},
\end{multline}
provided that
\begin{equation}\label{eq:second condition on m}
\frac{|S|}{n}\ge \frac{16}{k(k-1)^{\frac{m}{3}}}.
\end{equation}
We can ensure that our restrictions on $m$, namely~\eqref{eq:S
condiiton} and~\eqref{eq:second condition on m}, are satisfied for
some  $m\asymp 1+\log_k\left(n/|S|\right)$ that is divisible by $6$.
For such a value of $m$, we know that~\eqref{eq:increment} is valid
as long as $t$ satisfies~\eqref{eq:tree range assumption}. Thus, by
iterating~\eqref{eq:increment} we see that for some $t\asymp g/m$ we
have
\begin{equation}\label{eq:lower m step}
\E\left[d_G(Z_0,Z_{t})^2\right]\ge \left(\E\left[d_G(Z_0,Z_{t})\right]\right)^2\gtrsim (tm)^2\gtrsim g^2.
\end{equation}
If $f:S\to \ell_2$ satisfies
\begin{equation}\label{eq:ramsey bi lip assumption}
d_G(x,y)\le \|f(x)-f(y)\|_2\le Dd_G(x,y)
\end{equation}
for all $x,y\in S$, then it follows from the Markov type $2$
property of Hilbert space that
\begin{multline*}
g^2\stackrel{\eqref{eq:lower m step}}{\lesssim}
\E\left[d_G(Z_0,Z_{t})^2\right] \stackrel{\eqref{eq:ramsey bi lip
assumption}\wedge \eqref{eq:def Mtype}}{\le}
t\E\left[\|f(Z_1)-f(Z_0)\|_2^2\right]\\\stackrel{\eqref{eq:ramsey bi
lip assumption}}{\le}
tD^2\E\left[d_G(Z_0,Z_{1})^2\right]=D^2tm^2\asymp
D^2g\left(1+\log_k\left(\frac{n}{|S|}\right)\right).
\end{multline*}
This completes the proof of~\eqref{eq:ramsey girth}.\qed


\subsubsection{Discrete groups}\label{sec:compression}

Let $G$ be an infinite group which is generated by a finite
symmetric subset $S=S^{-1}\subseteq G$. Let $d_S$ denote the left
invariant word metric induced by $S$ on $G$, i.e., $d_S(x,y)$ is the
smallest integer $k\ge 0$ such that there exist $s_1,\ldots,s_k\in S$
with $x^{-1}y=s_1 s_2\cdots s_k$. It has long been established that
it is fruitful to study finitely generated groups as geometric
objects, i.e., as metric spaces when equipped with a word metric
(see~\cite{Gro93,delaHarpe00,CCJJV01} and the references therein for
an indication of the large amount of literature on this topic). Here
we will describe the role of Markov type in this context.

Assume that the metric space $(G,d_S)$ does not admit a bi-Lipschitz
embedding into Hilbert space, i.e., $c_{\ell_2}(G,d_S)=\infty$.
Based on the experience of researchers thus far, this assumption is
not restrictive: it is conjectured in~\cite{dCTV07} that if
$(G,d_S)$ does admit a bi-Lipschitz embedding into Hilbert space
then $G$ has an Abelian subgroup of finite index. Fix a mapping
$f:G\to \ell_2$. Note that if $f$ is not a Lipschitz function then
the mapping $x\mapsto \max_{s\in S}\|f(xs)-f(x)\|_2$ must be
unbounded on $G$. If we consider only mappings $f$ which have
bounded displacement on edges of the Cayley graph induced by $S$ on
$G$, then the fact that $c_{\ell_2}(G,d_S)=\infty$ must mean that if
we set $\omega_f(x)=\inf_{d_S(x,y)\ge t}\|f(x)-f(y)\|_2$ then
$\omega_f(t)=o(t)$ as $t\to \infty$. To see this consider the
mapping $\psi:G\to \ell_2\oplus \ell_2(G)\cong \ell_2$ given by
$\psi(x)=f(x)\oplus\delta_x$. The fact $\psi$ has infinite
distortion implies that $f$ must asymptotically compress arbitrarily
large distances in $G$.

The modulus $\omega_f(t)$ is called the compression function of $f$.
If we manage to show that for any $f:G\to \ell_2$ the rate at which
$\omega_f(t)/t$ tends to zero must be ``fast", then we might deduce
valuable structural information on the group $G$. This general
approach (including the terminology that we are using) is due to
Gromov (see Section 7.E in~\cite{Gro93}). Here we will study a
further refinement of this idea, which will yield a numerical
invariant of infinite groups called the compression exponent. This
elegant definition is due to Guentner and Kaminker~\cite{GK04}, and
it was extensively studied in recent years (see the introduction
to~\cite{NP11} for background and references). We will focus here on
the use of random walk techniques in the
 task of computing (or estimating) this invariant.

The Guentner-Kaminker definition is simple to state. Given a metric
space $(Y,d_Y)$, the $Y$-compression exponent of $G$, denoted
$\alpha^*_Y(G)$, is the supremum of those $\alpha\ge 0$ for which
there exists a Lipschitz function $f:G\to Y$ which satisfies
$d_Y(f(x),f(y))\gtrsim d_S(x,y)^\alpha$  for all $x,y\in X$. We
remark that in the notation $\alpha^*_Y(G)$ we dropped the explicit
reference to the generating set $S$. This is legitimate since, if we
switch to a different finite symmetric generating set $S'\subseteq
G$, then the resulting word metric $d_{S'}$ is bi-Lipschitz
equivalent to the original word metric $d_S$, and therefore the
$Y$-compression exponents of $S$ and $S'$ coincide. In other words,
the number $\alpha_Y^*(G)\in [0,1]$ is a true algebraic invariant of
the group $G$, which does not depend on the particular choice of a
finite symmetric set of generators. The parameter
$\alpha_{\ell_2}^*(G)$ is called the Hilbert compression exponent of
$G$. It was shown in~\cite{ADS06} that any $\alpha\in [0,1]$ is the
Hilbert compression exponent of some finitely generated group $G$
(see~\cite{Aus11,OO12} for the related question for amenable
groups). Nevertheless, there are relatively few concrete examples of
groups $G$ for which $\alpha_{\ell_2}^*(G)$ (and
$\alpha_{\ell_p}^*(G)$) has been computed. We will demonstrate how
Markov type is relevant to the problem of estimating
$\alpha_Y^*(G)$. This approach was introduced in~\cite{ANP09}, and
further refined in~\cite{NP08,NP11}.

We will examine the applicability of random walks to the computation
of compression exponents of discrete groups via an illustrative
example: the wreath product of the group of integers $\Z$ with
itself. Before doing so, we recall for the sake of completeness the
definition of the wreath product of two general groups $G, H$.
Readers who are not accustomed to this concept are encouraged to
focus on the case $G=H=\Z$, as it contains the essential ideas that
we wish to convey.

Let $G,H$ be groups which are generated by the finite symmetric sets
$S_G\subseteq G$, $S_H\subseteq H$. We denote by $e_G,e_H$ the
identity elements of $G,H$, respectively. We also denote by
$e_{G^H}$ the function from $H$ to $G$ which takes the value $e_G$
at all points $x\in H$. The (restricted) wreath product of $G$ with
$H$, denoted $G\bwr H$, is defined as the group of all pairs $(f,x)$
where $f:H\to G$ has finite support (i.e., $f(z)= e_{G^H}(z)=e_G$
for all but finitely many $z\in H$) and $x\in H$, equipped with the
product $$ (f,x)(g,y)= \left(z\mapsto f(z)g(x^{-1}z),xy\right).$$
$G\bwr H$ is generated by the set $\{(e_{G^H},x):\ x\in S_H\}\cup
\{(\delta_y,e_H):\ y\in S_G\}$, where $\delta_y:H\to G$ is the
function which takes the value $y$ at $e_H$ and the value $e_G$ on
$H\setminus\{e_H\}$.

When $G=C_{2}=\{0,1\}$, the cyclic group of order $2$, then the
group $C_2\bwr H$ is often called the lamplighter group on $H$. In
this case imagine that at every site $x\in H$ there is a lamp,
which can either be on or off. An element $(f,x)\in C_{2}\bwr H$ can
be thought of as indicating that a ``lamplighter" is located at
$x\in H$, and $f$ represents the locations of those (finitely many)
lamps which are on (these locations are the sites $y\in H$ where
$f(y)=1$). The distance in $C_2\bwr H$ between $(f,x)$ and $(g,y)$
is the minimum number of steps required for the lamplighter to start
at $x$, visit all the sites $z\in H$ for which $f(z)\neq g(z)$,
change $f(z)$ to $g(z)$, and end up at the site $y$. Here, by a
``step" we mean a move from $x$ to $xs$ for some $s\in S_H$, or a
change of the state of the lamp (from on to off or vice versa) at
the current location of the lamplighter. Thus, the distance between
$(f,x)$ and $(g,y)$ is, up to a factor of $2$, the shortest (in the
metric $d_{S_H}$) traveling salesman tour starting at $x$, covering
the symmetric difference of the supports of $f$ and $g$, and
terminating at $y$. For a general group $G$, the description of the
metric on $G\bwr H$ is similar, the only difference being that the
lamps can have $G$ different states (not just on or off), and the
cost of changing the state of a lamp from $a\in G$ to $b\in G$ is
$d_{S_G}(a,b)$. See Figure~\ref{fig:wr} for a schematic description
of the case $G=H=\Z$.

\begin{figure}[ht]\label{fig:wr}
\centering \fbox{
\begin{minipage}{4.7in}
\centering
\includegraphics[scale=0.6]{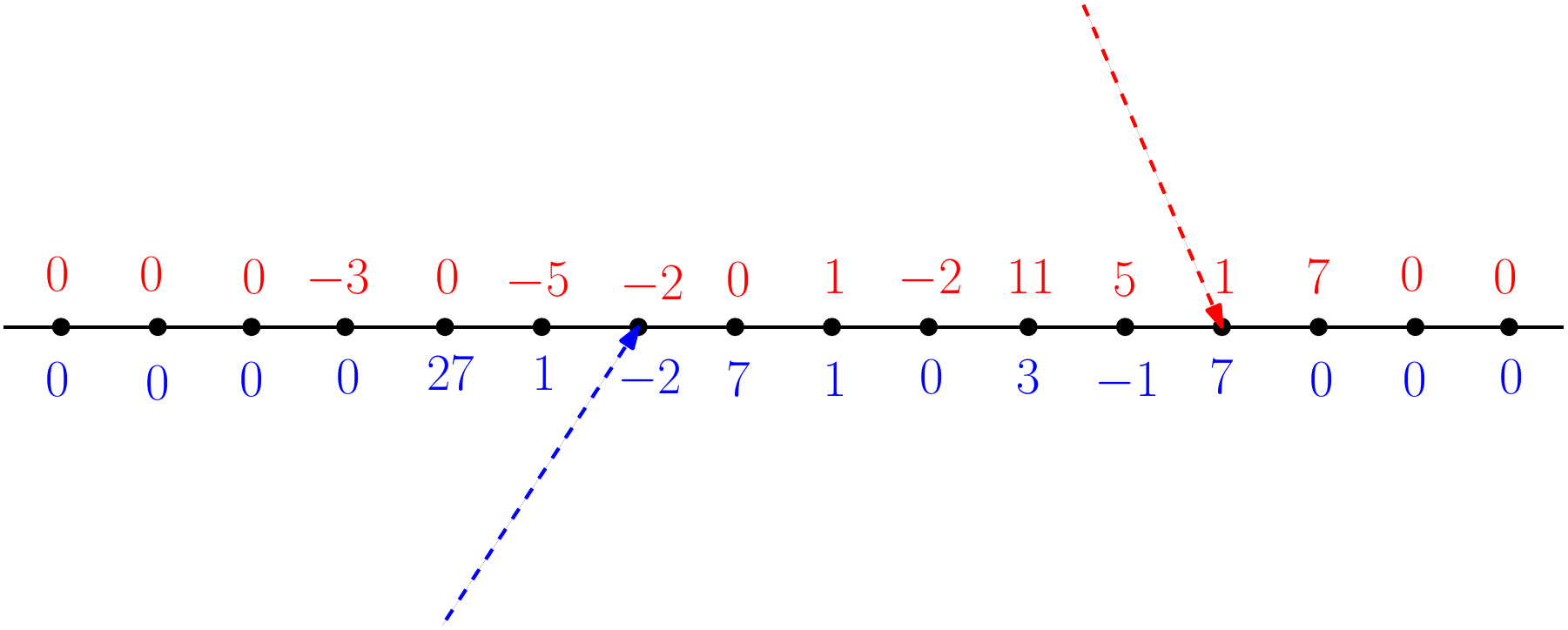}
 \caption{An example of two elements in $\Z\bwr \Z$. The numbers above
 and below the $\Z$-axis represent two $\Z$-valued finitely supported functions.
 The arrows indicate the location of the lamplighter in each of the corresponding
 elements of $\Z\bwr \Z$. In order to compute the distance between these elements,
 the lamplighter of the top configuration must visit  the locations where the top
 values differ from the bottom values, and at these locations the top value must
 be changed to the bottom value. At the end of the process, the lamplighter must
 end up at the location indicated by the bottom arrow. Each movement of the
 lamplighter to a neighboring integer adds a unit cost to this process, and
 an increment or decrement of $1$ to the value at the location of the lamplighter
 also incurs a unit cost. The distance in $\Z\bwr \Z$ is the minimum cost of such
 a process which transforms the top configuration to the bottom configuration.
 Thus, in the above example, the top lamplighter will first move one step to the right,
 incurring a unit cost, change the $7$ in the top row to $0$, incurring a cost of $7$ units,
 take one step to the left (incurring a unit cost), change the $1$ in the top row to a $7$
 (incurring a cost of $6$ units), and so on.}
  \end{minipage}
}
\end{figure}

We shall now describe an argument using random walks, showing that
$\alpha_{\ell_2}^*(\Z\bwr \Z)\le \frac23$. This approach is due
to~\cite{ANP09}. In fact, as shown in~\cite{NP08},
$\alpha_{\ell_2}^*(\Z\bwr \Z)\ge \frac23$, and therefore the
argument below is sharp, and yields the exact computation
$\alpha_{\ell_2}^*(\Z\bwr \Z)= \frac23$. More generally, it is shown
in~\cite{NP11} that for every $p\in [1,2]$ we have
\begin{equation}\label{eq:alphap}
\alpha_{\ell_p}^*(\Z\bwr\Z)=\frac{p}{2p-1}.
\end{equation}
The proof of~\eqref{eq:alphap} when $p\neq 2$ requires an additional
idea that we will not work out in detail here: instead of examining
the standard random walk on $\Z\bwr \Z$ one studies a discrete
version of a $q$-stable random walk  for every $q\in (p,2]$. This
yields a new twist of the Markov type method: it is beneficial to
adapt the random walk to the geometry of the target space, and  to
use random walks with unbounded increments (though, we have already
seen the latter occur in Section~\ref{sec:m-jump}). We refer
to~\cite{NP08,NP11} for more general results that go beyond that
case of $\Z\bwr \Z$, as well as an explanation of the background,
history, and applications of these types of problems. It suffices to
say here that we chose to focus on the group $\Z\bwr \Z$ since
before the introduction of random walk techniques, it was the
simplest concrete group which resisted the attempts to compute its
$\ell_p$ compression exponents.

Consider the standard random walk $\{W_t\}_{t=0}^\infty$ on $\Z \bwr
\Z$, starting at the identity element. Namely, we start at
$e_{\Z\bwr\Z}$, i.e., the configuration corresponding to all the
lamps being turned off, and the lamplighter being at $0$. At each
step a fair coin is tossed, and depending on the outcome of the coin
toss, either the lamplighter moves to one of its two neighboring
locations uniformly at random, or the value at the current location
of the lamplighter is changed by $+1$ or $-1$ uniformly at random.

After $t$ steps, we expect that a constant fraction of the coin
tosses resulted in a movement of the lamplighter, which is just a
standard random walk on the integers $\Z$. Thus, at time $t$ we
expect the lamplighter to be located at $\pm\asymp  \sqrt{t}$. One
might also expect that during the walk the lamplighter spent roughly
(up to constant factors) the same amount of total time at a definite
fraction of the sites between $0$ and its location at time $t$.
There are $\asymp \sqrt{t}$ such sites, and therefore, if this
intuition is indeed correct, we expect the time spent at each of
these sites to be $\asymp \frac{t}{\sqrt{t}}=\sqrt{t}$. At each such
site the value of the lamp is also the result of a random walk on
$\Z$, and therefore at time $t$ we expect $W_t$ to have $\asymp
\sqrt{t}$ sites at which the value of the lamp is $\pm\asymp
\sqrt{\sqrt{t}}=\pm\sqrt[4]{t}$. This heuristic argument suggests
that
\begin{equation}\label{eq:erschler}
\E\left[d_{\Z\bwr\Z}\left(W_t,e_{\Z\bwr\Z}\right)\right]\gtrsim \sqrt{t}\cdot\sqrt[4]{t}=
t^{\frac34}.
\end{equation}
These considerations can indeed be made to yield a rigorous proof
 of~\eqref{eq:erschler}; see~\cite{Ersh01} and also~\cite{Rev03}, as well as Section 6
in~\cite{NP08} for an extension to the case of general wreath
products.

The fact that $\ell_2$ has Markov type $2$ suggests that if
$f:\Z\bwr \Z\to \ell_2$ satisfies $d_{\Z\bwr\Z}(x,y)^\alpha \lesssim
\|f(x)-f(y)\|_2\lesssim d_{\Z\bwr\Z}(x,y)$, then
$$\E\left[\|f(W_t)-f(W_0)\|_2^2\right]\lesssim t,$$ yet due
to~\eqref{eq:erschler},
$$\E\left[\|f(W_t)-f(W_0)\|_2^2\right]\gtrsim t^{2\alpha\cdot
\frac34}.$$ This implies that $\alpha\le \frac23$, as required. But,
this argument is flawed: we are only allowed to use the Markov type
$2$ inequality~\eqref{eq:def Mtype} for stationary reversible Markov
chains. The Markov chain $\{W_t\}_{t=0}^\infty$ starts at the
deterministic point $e_{\Z\bwr\Z}$ rather than at a point chosen
uniformly at random over $\Z\bwr \Z$. Of course, since $\Z\bwr \Z$
is an infinite set, there is no way to make $W_0$ be uniformly
distributed over it. The above argument can be salvaged by either
considering instead an appropriately truncated random walk starting
at a uniformly chosen point from a large enough F\o lner set of
$\Z\bwr \Z$, or by applying an argument of Aharoni, Maurey and
Mityagin~\cite{AhaMauMit} and Gromov~\cite{dCTV07} (see
also~\cite{NP11}) to reduce the problem to equivariant embeddings,
and then to prove that the Markov type inequality does hold true for
images of the random walk $\{W_t\}_{t=0}^\infty$ (starting at
$e_{\Z\bwr\Z}$) under equivariant mappings. See~\cite{ANP09} for the
former approach and~\cite{NP08} for the latter approach.


\bibliographystyle{abbrv}
\bibliography{ribe}

\def\cprime{$'$} \def\cprime{$'$} \def\cprime{$'$} \def\cprime{$'$}
  \def\cdprime{$''$} \def\cprime{$'$}
\begin{thebibliography}{100}

\bibitem{AL78}
I.~Aharoni and J.~Lindenstrauss.
\newblock Uniform equivalence between {B}anach spaces.
\newblock {\em Bull. Amer. Math. Soc.}, 84(2):281--283, 1978.

\bibitem{AhaMauMit}
I.~Aharoni, B.~Maurey, and B.~S. Mityagin.
\newblock Uniform embeddings of metric spaces and of {B}anach spaces into
  {H}ilbert spaces.
\newblock {\em Israel J. Math.}, 52(3):251--265, 1985.

\bibitem{AM85}
N.~Alon and V.~D. Milman.
\newblock {$\lambda\sb 1,$} isoperimetric inequalities for graphs, and
  superconcentrators.
\newblock {\em J. Combin. Theory Ser. B}, 38(1):73--88, 1985.

\bibitem{And66}
R.~D. Anderson.
\newblock Hilbert space is homeomorphic to the countable infinite product of
  lines.
\newblock {\em Bull. Amer. Math. Soc.}, 72:515--519, 1966.

\bibitem{ANN10}
A.~Andoni, O.~Neiman, and A.~Naor.
\newblock Snowflake universality of {W}asserstein spaces.
\newblock Preprint, 2010.

\bibitem{ALN08}
S.~Arora, J.~R. Lee, and A.~Naor.
\newblock Euclidean distortion and the sparsest cut.
\newblock {\em J. Amer. Math. Soc.}, 21(1):1--21 (electronic), 2008.

\bibitem{ALNRRV06}
S.~Arora, L.~Lov{\'a}sz, I.~Newman, Y.~Rabani, Y.~Rabinovich, and S.~Vempala.
\newblock Local versus global properties of metric spaces (extended abstract).
\newblock In {\em Proceedings of the {S}eventeenth {A}nnual {ACM}-{SIAM}
  {S}ymposium on {D}iscrete {A}lgorithms}, pages 41--50, New York, 2006. ACM.

\bibitem{ADS06}
G.~Arzhantseva, C.~Drutu, and M.~Sapir.
\newblock Compression functions of uniform embeddings of groups into {H}ilbert
  and {B}anach spaces.
\newblock {\em J. Reine Angew. Math.}, 633:213--235, 2009.

\bibitem{Ass83}
P.~Assouad.
\newblock Plongements lipschitziens dans {${\bf R}^{n}$}.
\newblock {\em Bull. Soc. Math. France}, 111(4):429--448, 1983.

\bibitem{AR98}
Y.~Aumann and Y.~Rabani.
\newblock An {$O(\log k)$} approximate min-cut max-flow theorem and
  approximation algorithm.
\newblock {\em SIAM J. Comput.}, 27(1):291--301 (electronic), 1998.

\bibitem{Aus11}
T.~Austin.
\newblock Amenable groups with very poor compression into {L}ebesgue spaces.
\newblock {\em Duke Math. J.}, 159(2):187--222, 2011.

\bibitem{AN10}
T.~Austin and A.~Naor.
\newblock On the bi-{L}ipschitz structure of {W}asserstein spaces.
\newblock Preprint., 2009.

\bibitem{ANP09}
T.~Austin, A.~Naor, and Y.~Peres.
\newblock The wreath product of {$\Bbb Z$} with {$\Bbb Z$} has {H}ilbert
  compression exponent {$\frac{2}{3}$}.
\newblock {\em Proc. Amer. Math. Soc.}, 137(1):85--90, 2009.

\bibitem{ANV10}
T.~Austin, A.~Naor, and A.~Valette.
\newblock The {E}uclidean distortion of the lamplighter group.
\newblock {\em Discrete Comput. Geom.}, 44(1):55--74, 2010.

\bibitem{Bal92}
K.~Ball.
\newblock Markov chains, {R}iesz transforms and {L}ipschitz maps.
\newblock {\em Geom. Funct. Anal.}, 2(2):137--172, 1992.

\bibitem{Bal12}
K.~Ball.
\newblock The {R}ibe programme.
\newblock S\'eminaire Bourbaki, expos\'e 1047, 2012.

\bibitem{BCL94}
K.~Ball, E.~A. Carlen, and E.~H. Lieb.
\newblock Sharp uniform convexity and smoothness inequalities for trace norms.
\newblock {\em Invent. Math.}, 115(3):463--482, 1994.

\bibitem{Ban32}
S.~Banach.
\newblock {\em Th\'eorie des op\'erations lin\'eaires}.
\newblock PWN---Polish Scientific Publishers, Warsaw, 1932.
\newblock Monografie Matematyczne, Tom 1. [Mathematical Monographs, Vol. 1].

\bibitem{Bar98}
Y.~Bartal.
\newblock On approximating arbitrary metrics by tree metrics.
\newblock In {\em S{TOC} '98 ({D}allas, {TX})}, pages 161--168. ACM, New York,
  1999.

\bibitem{BBM}
Y.~Bartal, B.~Bollob{\'a}s, and M.~Mendel.
\newblock Ramsey-type theorems for metric spaces with applications to online
  problems.
\newblock {\em J. Comput. System Sci.}, 72(5):890--921, 2006.

\bibitem{BLMN05-dic}
Y.~Bartal, N.~Linial, M.~Mendel, and A.~Naor.
\newblock On metric {R}amsey-type dichotomies.
\newblock {\em J. London Math. Soc. (2)}, 71(2):289--303, 2005.

\bibitem{BLMN05}
Y.~Bartal, N.~Linial, M.~Mendel, and A.~Naor.
\newblock On metric {R}amsey-type phenomena.
\newblock {\em Ann. of Math. (2)}, 162(2):643--709, 2005.

\bibitem{BLMN05-low}
Y.~Bartal, N.~Linial, M.~Mendel, and A.~Naor.
\newblock Some low distortion metric {R}amsey problems.
\newblock {\em Discrete Comput. Geom.}, 33(1):27--41, 2005.

\bibitem{BF-C00}
M.~A. Bender and M.~Farach-Colton.
\newblock The {L}{C}{A} problem revisited.
\newblock In G.~H. Gonnet, D.~Panario, and A.~Viola, editors, {\em LATIN},
  volume 1776 of {\em Lecture Notes in Computer Science}, pages 88--94.
  Springer, 2000.

\bibitem{Ben85}
Y.~Benyamini.
\newblock The uniform classification of {B}anach spaces.
\newblock In {\em Texas functional analysis seminar 1984--1985 ({A}ustin,
  {T}ex.)}, Longhorn Notes, pages 15--38. Univ. Texas Press, Austin, TX, 1985.

\bibitem{BL00}
Y.~Benyamini and J.~Lindenstrauss.
\newblock {\em Geometric nonlinear functional analysis. {V}ol. 1}, volume~48 of
  {\em American Mathematical Society Colloquium Publications}.
\newblock American Mathematical Society, Providence, RI, 2000.

\bibitem{Bes65}
C.~Bessaga.
\newblock On topological classification of complete linear metric spaces.
\newblock {\em Fund. Math.}, 56:251--288, 1964/1965.

\bibitem{BP60}
C.~Bessaga and A.~Pe{\l}czy{\'n}ski.
\newblock Some remarks on homeomorphisms of {B}anach spaces.
\newblock {\em Bull. Acad. Polon. Sci. S\'er. Sci. Math. Astronom. Phys.},
  8:757--761, 1960.

\bibitem{BP75}
C.~Bessaga and A.~Pe{\l}czy{\'n}ski.
\newblock {\em Selected topics in infinite-dimensional topology}.
\newblock PWN---Polish Scientific Publishers, Warsaw, 1975.
\newblock Monografie Matematyczne, Tom 58. [Mathematical Monographs, Vol. 58].

\bibitem{BKRS00}
A.~Blum, H.~Karloff, Y.~Rabani, and M.~Saks.
\newblock A decomposition theorem for task systems and bounds for randomized
  server problems.
\newblock {\em SIAM J. Comput.}, 30(5):1624--1661 (electronic), 2000.

\bibitem{Bou85}
J.~Bourgain.
\newblock On {L}ipschitz embedding of finite metric spaces in {H}ilbert space.
\newblock {\em Israel J. Math.}, 52(1-2):46--52, 1985.

\bibitem{Bou86}
J.~Bourgain.
\newblock The metrical interpretation of superreflexivity in {B}anach spaces.
\newblock {\em Israel J. Math.}, 56(2):222--230, 1986.

\bibitem{Bou87}
J.~Bourgain.
\newblock Remarks on the extension of {L}ipschitz maps defined on discrete sets
  and uniform homeomorphisms.
\newblock In {\em Geometrical aspects of functional analysis (1985/86)}, volume
  1267 of {\em Lecture Notes in Math.}, pages 157--167. Springer, Berlin, 1987.

\bibitem{BFM86}
J.~Bourgain, T.~Figiel, and V.~Milman.
\newblock On {H}ilbertian subsets of finite metric spaces.
\newblock {\em Israel J. Math.}, 55(2):147--152, 1986.

\bibitem{BMW86}
J.~Bourgain, V.~Milman, and H.~Wolfson.
\newblock On type of metric spaces.
\newblock {\em Trans. Amer. Math. Soc.}, 294(1):295--317, 1986.

\bibitem{BH99}
M.~R. Bridson and A.~Haefliger.
\newblock {\em Metric spaces of non-positive curvature}, volume 319 of {\em
  Grundlehren der Mathematischen Wissenschaften [Fundamental Principles of
  Mathematical Sciences]}.
\newblock Springer-Verlag, Berlin, 1999.

\bibitem{BKL07}
B.~Brinkman, A.~Karagiozova, and J.~R. Lee.
\newblock Vertex cuts, random walks, and dimension reduction in series-parallel
  graphs.
\newblock In {\em S{TOC}'07---{P}roceedings of the 39th {A}nnual {ACM}
  {S}ymposium on {T}heory of {C}omputing}, pages 621--630. ACM, New York, 2007.

\bibitem{CK05}
M.~Charikar and A.~Karagiozova.
\newblock A tight threshold for metric {R}amsey phenomena.
\newblock In {\em Proceedings of the {S}ixteenth {A}nnual {ACM}-{SIAM}
  {S}ymposium on {D}iscrete {A}lgorithms}, pages 129--136 (electronic), New
  York, 2005. ACM.

\bibitem{CMM07}
M.~Charikar, K.~Makarychev, and Y.~Makarychev.
\newblock Local global tradeoffs in metric embeddings.
\newblock In {\em Proceedings of 48th Annual IEEE Symposium on Foundations of
  Computer Science (FOCS 2007)}, pages 713--723, 2007.

\bibitem{Che70}
J.~Cheeger.
\newblock A lower bound for the smallest eigenvalue of the {L}aplacian.
\newblock In {\em Problems in analysis ({P}apers dedicated to {S}alomon
  {B}ochner, 1969)}, pages 195--199. Princeton Univ. Press, Princeton, N. J.,
  1970.

\bibitem{CK10}
J.~Cheeger and B.~Kleiner.
\newblock Differentiating maps into {$L^1$}, and the geometry of {BV}
  functions.
\newblock {\em Ann. of Math. (2)}, 171(2):1347--1385, 2010.

\bibitem{CCJJV01}
P.-A. Cherix, M.~Cowling, P.~Jolissaint, P.~Julg, and A.~Valette.
\newblock {\em Groups with the {H}aagerup property}, volume 197 of {\em
  Progress in Mathematics}.
\newblock Birkh\"auser Verlag, Basel, 2001.
\newblock Gromov's a-T-menability.

\bibitem{Chr90}
M.~Christ.
\newblock A {$T(b)$} theorem with remarks on analytic capacity and the {C}auchy
  integral.
\newblock {\em Colloq. Math.}, 60/61(2):601--628, 1990.

\bibitem{CK63}
H.~Corson and V.~Klee.
\newblock Topological classification of convex sets.
\newblock In {\em Proc. {S}ympos. {P}ure {M}ath., {V}ol. {VII}}, pages 37--51.
  Amer. Math. Soc., Providence, R.I., 1963.

\bibitem{dCTV07}
Y.~de~Cornulier, R.~Tessera, and A.~Valette.
\newblock Isometric group actions on {H}ilbert spaces: growth of cocycles.
\newblock {\em Geom. Funct. Anal.}, 17(3):770--792, 2007.

\bibitem{delaHarpe00}
P.~de~la Harpe.
\newblock {\em Topics in geometric group theory}.
\newblock Chicago Lectures in Mathematics. University of Chicago Press,
  Chicago, IL, 2000.

\bibitem{Dvo60}
A.~Dvoretzky.
\newblock Some results on convex bodies and {B}anach spaces.
\newblock In {\em Proc. {I}nternat. {S}ympos. {L}inear {S}paces ({J}erusalem,
  1960)}, pages 123--160. Jerusalem Academic Press, Jerusalem, 1961.

\bibitem{Enf69-smirnov}
P.~Enflo.
\newblock On a problem of {S}mirnov.
\newblock {\em Ark. Mat.}, 8:107--109, 1969.

\bibitem{Enf69}
P.~Enflo.
\newblock On the nonexistence of uniform homeomorphisms between
  {$L_{p}$}-spaces.
\newblock {\em Ark. Mat.}, 8:103--105 (1969), 1969.

\bibitem{Enf70}
P.~Enflo.
\newblock Uniform structures and square roots in topological groups. {I}, {II}.
\newblock {\em Israel J. Math. 8 (1970), 230-252; ibid.}, 8:253--272, 1970.

\bibitem{Enf72}
P.~Enflo.
\newblock Banach spaces which can be given an equivalent uniformly convex norm.
\newblock {\em Israel J. Math.}, 13:281--288, 1973.

\bibitem{Enf76}
P.~Enflo.
\newblock Uniform homeomorphisms between {B}anach spaces.
\newblock In {\em S\'eminaire {M}aurey-{S}chwartz (1975--1976), {E}spaces,
  {$L^{p}$}, applications radonifiantes et g\'eom\'etrie des espaces de
  {B}anach, {E}xp. {N}o. 18}, page~7. Centre Math., \'Ecole Polytech.,
  Palaiseau, 1976.

\bibitem{Erd64}
P.~Erd{\H{o}}s.
\newblock Extremal problems in graph theory.
\newblock In {\em Theory of {G}raphs and its {A}pplications ({P}roc. {S}ympos.
  {S}molenice, 1963)}, pages 29--36. Publ. House Czechoslovak Acad. Sci.,
  Prague, 1964.

\bibitem{Ersh01}
A.~G. {\`E}rschler.
\newblock On the asymptotics of the rate of departure to infinity.
\newblock {\em Zap. Nauchn. Sem. S.-Peterburg. Otdel. Mat. Inst. Steklov.
  (POMI)}, 283(Teor. Predst. Din. Sist. Komb. i Algoritm. Metody. 6):251--257,
  263, 2001.

\bibitem{FHHMPZ01}
M.~Fabian, P.~Habala, P.~H{\'a}jek, V.~Montesinos~Santaluc{\'{\i}}a, J.~Pelant,
  and V.~Zizler.
\newblock {\em Functional analysis and infinite-dimensional geometry}.
\newblock CMS Books in Mathematics/Ouvrages de Math\'ematiques de la SMC, 8.
  Springer-Verlag, New York, 2001.

\bibitem{FRT04}
J.~Fakcharoenphol, S.~Rao, and K.~Talwar.
\newblock A tight bound on approximating arbitrary metrics by tree metrics.
\newblock {\em J. Comput. System Sci.}, 69(3):485--497, 2004.

\bibitem{Fer74}
X.~Fernique.
\newblock Regularit\'e des trajectoires des fonctions al\'eatoires gaussiennes.
\newblock In {\em \'{E}cole d'\'{E}t\'e de {P}robabilit\'es de {S}aint-{F}lour,
  {IV}-1974}, pages 1--96. Lecture Notes in Math., Vol. 480. Springer, Berlin,
  1975.

\bibitem{Fer76}
X.~Fernique.
\newblock \'{E}valuations de processus gaussiens compos\'es.
\newblock In {\em Probability in {B}anach spaces ({P}roc. {F}irst {I}nternat.
  {C}onf., {O}berwolfach, 1975)}, pages 67--83. Lecture Notes in Math., Vol.
  526. Springer, Berlin, 1976.

\bibitem{Fer78}
X.~Fernique.
\newblock Caract\'erisation de processus \`a trajectoires major\'ees ou
  continues.
\newblock In {\em S\'eminaire de {P}robabilit\'es, {XII} ({U}niv. {S}trasbourg,
  {S}trasbourg, 1976/1977)}, volume 649 of {\em Lecture Notes in Math.}, pages
  691--706. Springer, Berlin, 1978.

\bibitem{Fig68}
T.~Figiel.
\newblock On nonlinear isometric embeddings of normed linear spaces.
\newblock {\em Bull. Acad. Polon. Sci. S\'er. Sci. Math. Astronom. Phys.},
  16:185--188, 1968.

\bibitem{FLM77}
T.~Figiel, J.~Lindenstrauss, and V.~D. Milman.
\newblock The dimension of almost spherical sections of convex bodies.
\newblock {\em Acta Math.}, 139(1-2):53--94, 1977.

\bibitem{Fre28}
M.~Fr\'echet.
\newblock {\em {Les espaces abstraits et leur th\'eorie consid\'er\'ee comme
  introduction \`a l'analyse g\'en\'erale.}}
\newblock {Paris: Gauthier-Villars (Collection de monographies sur la th\'eorie
  des fonctions). XII, 296 p. }, 1928.

\bibitem{GMN11}
O.~Giladi, M.~Mendel, and A.~Naor.
\newblock Improved bounds in the metric cotype inequality for {B}anach spaces.
\newblock {\em J. Funct. Anal.}, 260(1):164--194, 2011.

\bibitem{GNS11}
O.~Giladi, A.~Naor, and G.~Schechtman.
\newblock Bourgain's discretization theorem.
\newblock Preprint available at \url{http://arxiv.org/abs/1110.5368}, to appear
  in {\em Ann. Fac. Sci. Toulouse Math.}, 2011.

\bibitem{Gor94}
E.~Gorelik.
\newblock The uniform nonequivalence of {$L_p$} and {$l_p$}.
\newblock {\em Israel J. Math.}, 87(1-3):1--8, 1994.

\bibitem{Gro82}
M.~Gromov.
\newblock Volume and bounded cohomology.
\newblock {\em Inst. Hautes \'Etudes Sci. Publ. Math.}, (56):5--99 (1983),
  1982.

\bibitem{Gro83}
M.~Gromov.
\newblock Filling {R}iemannian manifolds.
\newblock {\em J. Differential Geom.}, 18(1):1--147, 1983.

\bibitem{Gro93}
M.~Gromov.
\newblock Asymptotic invariants of infinite groups.
\newblock In {\em Geometric group theory, {V}ol.\ 2 ({S}ussex, 1991)}, volume
  182 of {\em London Math. Soc. Lecture Note Ser.}, pages 1--295. Cambridge
  Univ. Press, Cambridge, 1993.

\bibitem{Gro07}
M.~Gromov.
\newblock {\em Metric structures for {R}iemannian and non-{R}iemannian spaces}.
\newblock Modern Birkh\"auser Classics. Birkh\"auser Boston Inc., Boston, MA,
  english edition, 2007.
\newblock Based on the 1981 French original, With appendices by M. Katz, P.
  Pansu and S. Semmes, Translated from the French by Sean Michael Bates.

\bibitem{Gro53-dvo}
A.~Grothendieck.
\newblock Sur certaines classes de suites dans les espaces de {B}anach et le
  th\'eor\`eme de {D}voretzky-{R}ogers.
\newblock {\em Bol. Soc. Mat. S\~ao Paulo}, 8:81--110 (1956), 1953.

\bibitem{Gro56}
A.~Grothendieck.
\newblock Erratum au m\'emoire: {P}roduits tensoriels topologiques et espaces
  nucl\'eaires.
\newblock {\em Ann. Inst. Fourier, Grenoble}, 6:117--120, 1955--1956.

\bibitem{GK04}
E.~Guentner and J.~Kaminker.
\newblock Exactness and uniform embeddability of discrete groups.
\newblock {\em J. London Math. Soc. (2)}, 70(3):703--718, 2004.

\bibitem{GNRS04}
A.~Gupta, I.~Newman, Y.~Rabinovich, and A.~Sinclair.
\newblock Cuts, trees and {$l_1$}-embeddings of graphs.
\newblock {\em Combinatorica}, 24(2):233--269, 2004.

\bibitem{Han56}
O.~Hanner.
\newblock On the uniform convexity of {$L^p$} and {$l^p$}.
\newblock {\em Ark. Mat.}, 3:239--244, 1956.

\bibitem{H-PM06}
S.~Har-Peled and M.~Mendel.
\newblock Fast construction of nets in low-dimensional metrics and their
  applications.
\newblock {\em SIAM J. Comput.}, 35(5):1148--1184 (electronic), 2006.

\bibitem{HT84}
D.~Harel and R.~E. Tarjan.
\newblock Fast algorithms for finding nearest common ancestors.
\newblock {\em SIAM J. Comput.}, 13(2):338--355, 1984.

\bibitem{HM82}
S.~Heinrich and P.~Mankiewicz.
\newblock Applications of ultrapowers to the uniform and {L}ipschitz
  classification of {B}anach spaces.
\newblock {\em Studia Math.}, 73(3):225--251, 1982.

\bibitem{Hen67}
G.~M. Henkin.
\newblock The lack of a uniform homeomorphism between the spaces of smooth
  functions of one and of {$n$} variables {$(n\geq 2)$}.
\newblock {\em Mat. Sb. (N.S.)}, 74 (116):595--607, 1967.

\bibitem{HLW06}
S.~Hoory, N.~Linial, and A.~Wigderson.
\newblock Expander graphs and their applications.
\newblock {\em Bull. Amer. Math. Soc. (N.S.)}, 43(4):439--561 (electronic),
  2006.

\bibitem{How95}
J.~D. Howroyd.
\newblock On dimension and on the existence of sets of finite positive
  {H}ausdorff measure.
\newblock {\em Proc. London Math. Soc. (3)}, 70(3):581--604, 1995.

\bibitem{Hug04}
B.~Hughes.
\newblock Trees and ultrametric spaces: a categorical equivalence.
\newblock {\em Adv. Math.}, 189(1):148--191, 2004.

\bibitem{HLN12}
T.~Hyt\"onen, S.~Li, and A.~Naor.
\newblock Quantitative affine approximation for {U}{M}{D} targets.
\newblock Preprint, 2012.

\bibitem{HN12}
T.~Hyt\"onen and A.~Naor.
\newblock {P}isier's inequality revisited.
\newblock Preprint, 2012.

\bibitem{Jam72-self}
R.~C. James.
\newblock Some self-dual properties of normed linear spaces.
\newblock In {\em Symposium on {I}nfinite-{D}imensional {T}opology ({L}ouisiana
  {S}tate {U}niv., {B}aton {R}ouge, {L}a., 1967)}, pages 159--175. Ann. of
  Math. Studies, No. 69. Princeton Univ. Press, Princeton, N.J., 1972.

\bibitem{Jam72}
R.~C. James.
\newblock Super-reflexive {B}anach spaces.
\newblock {\em Canad. J. Math.}, 24:896--904, 1972.

\bibitem{Jam78}
R.~C. James.
\newblock Nonreflexive spaces of type {$2$}.
\newblock {\em Israel J. Math.}, 30(1-2):1--13, 1978.

\bibitem{Joh48}
F.~John.
\newblock Extremum problems with inequalities as subsidiary conditions.
\newblock In {\em Studies and {E}ssays {P}resented to {R}. {C}ourant on his
  60th {B}irthday, {J}anuary 8, 1948}, pages 187--204. Interscience Publishers,
  Inc., New York, N. Y., 1948.

\bibitem{JL84}
W.~B. Johnson and J.~Lindenstrauss.
\newblock Extensions of {L}ipschitz mappings into a {H}ilbert space.
\newblock In {\em Conference in modern analysis and probability ({N}ew {H}aven,
  {C}onn., 1982)}, volume~26 of {\em Contemp. Math.}, pages 189--206. Amer.
  Math. Soc., Providence, RI, 1984.

\bibitem{JLS96}
W.~B. Johnson, J.~Lindenstrauss, and G.~Schechtman.
\newblock Banach spaces determined by their uniform structures.
\newblock {\em Geom. Funct. Anal.}, 6(3):430--470, 1996.

\bibitem{JS09}
W.~B. Johnson and G.~Schechtman.
\newblock Diamond graphs and super-reflexivity.
\newblock {\em J. Topol. Anal.}, 1(2):177--189, 2009.

\bibitem{Kad58}
M.~{\u{I}}. Kadec{\cprime}.
\newblock On strong and weak convergence.
\newblock {\em Dokl. Akad. Nauk SSSR}, 122:13--16, 1958.

\bibitem{Kad66}
M.~{\u{I}}. Kadec{\cprime}.
\newblock Topological equivalence of all separable {B}anach spaces.
\newblock {\em Dokl. Akad. Nauk SSSR}, 167:23--25, 1966.

\bibitem{Kad67}
M.~I. Kadec.
\newblock A proof of the topological equivalence of all separable
  infinite-dimensional {B}anach spaces.
\newblock {\em Funkcional. Anal. i Prilo\v zen.}, 1:61--70, 1967.

\bibitem{Kah64}
J.-P. Kahane.
\newblock Sur les sommes vectorielles {$\sum \pm u_{n}$}.
\newblock {\em C. R. Acad. Sci. Paris}, 259:2577--2580, 1964.

\bibitem{Kal11}
N.~Kalton.
\newblock The uniform structure of {B}anach spaces.
\newblock To appear in {\em Math. Ann.}, 2011.

\bibitem{Kal08}
N.~J. Kalton.
\newblock The nonlinear geometry of {B}anach spaces.
\newblock {\em Rev. Mat. Complut.}, 21(1):7--60, 2008.

\bibitem{KKR94}
H.~Karloff, Y.~Rabani, and Y.~Ravid.
\newblock Lower bounds for randomized {$k$}-server and motion-planning
  algorithms.
\newblock {\em SIAM J. Comput.}, 23(2):293--312, 1994.

\bibitem{KMZ12}
T.~Keleti, A.~M\'ath\'e, and O.~Zindulka.
\newblock Hausdorff dimension of metric spaces and {L}ipschitz maps onto cubes.
\newblock Preprint, available at \url{http://arxiv.org/abs/1203.0686}, 2012.

\bibitem{KN06}
S.~Khot and A.~Naor.
\newblock Nonembeddability theorems via {F}ourier analysis.
\newblock {\em Math. Ann.}, 334(4):821--852, 2006.

\bibitem{KS09}
S.~Khot and R.~Saket.
\newblock S{D}{P} integrality gaps with local $\ell_1$-embeddability.
\newblock In {\em Proceedings of 50th Symposium on Foundations of Computer
  Science (FOCS 2009)}, pages 565--574, 2009.

\bibitem{Laa02}
T.~J. Laakso.
\newblock Plane with {$A_\infty$}-weighted metric not bi-{L}ipschitz embeddable
  to {${\Bbb R}^N$}.
\newblock {\em Bull. London Math. Soc.}, 34(6):667--676, 2002.

\bibitem{LP01}
U.~Lang and C.~Plaut.
\newblock Bilipschitz embeddings of metric spaces into space forms.
\newblock {\em Geom. Dedicata}, 87(1-3):285--307, 2001.

\bibitem{LT91}
M.~Ledoux and M.~Talagrand.
\newblock {\em Probability in {B}anach spaces}, volume~23 of {\em Ergebnisse
  der Mathematik und ihrer Grenzgebiete (3) [Results in Mathematics and Related
  Areas (3)]}.
\newblock Springer-Verlag, Berlin, 1991.
\newblock Isoperimetry and processes.

\bibitem{LNP09}
J.~R. Lee, A.~Naor, and Y.~Peres.
\newblock Trees and {M}arkov convexity.
\newblock {\em Geom. Funct. Anal.}, 18(5):1609--1659, 2009.

\bibitem{LN12}
S.~Li and A.~Naor.
\newblock Discretization and affine approximation in high dimensions.
\newblock Preprint available at \url{http://arxiv.org/abs/1202.2567}. To appear
  in Israel J. Math., 2012.

\bibitem{Lin63}
J.~Lindenstrauss.
\newblock On the modulus of smoothness and divergent series in {B}anach spaces.
\newblock {\em Michigan Math. J.}, 10:241--252, 1963.

\bibitem{Lin64}
J.~Lindenstrauss.
\newblock On nonlinear projections in {B}anach spaces.
\newblock {\em Michigan Math. J.}, 11:263--287, 1964.

\bibitem{Lin70}
J.~Lindenstrauss.
\newblock Some aspects of the theory of {B}anach spaces.
\newblock {\em Advances in Math.}, 5:159--180 (1970), 1970.

\bibitem{Lin98}
J.~Lindenstrauss.
\newblock Uniform embeddings, homeomorphisms and quotient maps between {B}anach
  spaces (a short survey).
\newblock {\em Topology Appl.}, 85(1-3):265--279, 1998.
\newblock 8th Prague Topological Symposium on General Topology and Its
  Relations to Modern Analysis and Algebra (1996).

\bibitem{Lin02}
N.~Linial.
\newblock Finite metric-spaces---combinatorics, geometry and algorithms.
\newblock In {\em Proceedings of the {I}nternational {C}ongress of
  {M}athematicians, {V}ol. {III} ({B}eijing, 2002)}, pages 573--586, Beijing,
  2002. Higher Ed. Press.

\bibitem{LLR95}
N.~Linial, E.~London, and Y.~Rabinovich.
\newblock The geometry of graphs and some of its algorithmic applications.
\newblock {\em Combinatorica}, 15(2):215--245, 1995.

\bibitem{LMN02}
N.~Linial, A.~Magen, and A.~Naor.
\newblock Girth and {E}uclidean distortion.
\newblock {\em Geom. Funct. Anal.}, 12(2):380--394, 2002.

\bibitem{LS03}
N.~Linial and M.~Saks.
\newblock The {E}uclidean distortion of complete binary trees.
\newblock {\em Discrete Comput. Geom.}, 29(1):19--21, 2003.

\bibitem{MM10}
K.~Makarychev and Y.~Makarychev.
\newblock Metric extension operators, vertex sparsifiers and {L}ipschitz
  extendability.
\newblock In {\em 51th Annual IEEE Symposium on Foundations of Computer
  Science}, pages 255--264, 2010.

\bibitem{MP84}
M.~B. Marcus and G.~Pisier.
\newblock Characterizations of almost surely continuous {$p$}-stable random
  {F}ourier series and strongly stationary processes.
\newblock {\em Acta Math.}, 152(3-4):245--301, 1984.

\bibitem{Mat92}
J.~Matou{\v{s}}ek.
\newblock Ramsey-like properties for bi-{L}ipschitz mappings of finite metric
  spaces.
\newblock {\em Comment. Math. Univ. Carolin.}, 33(3):451--463, 1992.

\bibitem{Mat99}
J.~Matou{\v{s}}ek.
\newblock On embedding trees into uniformly convex {B}anach spaces.
\newblock {\em Israel J. Math.}, 114:221--237, 1999.

\bibitem{Mat02}
J.~Matou{\v{s}}ek.
\newblock {\em Lectures on discrete geometry}, volume 212 of {\em Graduate
  Texts in Mathematics}.
\newblock Springer-Verlag, New York, 2002.

\bibitem{Mattila}
P.~Mattila.
\newblock {\em Geometry of sets and measures in {E}uclidean spaces}, volume~44
  of {\em Cambridge Studies in Advanced Mathematics}.
\newblock Cambridge University Press, Cambridge, 1995.
\newblock Fractals and rectifiability.

\bibitem{Mau74}
B.~Maurey.
\newblock {\em Th\'eor\`emes de factorisation pour les op\'erateurs lin\'eaires
  \`a valeurs dans les espaces {$L^{p}$}}.
\newblock Soci\'et\'e Math\'ematique de France, Paris, 1974.
\newblock With an English summary, Ast{\'e}risque, No. 11.

\bibitem{Mau03}
B.~Maurey.
\newblock Type, cotype and {$K$}-convexity.
\newblock In {\em Handbook of the geometry of {B}anach spaces, {V}ol.\ 2},
  pages 1299--1332. North-Holland, Amsterdam, 2003.

\bibitem{MP76}
B.~Maurey and G.~Pisier.
\newblock S\'eries de variables al\'eatoires vectorielles ind\'ependantes et
  propri\'et\'es g\'eom\'etriques des espaces de {B}anach.
\newblock {\em Studia Math.}, 58(1):45--90, 1976.

\bibitem{MU32}
S.~Mazur and S.~Ulam.
\newblock Sur les transformations isom\'etriques d'espaces vectoriels norm\'es.
\newblock {\em C. R. Math. Acad. Sci. Paris}, 194:946--948, 1932.

\bibitem{Men09}
M.~Mendel.
\newblock Metric dichotomies.
\newblock In {\em Limits of graphs in group theory and computer science}, pages
  59--76. EPFL Press, Lausanne, 2009.

\bibitem{MN04}
M.~Mendel and A.~Naor.
\newblock Euclidean quotients of finite metric spaces.
\newblock {\em Adv. Math.}, 189(2):451--494, 2004.

\bibitem{MN06-ball}
M.~Mendel and A.~Naor.
\newblock Some applications of {B}all's extension theorem.
\newblock {\em Proc. Amer. Math. Soc.}, 134(9):2577--2584 (electronic), 2006.

\bibitem{MN07}
M.~Mendel and A.~Naor.
\newblock Ramsey partitions and proximity data structures.
\newblock {\em J. Eur. Math. Soc. (JEMS)}, 9(2):253--275, 2007.

\bibitem{MN07-scaled}
M.~Mendel and A.~Naor.
\newblock Scaled {E}nflo type is equivalent to {R}ademacher type.
\newblock {\em Bull. Lond. Math. Soc.}, 39(3):493--498, 2007.

\bibitem{MN08-markov}
M.~Mendel and A.~Naor.
\newblock Markov convexity and local rigidity of distorted metrics [extended
  abstract].
\newblock In {\em Computational geometry ({SCG}'08)}, pages 49--58. ACM, New
  York, 2008.
\newblock Full version to appear in J. Eur. Math. Soc. (JEMS).

\bibitem{MN08}
M.~Mendel and A.~Naor.
\newblock Metric cotype.
\newblock {\em Ann. of Math. (2)}, 168(1):247--298, 2008.

\bibitem{MN10-max}
M.~Mendel and A.~Naor.
\newblock Maximum gradient embeddings and monotone clustering.
\newblock {\em Combinatorica}, 30(5):581--615, 2010.

\bibitem{MN10-calculus}
M.~Mendel and A.~Naor.
\newblock Towards a calculus for non-linear spectral gaps.
\newblock In {\em Proceedings of the Twenty-First Annual ACM-SIAM Symposium on
  Discrete Algorithms}, pages 236--255, 2010.

\bibitem{MN11-dich}
M.~Mendel and A.~Naor.
\newblock A note on dichotomies for metric transforms.
\newblock Available at \url{http://arxiv.org/abs/1102.1800}, 2011.

\bibitem{MN11-skel}
M.~Mendel and A.~Naor.
\newblock Ultrametric skeletons.
\newblock Preprint available at~\url{http://arxiv.org/abs/1112.3416}. To appear
  in {\em Proc. Natl. Acad. Sci. USA}, 2011.

\bibitem{MN12-ext}
M.~Mendel and A.~Naor.
\newblock Spectral calculus and {L}ipschitz extension for barycentric metric
  spaces.
\newblock Preprint, 2012.

\bibitem{MN11-ultra}
M.~Mendel and A.~Naor.
\newblock Ultrametric subsets with large hausdorff dimension.
\newblock {\em Invent. Math.} DOI 10.1007/s00222-012-0402-7, 2012.

\bibitem{MS09}
M.~Mendel and C.~Schwob.
\newblock Fast {C}-{K}-{R} partitions of sparse graphs.
\newblock {\em Chic. J. Theoret. Comput. Sci.}, pages Article 2, 15, 2009.

\bibitem{MS99}
V.~Milman and G.~Schechtman.
\newblock An ``isomorphic'' version of {D}voretzky's theorem. {II}.
\newblock In {\em Convex geometric analysis ({B}erkeley, {CA}, 1996)},
  volume~34 of {\em Math. Sci. Res. Inst. Publ.}, pages 159--164. Cambridge
  Univ. Press, Cambridge, 1999.

\bibitem{Mil71}
V.~D. Milman.
\newblock A new proof of {A}. {D}voretzky's theorem on cross-sections of convex
  bodies.
\newblock {\em Funkcional. Anal. i Prilo\v zen.}, 5(4):28--37, 1971.

\bibitem{Mil85}
V.~D. Milman.
\newblock Almost {E}uclidean quotient spaces of subspaces of a
  finite-dimensional normed space.
\newblock {\em Proc. Amer. Math. Soc.}, 94(3):445--449, 1985.

\bibitem{MT77}
J.~Milnor and W.~Thurston.
\newblock Characteristic numbers of {$3$}-manifolds.
\newblock {\em Enseignement Math. (2)}, 23(3-4):249--254, 1977.

\bibitem{Mil99}
P.~B. Miltersen.
\newblock Cell probe complexity - a survey.
\newblock In {\em 19th Conference on the Foundations of Software Technology and
  Theoretical Computer Science}, 1999.

\bibitem{Mos68}
G.~D. Mostow.
\newblock Quasi-conformal mappings in {$n$}-space and the rigidity of
  hyperbolic space forms.
\newblock {\em Inst. Hautes \'Etudes Sci. Publ. Math.}, (34):53--104, 1968.

\bibitem{Nao01}
A.~Naor.
\newblock A phase transition phenomenon between the isometric and isomorphic
  extension problems for {H}\"older functions between {$L_p$} spaces.
\newblock {\em Mathematika}, 48(1-2):253--271 (2003), 2001.

\bibitem{Nao06}
A.~Naor.
\newblock An application of metric cotype to quasisymmetric embeddings.
\newblock Available at \url{http://arxiv.org/abs/math/0607644}, 2006.

\bibitem{Nao10}
A.~Naor.
\newblock {$L_1$} embeddings of the {H}eisenberg group and fast estimation of
  graph isoperimetry.
\newblock In {\em Proceedings of the {I}nternational {C}ongress of
  {M}athematicians. {V}olume {III}}, pages 1549--1575, New Delhi, 2010.
  Hindustan Book Agency.

\bibitem{NP08}
A.~Naor and Y.~Peres.
\newblock Embeddings of discrete groups and the speed of random walks.
\newblock {\em Int. Math. Res. Not. IMRN}, pages Art. ID rnn 076, 34, 2008.

\bibitem{NP11}
A.~Naor and Y.~Peres.
\newblock {$L_p$} compression, traveling salesmen, and stable walks.
\newblock {\em Duke Math. J.}, 157(1):53--108, 2011.

\bibitem{NPSS06}
A.~Naor, Y.~Peres, O.~Schramm, and S.~Sheffield.
\newblock Markov chains in smooth {B}anach spaces and {G}romov-hyperbolic
  metric spaces.
\newblock {\em Duke Math. J.}, 134(1):165--197, 2006.

\bibitem{NS02}
A.~Naor and G.~Schechtman.
\newblock Remarks on non linear type and {P}isier's inequality.
\newblock {\em J. Reine Angew. Math.}, 552:213--236, 2002.

\bibitem{NT10-max}
A.~Naor and T.~Tao.
\newblock Random martingales and localization of maximal inequalities.
\newblock {\em J. Funct. Anal.}, 259(3):731--779, 2010.

\bibitem{NT10}
A.~Naor and T.~Tao.
\newblock Scale-oblivious metric fragmentation and the nonlinear {D}voretzky
  theorem.
\newblock Preprint available at \url{http://arxiv.org/abs/1003.4013}, to appear
  in {\em Israel J. Math.}, 2010.

\bibitem{NTV03}
F.~Nazarov, S.~Treil, and A.~Volberg.
\newblock The {$Tb$}-theorem on non-homogeneous spaces.
\newblock {\em Acta Math.}, 190(2):151--239, 2003.

\bibitem{Oht09}
S.-I. Ohta.
\newblock Markov type of {A}lexandrov spaces of non-negative curvature.
\newblock {\em Mathematika}, 55(1-2):177--189, 2009.

\bibitem{OO12}
A.~Olshanskii and D.~Osin.
\newblock A quasi-isometric embedding theorem for groups.
\newblock Preprint available at \url{http://arxiv.org/abs/1202.6437}, 2011.

\bibitem{Ost11}
M.~I. Ostrovskii.
\newblock On metric characterizations of some classes of {B}anach spaces.
\newblock {\em C. R. Acad. Bulgare Sci.}, 64(6):775--784, 2011.

\bibitem{Ost12-forms}
M.~I. Ostrovskii.
\newblock Different forms of metric characterizations of classes of {B}anach
  spaces.
\newblock Preprint available at \url{http://arxiv.org/abs/1112.0801}, 2012.

\bibitem{Ost12}
M.~I. Ostrovskii.
\newblock Low-distortion embeddings of graphs with large girth.
\newblock {\em J. Funct. Anal.}, 262(8), 2012.

\bibitem{Ost12-test}
M.~I. Ostrovskii.
\newblock Test-space characterizations of some classes of {B}anach spaces.
\newblock Preprint available at \url{http://arxiv.org/abs/1112.3086}, 2012.

\bibitem{Pan89}
P.~Pansu.
\newblock M\'etriques de {C}arnot-{C}arath\'eodory et quasiisom\'etries des
  espaces sym\'etriques de rang un.
\newblock {\em Ann. of Math. (2)}, 129(1):1--60, 1989.

\bibitem{Pis73}
G.~Pisier.
\newblock Sur les espaces de {B}anach qui ne contiennent pas uniform\'ement de
  {$l^{1}_{n}$}.
\newblock {\em C. R. Acad. Sci. Paris S\'er. A-B}, 277:A991--A994, 1973.

\bibitem{Pis75-martingales}
G.~Pisier.
\newblock Martingales with values in uniformly convex spaces.
\newblock {\em Israel J. Math.}, 20(3-4):326--350, 1975.

\bibitem{Pis82}
G.~Pisier.
\newblock Holomorphic semigroups and the geometry of {B}anach spaces.
\newblock {\em Ann. of Math. (2)}, 115(2):375--392, 1982.

\bibitem{Pis86}
G.~Pisier.
\newblock Probabilistic methods in the geometry of {B}anach spaces.
\newblock In {\em Probability and analysis ({V}arenna, 1985)}, volume 1206 of
  {\em Lecture Notes in Math.}, pages 167--241. Springer, Berlin, 1986.

\bibitem{PX87}
G.~Pisier and Q.~H. Xu.
\newblock Random series in the real interpolation spaces between the spaces
  {$v_p$}.
\newblock In {\em Geometrical aspects of functional analysis (1985/86)}, volume
  1267 of {\em Lecture Notes in Math.}, pages 185--209. Springer, Berlin, 1987.

\bibitem{Rab08}
Y.~Rabinovich.
\newblock On average distortion of embedding metrics into the line.
\newblock {\em Discrete Comput. Geom.}, 39(4):720--733, 2008.

\bibitem{RS09}
P.~Raghavendra and D.~Steurer.
\newblock Integrality gaps for strong {S}{D}{P} relaxations of unique games.
\newblock In {\em Proceedings of 50th Symposium on Foundations of Computer
  Science (FOCS 2009)}, pages 575--585, 2009.

\bibitem{Rev03}
D.~Revelle.
\newblock Rate of escape of random walks on wreath products and related groups.
\newblock {\em Ann. Probab.}, 31(4):1917--1934, 2003.

\bibitem{Rib76}
M.~Ribe.
\newblock On uniformly homeomorphic normed spaces.
\newblock {\em Ark. Mat.}, 14(2):237--244, 1976.

\bibitem{Rib78}
M.~Ribe.
\newblock On uniformly homeomorphic normed spaces. {II}.
\newblock {\em Ark. Mat.}, 16(1):1--9, 1978.

\bibitem{Rib84}
M.~Ribe.
\newblock Existence of separable uniformly homeomorphic nonisomorphic {B}anach
  spaces.
\newblock {\em Israel J. Math.}, 48(2-3):139--147, 1984.

\bibitem{Sac63}
H.~Sachs.
\newblock Regular graphs with given girth and restricted circuits.
\newblock {\em J. London Math. Soc.}, 38:423--429, 1963.

\bibitem{Sag94}
H.~Sagan.
\newblock {\em Space-filling curves}.
\newblock Universitext. Springer-Verlag, New York, 1994.

\bibitem{Sch06}
G.~Schechtman.
\newblock Two observations regarding embedding subsets of {E}uclidean spaces in
  normed spaces.
\newblock {\em Adv. Math.}, 200(1):125--135, 2006.

\bibitem{Sem96}
S.~Semmes.
\newblock On the nonexistence of bi-{L}ipschitz parameterizations and geometric
  problems about {$A_\infty$}-weights.
\newblock {\em Rev. Mat. Iberoamericana}, 12(2):337--410, 1996.

\bibitem{Sol92}
P.~Sol{\'e}.
\newblock The second eigenvalue of regular graphs of given girth.
\newblock {\em J. Combin. Theory Ser. B}, 56(2):239--249, 1992.

\bibitem{SVY09}
C.~Sommer, E.~Verbin, and W.~Yu.
\newblock Distance oracles for sparse graphs.
\newblock In {\em 2009 50th {A}nnual {IEEE} {S}ymposium on {F}oundations of
  {C}omputer {S}cience ({FOCS} 2009)}, pages 703--712. IEEE Computer Soc., Los
  Alamitos, CA, 2009.

\bibitem{Tal87}
M.~Talagrand.
\newblock Regularity of {G}aussian processes.
\newblock {\em Acta Math.}, 159(1-2):99--149, 1987.

\bibitem{Tal93}
M.~Talagrand.
\newblock Isoperimetry, logarithmic {S}obolev inequalities on the discrete
  cube, and {M}argulis' graph connectivity theorem.
\newblock {\em Geom. Funct. Anal.}, 3(3):295--314, 1993.

\bibitem{Tal05}
M.~Talagrand.
\newblock {\em The generic chaining}.
\newblock Springer Monographs in Mathematics. Springer-Verlag, Berlin, 2005.
\newblock Upper and lower bounds of stochastic processes.

\bibitem{Tal11}
M.~Talagrand.
\newblock {\em Upper and Lower Bounds for Stochastic Processes}.
\newblock 2011.
\newblock Modern Methods and Classical Problems. Forthcoming book.

\bibitem{TZ05}
M.~Thorup and U.~Zwick.
\newblock Approximate distance oracles.
\newblock {\em J. ACM}, 52(1):1--24 (electronic), 2005.

\bibitem{Tor81}
H.~Toru{\'n}czyk.
\newblock Characterizing {H}ilbert space topology.
\newblock {\em Fund. Math.}, 111(3):247--262, 1981.

\bibitem{Urb09}
M.~Urba{\'n}ski.
\newblock Transfinite {H}ausdorff dimension.
\newblock {\em Topology Appl.}, 156(17):2762--2771, 2009.

\bibitem{Vai99}
J.~V{\"a}is{\"a}l{\"a}.
\newblock The free quasiworld. {F}reely quasiconformal and related maps in
  {B}anach spaces.
\newblock In {\em Quasiconformal geometry and dynamics ({L}ublin, 1996)},
  volume~48 of {\em Banach Center Publ.}, pages 55--118. Polish Acad. Sci.,
  Warsaw, 1999.

\bibitem{VT79}
I.~A. Vestfrid and A.~F. Timan.
\newblock A universality property of {H}ilbert spaces.
\newblock {\em Dokl. Akad. Nauk SSSR}, 246(3):528--530, 1979.

\bibitem{Vil03}
C.~Villani.
\newblock {\em Topics in optimal transportation}, volume~58 of {\em Graduate
  Studies in Mathematics}.
\newblock American Mathematical Society, Providence, RI, 2003.

\bibitem{Wag00}
R.~Wagner.
\newblock Notes on an inequality by {P}isier for functions on the discrete
  cube.
\newblock In {\em Geometric aspects of functional analysis}, volume 1745 of
  {\em Lecture Notes in Math.}, pages 263--268. Springer, Berlin, 2000.

\bibitem{WS11}
D.~P. Williamson and D.~B. Shmoys.
\newblock {\em The design of approximation algorithms}.
\newblock Cambridge University Press, Cambridge, 2011.

\bibitem{Woj91}
P.~Wojtaszczyk.
\newblock {\em Banach spaces for analysts}, volume~25 of {\em Cambridge Studies
  in Advanced Mathematics}.
\newblock Cambridge University Press, Cambridge, 1991.

\bibitem{Wul12}
C.~Wulff-{N}ilsen.
\newblock Approximate distance oracles with improved query time.
\newblock Preprint, available at~\url{http://arxiv.org/abs/1202.2336}, 2011.

\end{thebibliography}
\end{document}

\begin{remark}\label{rem:history}
{\em Despite the fact that it was first formulated by Bourgain, the
Ribe program is called this way because it is inspired by Ribe's
rigidity theorem. I do not know the exact origin of this name.
In~\cite{Bou86} Bourgain explains the program and its motivation
from Ribe's theorem, describes the basic ``dictionary" that relates
Banach space concepts to metric space concepts, presents examples of
natural steps of the program, raises some open questions, and proves
his metric characterization of isomorphic uniform convexity as the
first successful completion of a step in the program. Bourgain also
writes in~\cite{Bou86} that ``A detailed exposition of this program
will appear in J. Lindenstrauss's forthcoming survey paper [5]."
Reference [5] in~\cite{Bou86} is cited as ``J. Lindenstrauss, {\em
Topics in the geometry of metric spaces}, to appear."  Probably
referring
 to the same unpublished survey,
in~\cite{Bou85} Bourgain also discusses the Ribe program and writes
``We refer the reader to the survey of J. Lindenstrauss [4] for a
detailed exposition of this theme", where reference [4]
of~\cite{Bou85} is ``J. Lindenstrauss, {\em Proceedings Missouri
Conf., Missouri -- Columbia (1984)}, to appear." Unfortunately,
Lindenstrauss' paper was never published. It seems reasonable to
believe that following Ribe's theorem the general idea of the
program occurred to several researchers, and it was natural for
Bourgain to be the first to describe it in writing due to his
ground-breaking result. Nevertheless, it is clear from these
references that Lindenstrauss played an important role here.}
\end{remark}